\documentclass[11pt, reqno]{amsart}
\usepackage{amsfonts,latexsym,enumerate}
\usepackage{amsmath}
\usepackage{amscd}
\usepackage{float,amsmath,amssymb,mathrsfs,bm,multirow,graphics}
\usepackage[dvips]{graphicx}
\usepackage[percent]{overpic}
\usepackage{amsaddr}
\usepackage{todonotes}

\newcommand{\nequation}{\setcounter{equation}{0}}
\renewcommand{\theequation}{\mbox{\arabic{section}.\arabic{equation}}}
\newcommand{\R}{{\Bbb R}}

\newcommand{\C}{{\Bbb C}}
\newcommand{\D}{{\Bbb D}}

\newcommand{\proofbegin}{\noindent{\it Proof.\quad}}
\newcommand{\proofend}{\hfill$\Box$\bigskip}

\newcommand{\diag}{\text{\upshape diag\,}}
\newcommand{\re}{\text{\upshape Re\,}}

\def\ds{\displaystyle}

\DeclareMathOperator{\dist}{dist}

\usepackage{tikz}
\usepackage{pict2e}
\usetikzlibrary{arrows,calc,chains, positioning, shapes.geometric,shapes.symbols,decorations.markings,arrows.meta}
\usetikzlibrary{patterns,angles,quotes}
\usetikzlibrary{decorations.markings}
\tikzset{middlearrow/.style={
			decoration={markings,
				mark= at position 0.6 with {\arrow{#1}} ,
			},
			postaction={decorate}
		}
	}
\tikzset{->-/.style={decoration={
				markings,
				mark=at position #1 with {\arrow{latex}}},postaction={decorate}}}
	
\tikzset{-<-/.style={decoration={
				markings,
				mark=at position #1 with {\arrowreversed{latex}}},postaction={decorate}}}
				
				\tikzset{
	master/.style={
		execute at end picture={
			\coordinate (lower right) at (current bounding box.south east);
			\coordinate (upper left) at (current bounding box.north west);
		}
	},
	slave/.style={
		execute at end picture={
			\pgfresetboundingbox
			\path (upper left) rectangle (lower right);
		}
	}
}
\tikzset{cross/.style={cross out, draw, 
         minimum size=2*(#1-\pgflinewidth), 
         inner sep=0pt, outer sep=0pt}}

\def\XXint#1#2#3{{\setbox0=\hbox{$#1{#2#3}{\int}$}
\vcenter{\hbox{$#2#3$}}\kern-.5\wd0}}

\allowdisplaybreaks

\newtheorem{theorem}{Theorem}[section]
\newtheorem{proposition}[theorem]{Proposition}
\newtheorem{lemma}[theorem]{Lemma}
\newtheorem{definition}[theorem]{Definition}
\newtheorem{assumption}[theorem]{Assumption}
\newtheorem{remark}[theorem]{Remark}

\newtheorem{figuretext}{Figure}
\newtheorem{RHproblem}[theorem]{RH problem}

\usepackage[colorlinks=true]{hyperref}
\hypersetup{urlcolor=blue, citecolor=red, linkcolor=blue}

\input epsf
\date{\today}
\title[T\MakeLowercase{he} B\MakeLowercase{oussinesq equation on the half-line}]
{\Large T\MakeLowercase{he} B\MakeLowercase{oussinesq equation on the half-line}}

\author{C\MakeLowercase{hristophe} C\MakeLowercase{harlier}}
\address{\small Research Institute in Mathematics and Physics, UCLouvain, 1348 Louvain-La-Neuve, Belgium.}
\email{christophe.charlier@uclouvain.be}

\begin{document}

\begin{abstract} 
We study the initial-boundary value problem for the Boussinesq equation on the half-line. Assuming that the solution exists, we prove that it can be recovered from its initial-boundary values via the solution of a $3\times 3$ Riemann-Hilbert problem. The contour consists of $18$ arcs on the unit circle, $18$ segments and $18$ half-lines, and the associated jump matrices involve $9$ reflection coefficients.
\end{abstract}

\maketitle




\section{Introduction}\nequation

In 1872, Boussinesq \cite{B1872} introduced the equation 
\begin{align}\label{badboussinesq}
u_{tt} = u_{xx} + (u^2)_{xx} + u_{xxxx}
\end{align}
as a model for dispersive long waves of small amplitude propagating in a rectangular channel with a flat bottom.  Here $u(x,t)$ is a real-valued function and subscripts denote partial derivatives. This equation supports solitons \cite{H1973} and admits a Lax pair \cite{Z1974}. It is also linearly ill-posed and is therefore sometimes referred to as the ``bad" Boussinesq equation. The existence of global solutions has been studied on the line \cite{CLmain} as well as for some initial-boundary value problems \cite{KL1977, LS1985, Y2002}. 

In \cite{CLmain}, the direct and inverse scattering problems for \eqref{badboussinesq} on the line --- with smooth solitonless solutions that decay rapidly as $|x|\to +\infty$ --- were solved via a Riemann-Hilbert (RH) problem. This approach was then extended in \cite{CLscatteringsolitons} to handle solutions with solitons, and used in \cite{CLmain, CLsectorI, CLsectorII, CLsectorIV, CLsectorV} to derive asymptotic formulas for the long-time behavior of some global solutions. This approach was also used in \cite{C blowup} to analyze blow-up solutions.

 The goal of this paper is to initiate the study of the direct and inverse problems for the Boussinesq equation on the half-line, for smooth solutions without solitons and with fast decay as $x\to +\infty$. More precisely, we consider the initial-boundary value problem for \eqref{badboussinesq} in the half-line domain
\begin{align}\label{halflinedomain}
  \{0 \leq x < +\infty, \; 0 \leq t \leq T\},
\end{align}
with the initial data
\begin{align}\label{initial data u}
u(x,0)=u_{0}(x), \qquad u_{t}(x,0)=u_{1}(x), \qquad x \geq 0,
\end{align}
and with the boundary values
\begin{align}\label{boundaryvalues u}
u(0,t) = \tilde{u}_{0}(t), \quad u_{x}(0,t) = \tilde{u}_{1}(t), \quad u_{xx}(0,t) = \tilde{u}_{2}(t), \quad u_{xxx}(0,t) = \tilde{u}_3(t), \qquad t \in [0,T].
\end{align}
In \eqref{halflinedomain}--\eqref{boundaryvalues u}, $T > 0$ is a finite constant, $u_{0}, u_{1}$ are smooth functions on $[0,+\infty)$ with fast decay at $+\infty$, and $\tilde{u}_{0}, \tilde{u}_{1}, \tilde{u}_{2}, \tilde{u}_3$ are smooth functions on $[0,T]$.

Assuming that the solution $u(x,t)$ of this initial-boundary value problem exists, we will show that it can be expressed in terms of the solution $M(x,t,k)$ of a $3 \times 3$ RH problem, which depends on $u$ only through its initial-boundary values. The jump contour $\Gamma$ of this RH problem consists of nine lines and the unit circle,
\begin{align*}
\Gamma = i \R \cup \omega i \R \cup \omega^{2} i \R \cup e^{\frac{\pi i}{12}} \R \cup \omega e^{\frac{\pi i}{12}} \R \cup \omega^{2} e^{\frac{\pi i}{12}} \R \cup e^{\frac{3\pi i}{12}} \R \cup \omega e^{\frac{3\pi i}{12}} \R \cup \omega^{2} e^{\frac{3\pi i}{12}} \R \cup \partial \D,
\end{align*}
where $\omega := e^{\frac{2\pi i}{3}}$ and $\partial \D := \{k \in \C \, | \, |k| = 1\}$ (see Figure \ref{fig: Gamma}), and the associated jump matrix is expressed in terms of nine reflection coefficients $r_1(k), \tilde{r}_1(k), r_{2}(k)$, $\tilde{r}_{2}(k), \hat{r}_{2}(k), \check{r}_{2}(k),R_{1}(k)$, $R_{2}(k), \tilde{R}_{2}(k)$. Our main results can be summarized as follows: 
\begin{enumerate}[$-$]
\item Theorem \ref{thm:r1r2} solves the direct problem: given some initial-boundary values $u_0, u_1, \tilde{u}_{0}, \tilde{u}_{1}$, $\tilde{u}_{2}, \tilde{u}_3$, it establishes several properties of $r_1, \tilde{r}_1, r_{2}, \tilde{r}_{2}, \hat{r}_{2}, \check{r}_{2}, R_{1}, R_{2}, \tilde{R}_{2}$. 

\item Theorem \ref{thm:inverse sca} solves the inverse problem: it shows that the solution $u(x,t)$ of (\ref{badboussinesq}) on the half-line domain \eqref{halflinedomain} with initial-boundary values \eqref{initial data u}-\eqref{boundaryvalues u} can be recovered from the solution $M(x,t,k)$ of the aforementioned $3 \times 3$ RH problem.
\end{enumerate}

\begin{figure}
\begin{center}
\scalebox{0.79}{\begin{tikzpicture}[scale=1.1]
\node at (0,0) {};
\draw[black,line width=0.45 mm,->-=0.55,->-=0.87] (0,0)--(15:4.5);
\draw[black,line width=0.45 mm,->-=0.55,->-=0.87] (0,0)--(30:4.5);
\draw[black,line width=0.45 mm,->-=0.55,->-=0.87] (0,0)--(45:4.5);
\draw[black,line width=0.45 mm,->-=0.55,->-=0.87] (0,0)--(75:4.5);
\draw[black,line width=0.45 mm,->-=0.55,->-=0.87] (0,0)--(90:4.5);
\draw[black,line width=0.45 mm,->-=0.55,->-=0.87] (0,0)--(105:4.5);
\draw[black,line width=0.45 mm,->-=0.55,->-=0.87] (0,0)--(135:4.5);
\draw[black,line width=0.45 mm,->-=0.55,->-=0.87] (0,0)--(150:4.5);
\draw[black,line width=0.45 mm,->-=0.55,->-=0.87] (0,0)--(165:4.5);
\draw[black,line width=0.45 mm,->-=0.55,->-=0.87] (0,0)--(-15:4.5);
\draw[black,line width=0.45 mm,->-=0.55,->-=0.87] (0,0)--(-30:4.5);
\draw[black,line width=0.45 mm,->-=0.55,->-=0.87] (0,0)--(-45:4.5);
\draw[black,line width=0.45 mm,->-=0.55,->-=0.87] (0,0)--(-75:4.5);
\draw[black,line width=0.45 mm,->-=0.55,->-=0.87] (0,0)--(-90:4.5);
\draw[black,line width=0.45 mm,->-=0.55,->-=0.87] (0,0)--(-105:4.5);
\draw[black,line width=0.45 mm,->-=0.55,->-=0.87] (0,0)--(-135:4.5);
\draw[black,line width=0.45 mm,->-=0.55,->-=0.87] (0,0)--(-150:4.5);
\draw[black,line width=0.45 mm,->-=0.55,->-=0.87] (0,0)--(-165:4.5);

\draw[black,line width=0.45 mm,-<-=0,-<-=0.055,->-=0.115,->-=0.175,->-=0.237,-<-=0.26,-<-=0.32,-<-=0.385,->-=0.445,->-=0.503,->-=0.568,-<-=0.595,-<-=0.66,-<-=0.72,->-=0.78,->-=0.845,->-=0.905,-<-=0.93] ([shift=(-180:3cm)]0,0) arc (-180:180:3cm);

\node at (-85:4) {\footnotesize $\Gamma_{10''}$ };
\node at (-70:4) {\footnotesize $\Gamma_{11''}$ };
\node at (-41:4) {\footnotesize $\Gamma_{12''}$ };
\node at (-26:4) {\footnotesize $\Gamma_{13''}$ };
\node at (-11.5:4) {\footnotesize $\Gamma_{14''}$ };
\node at (18.3:4) {\footnotesize $\Gamma_{15''}$ };
\node at (33.7:4) {\footnotesize $\Gamma_{16''}$ };
\node at (49:4) {\footnotesize $\Gamma_{17''}$ };
\node at (79.2:4) {\footnotesize $\Gamma_{18''}$ };
\node at (93.6:4) {\footnotesize $\Gamma_{1''}$ };
\node at (108.5:4) {\footnotesize $\Gamma_{2''}$ };
\node at (138:4) {\footnotesize $\Gamma_{3''}$ };
\node at (153:4) {\footnotesize $\Gamma_{4''}$ };
\node at (168:4) {\footnotesize $\Gamma_{5''}$ };
\node at (-161.5:4) {\footnotesize $\Gamma_{6''}$ };
\node at (-146:4) {\footnotesize $\Gamma_{7''}$ };
\node at (-130.5:4) {\footnotesize $\Gamma_{8''}$ };
\node at (-100.5:4) {\footnotesize $\Gamma_{9''}$ };

\node at (95:2.5) {\footnotesize $\Gamma_{1'}$ };
\node at (109.5:2.5) {\footnotesize $\Gamma_{2'}$ };
\node at (139:2.5) {\footnotesize $\Gamma_{3'}$ };
\node at (154:2.5) {\footnotesize $\Gamma_{4'}$ };
\node at (168.5:2.5) {\footnotesize $\Gamma_{5'}$ };
\node at (-160:2.5) {\footnotesize $\Gamma_{6'}$ };
\node at (-144.5:2.5) {\footnotesize $\Gamma_{7'}$ };
\node at (-128.5:2.5) {\footnotesize $\Gamma_{8'}$ };
\node at (-99:2.5) {\footnotesize $\Gamma_{9'}$ };
\node at (-83:2.5) {\footnotesize $\Gamma_{10'}$ };
\node at (-68:2.5) {\footnotesize $\Gamma_{11'}$ };
\node at (-39:2.5) {\footnotesize $\Gamma_{12'}$ };
\node at (-24:2.5) {\footnotesize $\Gamma_{13'}$ };
\node at (-9.5:2.5) {\footnotesize $\Gamma_{14'}$ };
\node at (20:2.5) {\footnotesize $\Gamma_{15'}$ };
\node at (35.5:2.5) {\footnotesize $\Gamma_{16'}$ };
\node at (51:2.5) {\footnotesize $\Gamma_{17'}$ };
\node at (81:2.5) {\footnotesize $\Gamma_{18'}$ };


\node at (-22.5:3.3) {\footnotesize $\Gamma_{13}$};
\node at (0:3.3) {\footnotesize $\Gamma_{14}$};
\node at (22.5:3.3) {\footnotesize $\Gamma_{15}$};
\node at (37.5:3.3) {\footnotesize $\Gamma_{16}$};
\node at (60:3.27) {\footnotesize $\Gamma_{17}$};
\node at (82.5:3.25) {\footnotesize $\Gamma_{18}$};
\node at (96.5:3.2) {\footnotesize $\Gamma_{1}$};
\node at (120:3.25) {\footnotesize $\Gamma_{2}$};
\node at (142:3.27) {\footnotesize $\Gamma_{3}$};
\node at (157:3.25) {\footnotesize $\Gamma_{4}$};
\node at (180:3.25) {\footnotesize $\Gamma_{5}$};
\node at (202.5:3.25) {\footnotesize $\Gamma_{6}$};
\node at (217.5:3.25) {\footnotesize $\Gamma_{7}$};
\node at (240:3.23) {\footnotesize $\Gamma_{8}$};
\node at (262.5:3.25) {\footnotesize $\Gamma_{9}$};
\node at (277.5:3.25) {\footnotesize $\Gamma_{10}$};
\node at (300:3.3) {\footnotesize $\Gamma_{11}$};
\node at (322.5:3.35) {\footnotesize $\Gamma_{12}$};

\end{tikzpicture} \begin{tikzpicture}[scale=1.1]
\node at (0,0) {};
\draw[black,line width=0.45 mm] (0,0)--(15:4.5);
\draw[black,line width=0.45 mm] (0,0)--(30:4.5);
\draw[black,line width=0.45 mm] (0,0)--(45:4.5);
\draw[black,line width=0.45 mm] (0,0)--(75:4.5);
\draw[black,line width=0.45 mm] (0,0)--(90:4.5);
\draw[black,line width=0.45 mm] (0,0)--(105:4.5);
\draw[black,line width=0.45 mm] (0,0)--(135:4.5);
\draw[black,line width=0.45 mm] (0,0)--(150:4.5);
\draw[black,line width=0.45 mm] (0,0)--(165:4.5);
\draw[black,line width=0.45 mm] (0,0)--(-15:4.5);
\draw[black,line width=0.45 mm] (0,0)--(-30:4.5);
\draw[black,line width=0.45 mm] (0,0)--(-45:4.5);
\draw[black,line width=0.45 mm] (0,0)--(-75:4.5);
\draw[black,line width=0.45 mm] (0,0)--(-90:4.5);
\draw[black,line width=0.45 mm] (0,0)--(-105:4.5);
\draw[black,line width=0.45 mm] (0,0)--(-135:4.5);
\draw[black,line width=0.45 mm] (0,0)--(-150:4.5);
\draw[black,line width=0.45 mm] (0,0)--(-165:4.5);

\draw[black,line width=0.45 mm] ([shift=(-180:3cm)]0,0) arc (-180:180:3cm);

\node at (-82.5:4) {\footnotesize $D_{10}$ };
\node at (-60:4) {\footnotesize $D_{11}$ };
\node at (-37.5:4) {\footnotesize $D_{12}$ };
\node at (-22.5:4) {\footnotesize $D_{13}$ };
\node at (0:4) {\footnotesize $D_{14}$ };
\node at (22.5:4) {\footnotesize $D_{15}$ };
\node at (37.5:4) {\footnotesize $D_{16}$ };
\node at (60:4) {\footnotesize $D_{17}$ };
\node at (82.5:4) {\footnotesize $D_{18}$ };
\node at (97.5:4) {\footnotesize $D_{1}$ };
\node at (120:4) {\footnotesize $D_{2}$ };
\node at (142.5:4) {\footnotesize $D_{3}$ };
\node at (157.5:4) {\footnotesize $D_{4}$ };
\node at (180:4) {\footnotesize $D_{5}$ };
\node at (-157.5:4) {\footnotesize $D_{6}$ };
\node at (-142.5:4) {\footnotesize $D_{7}$ };
\node at (-120:4) {\footnotesize $D_{8}$ };
\node at (-97.5:4) {\footnotesize $D_{9}$ };

\node at (97.5:2.5) {\footnotesize $E_{1}$ };
\node at (120:2.5) {\footnotesize $E_{2}$ };
\node at (142.5:2.5) {\footnotesize $E_{3}$ };
\node at (157.5:2.5) {\footnotesize $E_{4}$ };
\node at (180:2.5) {\footnotesize $E_{5}$ };
\node at (-157.5:2.5) {\footnotesize $E_{6}$ };
\node at (-142.5:2.5) {\footnotesize $E_{7}$ };
\node at (-120:2.5) {\footnotesize $E_{8}$ };
\node at (-97.5:2.5) {\footnotesize $E_{9}$ };
\node at (-82.5:2.5) {\footnotesize $E_{10}$ };
\node at (-60:2.5) {\footnotesize $E_{11}$ };
\node at (-37.5:2.5) {\footnotesize $E_{12}$ };
\node at (-22.5:2.5) {\footnotesize $E_{13}$ };
\node at (0:2.5) {\footnotesize $E_{14}$ };
\node at (22.5:2.5) {\footnotesize $E_{15}$ };
\node at (37.5:2.5) {\footnotesize $E_{16}$ };
\node at (60:2.5) {\footnotesize $E_{17}$ };
\node at (82.5:2.5) {\footnotesize $E_{18}$ };


\draw[fill] (0:3) circle (0.08);
\draw[fill] (60:3) circle (0.08);
\draw[fill] (120:3) circle (0.08);
\draw[fill] (180:3) circle (0.08);
\draw[fill] (240:3) circle (0.08);
\draw[fill] (300:3) circle (0.08);

\node at (0:3.3) {\footnotesize $\kappa_1$};
\node at (60:3.25) {\footnotesize $\kappa_2$};
\node at (120:3.27) {\footnotesize $\kappa_3$};
\node at (180:3.3) {\footnotesize $\kappa_4$};
\node at (240:3.25) {\footnotesize $\kappa_5$};
\node at (300:3.3) {\footnotesize $\kappa_6$};

\draw[dashed] (-4.5,-4.5)--(-4.5,4.5);

\end{tikzpicture}}
\end{center}
\begin{figuretext}\label{fig: Gamma}
The contour $\Gamma = \cup_{j=1}^{18} \overline{\Gamma_j\cup \Gamma_{j'} \cup \Gamma_{j''}}$ (left), and the open sets $D_{n},E_{n}$, $n=1,\ldots,18$ together with the sixth roots of unity $\kappa_{j}$, $j=1,\ldots,6$ (right).
\end{figuretext}
\end{figure}

In 1997, Fokas developed a method for analyzing initial-boundary value problems for integrable partial differential equations with $2\times 2$ Lax pairs \cite{Fokas1997, Fokas2002}. Since then, this method has been generalized and applied to various integrable equations on the half-line, such as the nonlinear Schr\"{o}dinger equation \cite{BFS2003, FHM2017, LL2023}, the mKdV equation \cite{BFS2004, LNonlinearFourier}, the ``good" Boussinesq equation \cite{HM2015}, the Camassa-Holm equation \cite{BS2008}, and the Degasperis--Procesi equation \cite{L degasperis}. This work can be viewed as a contribution to this ongoing effort. The approach in this paper is primarily inspired by \cite{CLmain}, but also by an extension of the Fokas method developed by Lenells \cite{L3x3} for solving initial-boundary value problems with $3\times 3$ Lax pairs.

We note that several Riemann--Hilbert formulations have recently been developed for integrable equations on the whole real line with $3\times3$ Lax pairs \cite{CLgoodboussinesq, CLmain, HWZ2026, WZ2026}. It would also be interesting to study the initial-boundary value problems on the half-line for the equations considered in \cite{CLgoodboussinesq, HWZ2026, WZ2026}. Another direction for future research would be to investigate the long-time asymptotics of these half-line problems using the nonlinear steepest descent method, in the spirit of e.g. \cite{HL2018} (although that work deals with a $2\times 2$ Lax pair).

\section{Main results}\nequation\label{mainsec}
Our first result concerns the direct problem of constructing  $r_1(k), \tilde{r}_1(k), r_{2}(k)$, $\tilde{r}_{2}(k), \hat{r}_{2}(k)$, $\check{r}_{2}(k), R_{1}(k), R_{2}(k), \tilde{R}_{2}(k)$ in terms of the initial and boundary values $u_0(x), u_1(x), \tilde{u}_{0}(t)$, $\tilde{u}_{1}(t), \tilde{u}_{2}(t), \tilde{u}_3(t)$.

\subsection{The direct problem}
Let $\R_{+} := [0,+\infty)$ and let $\mathcal{S}(\R_{+})$ denote the Schwartz class of all smooth functions $f$ on $\R_{+}$ such that $f$ and all its derivatives have rapid decay as $x \to + \infty$.  
\begin{definition}\label{def: sol of bad bouss}\upshape
We say that $u(x,t)$ is a {\it Schwartz class solution of \eqref{badboussinesq} with existence time $T\in (0,+\infty)$, initial data $u_0, u_{1} \in \mathcal{S}(\R_{+})$ and boundary values $\tilde{u}_{0},\tilde{u}_{1}, \tilde{u}_{2}, \tilde{u}_{3} \in C^{\infty}([0,T])$} if
\begin{enumerate}[$(i)$] 
  \item $u$ is a smooth real-valued function of $(x,t) \in \R_{+} \times [0,T]
  $.

\item $u$ satisfies \eqref{badboussinesq} for $(x,t) \in \R_{+} \times [0,T]$ and 
\begin{align*}
& u(x,0) = u_0(x), & & u_{t}(x,0) = u_1(x), & &   & &    & & x \in \R_{+}, \\
& u(0,t) = \tilde{u}_{0}(t), & & u_{x}(0,t) = \tilde{u}_{1}(t), & & u_{xx}(0,t) = \tilde{u}_{2}(t), & & u_{xxx}(0,t) = \tilde{u}_{3}(t), & & t \in [0,T].
\end{align*}

  \item $u$ has rapid decay as $x \to +\infty$: for each integer $N \geq 1$,
\begin{align}\label{rapiddecay u}
\sup_{\substack{x \in [0,+\infty) \\ t \in [0, T]}} \sum_{i =0}^N (1+|x|)^N |\partial_x^i u(x,t)| < +\infty.
\end{align}
\end{enumerate} 
\end{definition}

Let $u_0, u_1 \in \mathcal{S}(\R_{+})$ and $\tilde{u}_{0}, \tilde{u}_{1}, \tilde{u}_{2}, \tilde{u}_3 \in C^{\infty}([0,T])$ be some given real-valued functions. In what follows, we suppose that there exists a Schwartz class solution $u(x,t)$ of \eqref{badboussinesq} with existence time $T\in (0,\infty)$, initial data $u_0, u_{1}$ and boundary values $\tilde{u}_{0},\tilde{u}_{1}, \tilde{u}_{2}, \tilde{u}_{3}$. The associated reflection coefficients $r_1, \tilde{r}_1, r_{2}, \tilde{r}_{2}, \hat{r}_{2}, \check{r}_{2}, R_{1}, R_{2}, \tilde{R}_{2}$ are defined as follows. 

For $x \geq 0$ and $t\in [0,T]$, define 
\begin{align}
& v(x,t) = \int_{+\infty}^x u_t(x', t) dx', \qquad v_0(x) = \int_{+\infty}^x u_1(x') dx', \label{vxtdef intro} \\
& \tilde{v}_{0}(t) = \int_{+\infty}^0 u_1(x') dx' + \int_{0}^{t}(\tilde{u}_{1}+2\tilde{u}_{0}\tilde{u}_{1}+\tilde{u}_{3})(t')dt'. \nonumber
\end{align} 
Let $\omega := e^{\frac{2\pi i}{3}}$ and define $\{l_j(k), z_j(k)\}_{j=1}^3$ by
\begin{align}\label{lmexpressions intro}
& l_{j}(k) = i \frac{\omega^{j}k + (\omega^{j}k)^{-1}}{2\sqrt{3}}, \qquad z_{j}(k) = i \frac{(\omega^{j}k)^{2} + (\omega^{j}k)^{-2}}{4\sqrt{3}}, \qquad k \in \C\setminus \{0\}.
\end{align}
For $x \geq 0$, $t\in [0,T]$ and $k\in \C\setminus \{0\}$, let $\mathsf{U}(x,t,k)$ and $\mathsf{V}(x,t,k)$ be given by
\begin{align}\label{mathsfUdef intro}
& \mathsf{U}(x,t,k) = P(k)^{-1} \begin{pmatrix}
0 & 0 & 0 \\
0 & 0 & 0 \\
-\frac{u_{x}(x,t)}{4}-\frac{iv(x,t)}{4\sqrt{3}} & -\frac{u(x,t)}{2} & 0
\end{pmatrix} P(k), \\
& \mathsf{V}(x,t,k) = P(k)^{-1} \begin{pmatrix}
-i \frac{u(x,t)}{\sqrt{3}} & 0 & 0 \\
- i \frac{u_{x}(x,t)}{4\sqrt{3}} - \frac{v(x,t)}{4} & i \frac{u(x,t)}{2\sqrt{3}} & 0 \\
-i \frac{u_{xx}(x,t)}{4\sqrt{3}}-\frac{v_{x}(x,t)}{4} & \frac{iu_{x}(x,t)}{4\sqrt{3}}-\frac{v(x,t)}{4} & i \frac{u(x,t)}{2\sqrt{3}}
\end{pmatrix} P(k),
\end{align} 
where
\begin{align}\label{Pdef intro}
P(k) = \begin{pmatrix}
1 & 1 & 1  \\
l_{1}(k) & l_{2}(k) & l_{3}(k) \\
l_{1}(k)^{2} & l_{2}(k)^{2} & l_{3}(k)^{2}
\end{pmatrix}.
\end{align}

Let $\mu_{1}(0,t,k), \mu_{3}(x,t,k), \mu_{1}^A(0,t,k), \mu_{3}^A(x,t,k)$ be the unique $3 \times 3$-matrix valued solutions of the Volterra integral equations
\begin{subequations}\label{XXAdef intro}
\begin{align}
& \mu_1(0,t,k) = I -  \int_{t}^{T} e^{(t-t')\hat{\mathcal{Z}}(k)} (\mathsf{V}\mu_1)(0,t',k)dt', \label{def mu 1 intro} \\
& \mu_3(x,t,k) = I -  \int_{x}^{+\infty} e^{(x-x')\hat{\mathcal{L}}(k)} (\mathsf{U}\mu_3)(x',t,k)dx',  \label{def mu 3 intro} \\
& \mu^{A}_1(0,t,k) = I + \int_{t}^{T} e^{(t'-t)\hat{\mathcal{Z}}(k)} (\mathsf{V}^{T}\mu_1^{A})(0,t',k)dt', \label{def muA 1 intro} \\
& \mu^{A}_3(x,t,k) = I + \int_{x}^{+\infty} e^{(x'-x)\hat{\mathcal{L}}(k)} (\mathsf{U}^{T}\mu_3^{A})(x',t,k)dx', \label{def muA 3 intro}
\end{align}
\end{subequations}
where $\mathsf{U}^T$ and $\mathsf{V}^T$ denote the transposes of $\mathsf{U}$ and $\mathsf{V}$, respectively, $\mathcal{L} = \diag(l_1 , l_2 , l_3)$, $\mathcal{Z} = \diag(z_1 , z_2 , z_3)$, and given a $3\times 3$ matrix $B$, $e^{\hat{B}}$ denotes the operator which acts on a $3\times 3$ matrix $A$ by $e^{\hat{B}}A = e^{B} A e^{-B}$. Note that the definitions of $\mu_1(0,t,k), \mu_3(x,0,k), \mu^{A}_1(0,t,k)$, $\mu^{A}_3(x,0,k)$ depend on $u(x,t)$ only through its initial-boundary values. 

Define also $s,s^{A},S,S^{A}$ by
\begin{align}
& S(k) = \mu_1(0,0,k), & &  s(k) = \mu_3(0,0,k), & & S^{A}(k) = \mu_{1}^{A}(0,0,k), & & s^{A}(k) = \mu_{3}^{A}(0,0,k). \label{sS SAsA def intro}
\end{align}
The six functions $r_{1},\tilde{r}_{1},r_{2},\tilde{r}_{2},\hat{r}_{2},\check{r}_{2}$ are defined by
\begin{align}\label{r1r2def}
\begin{cases}
r_1(k) = \frac{s_{12}(k)}{s_{11}(k)}, & k \in (\Gamma_{1'}\cup \Gamma_{10''}\cup \partial \D)\setminus \mathcal{Q}, \\[0.1cm]
\tilde{r}_1(k) = \frac{(S^{-1}s)_{12}(k)}{(S^{-1}s)_{11}(k)}, & k \in (\Gamma_{1'}\cup \Gamma_{10''}\cup \partial \D) \setminus \mathcal{Q}, \\[0.1cm]
r_2(k) = \frac{s^A_{12}(k)}{s^A_{11}(k)}, \quad & k \in (\Gamma_{1''}\cup \Gamma_{10'}\cup \partial \D)\setminus \mathcal{Q}, \\[0.15cm]
\tilde{r}_2(k) = \frac{(S^{T}s^{A})_{12}(k)}{(S^{T}s^{A})_{11}(k)}, \quad & k \in (\Gamma_{1''}\cup \Gamma_{10'}\cup \partial \D)\setminus \mathcal{Q}, \\[0.1cm]
\hat{r}_2(k) = \frac{s^A_{12}S^{A}_{33}-s^{A}_{32}S^{A}_{13}}{s^A_{11}S^{A}_{33}-s^{A}_{31}S^{A}_{13}}(k), \quad & k \in (\Gamma_{1''}\cup \Gamma_{10'}\cup \partial \D) \setminus \mathcal{Q}, \\[0.15cm]
\check{r}_2(k) = \frac{s^A_{12}S^{A}_{22}-s^{A}_{22}S^{A}_{12}}{s^A_{11}S^{A}_{22}-s^{A}_{21}S^{A}_{12}}(k), \quad & k \in (\Gamma_{1''}\cup \Gamma_{10'}\cup \partial \D) \setminus \mathcal{Q},
\end{cases}
\end{align}	
and the functions $R_{1}, R_{2}, \tilde{R}_{2}$ are defined by
\begin{align}\label{def of R1 R2}
\begin{cases}
\ds R_{1}(k) = \frac{s^{A}_{23}S^{A}_{31}-s^{A}_{33}S^{A}_{21}}{s_{11}(S^{-1}s)_{11}}, & k \in \big( \bar{D}_{10} \cup \bar{D}_{11} \cup \bar{D}_{12} \cup \bar{E}_{1} \cup \bar{E}_{2} \cup \bar{E}_{3} \cup \partial \D \big) \setminus \hat{\mathcal{Q}}, \\[0.3cm]
\ds R_{2}(k) = \frac{S^{A}_{12}s_{33}}{s^{A}_{11}(s^{A}_{11}S^{A}_{22}-s^{A}_{21}S^{A}_{12})}, & k \in \big( \bar{D}_{1} \cup \bar{D}_{2} \cup \bar{D}_{3} \cup \bar{E}_{10} \cup \bar{E}_{11} \cup \bar{E}_{12} \cup \partial \D \big) \setminus \hat{\mathcal{Q}}, \\[0.4cm]
\ds \tilde{R}_{2}(k) = \frac{S_{21}(S^{-1}s)_{33}}{(S^{T}s^{A})_{11}(s^{A}_{11}S^{A}_{33}-s^{A}_{31}S^{A}_{13})}, & k \in \big( \bar{D}_{1} \cup \bar{D}_{2} \cup \bar{D}_{3} \cup \bar{E}_{10} \cup \bar{E}_{11} \cup \bar{E}_{12} \cup \partial \D \big) \setminus \hat{\mathcal{Q}},
\end{cases}
\end{align}
where $\partial \D$ is the unit circle, the contours $\Gamma_{1'},\Gamma_{1''},\Gamma_{10'},\Gamma_{10''}$ are the open segments and half-lines shown in Figure \ref{fig: Gamma}, i.e.
\begin{align*}
\Gamma_{1'}=(0,i), \quad \Gamma_{1''}=(i,i\infty), \quad \Gamma_{10'}=(0,-i), \quad \Gamma_{10''}=(-i,-i\infty),
\end{align*}
$\mathcal{Q}=\{\kappa_{j}\}_{j=1}^{6}$, $\hat{\mathcal{Q}}=\mathcal{Q}\cup \{0\}$, $\kappa_{j}=e^{\frac{\pi i(j-1)}{3}}$, $j=1,\ldots,6$ and the open sets $D_{n},E_{n}$, $n=1,\ldots,18$ are as indicated in Figure \ref{fig: Gamma} (see also \eqref{def of Dj and Ej}).

\subsubsection{Assumption of no solitons}

Solitons correspond to the presence of zeros for certain spectral functions involving $s,S,s^{A},S^{A}$, so we first discuss the domains of definition of these functions. The functions $S(k), S^{A}(k)$ are well-defined for $\C\setminus \hat{\mathcal{Q}}$, $s(k)$ is well-defined for  $k \in (\omega^{2} \hat{\mathcal{S}}_{3}, \omega \hat{\mathcal{S}}_{3}, \hat{\mathcal{S}}_{3})\setminus \hat{\mathcal{Q}}$,\footnote{The notation $k \in (\omega^{2} \hat{\mathcal{S}}_{3}, \omega \hat{\mathcal{S}}_{3}, \hat{\mathcal{S}}_{3})\setminus \hat{\mathcal{Q}}$ indicates that the first, second, and third columns of $s(k)$ are well-defined for $k$ in the sets $\omega^2\hat{\mathcal{S}}_{3}\setminus \hat{\mathcal{Q}}$, $\omega \hat{\mathcal{S}}_{3}\setminus \hat{\mathcal{Q}}$, and $\hat{\mathcal{S}}_{3}\setminus \hat{\mathcal{Q}}$, respectively.} and $s^{A}(k)$ is well-defined for  $k \in (\omega^{2} \hat{\mathcal{S}}_{3}^{*}, \omega \hat{\mathcal{S}}_{3}^{*}, \hat{\mathcal{S}}_{3}^{*})\setminus \hat{\mathcal{Q}}$, where 
\begin{align*}
\mathcal{S}_{3} & = \big\{k \in \C: \big(\arg k \in (\tfrac{\pi}{6},\tfrac{5\pi}{6}) \mbox{ and } |k| > 1\big) \mbox{ or } \big(\arg k \in (-\tfrac{5\pi}{6},-\tfrac{\pi}{6}) \mbox{ and } |k| < 1\big) \big\} \\
& = \mbox{int}(\cup_{j=16}^{18} \bar{D}_{j} \cup \cup_{j=1}^{3} \bar{D}_{j} \cup \cup_{j=7}^{12} \bar{E}_{j}),
\end{align*}
$\hat{\mathcal{S}}_{3} = \partial \D \cup \bar{\mathcal{S}}_{3}$, and $\hat{\mathcal{S}}_{3}^{*} := \{z : \overline{z}^{-1}\in \hat{\mathcal{S}}_{3}\}$, see also Figure \ref{fig: s and sA}. (See also Propositions \ref{sprop}, \ref{sAprop} for further properties of $s,S,s^{A},S^{A}$.) For simplicity, we assume that no solitons are present. 

\begin{figure}
\begin{center}
\scalebox{0.5}{

\begin{tikzpicture}[scale=1.1,master]
\node at (0,0) {};

\draw[line width=0.025 mm, dashed,fill=gray!30] (0,0) --  (30:4.5) arc(30:150:4.5) -- cycle;
\draw[line width=0.25 mm, dashed,fill=white] (0,0) --  (30:3) arc(30:150:3) -- cycle;
\node at (90:3.8) {$\mathcal{S}_{3}$ };

\draw[line width=0.025 mm, dashed,fill=gray!30] (0,0) --  (150:4.5) arc(150:270:4.5) -- cycle;
\draw[line width=0.25 mm, dashed,fill=white] (0,0) --  (150:3) arc(150:270:3) -- cycle;
\node at (210:3.8) {$\omega\mathcal{S}_{3}$ };

\draw[line width=0.025 mm, dashed,fill=gray!30] (0,0) --  (270:4.5) arc(270:390:4.5) -- cycle;
\draw[line width=0.25 mm, dashed,fill=white] (0,0) --  (270:3) arc(270:390:3) -- cycle;
\node at (330:3.8) {$\omega^{2}\mathcal{S}_{3}$ };

\draw[line width=0.25 mm, dashed,fill=gray!30] (0,0) --  (-150:3) arc(-150:-30:3) -- cycle;
\node at (-90:2) {$\mathcal{S}_{3}$ };

\draw[line width=0.25 mm, dashed,fill=gray!30] (0,0) --  (-30:3) arc(-30:90:3) -- cycle;
\node at (30:2) {$\omega\mathcal{S}_{3}$ };

\draw[line width=0.25 mm, dashed,fill=gray!30] (0,0) --  (90:3) arc(90:210:3) -- cycle;
\node at (150:2) {$\omega^{2}\mathcal{S}_{3}$ };

\draw[black,line width=0.45 mm] (30:3)--(30:4.5);
\draw[black,line width=0.45 mm] (0,0)--(90:3);
\draw[black,line width=0.45 mm] (150:3)--(150:4.5);
\draw[black,line width=0.45 mm] (0,0)--(-30:3);
\draw[black,line width=0.45 mm] (-90:3)--(-90:4.5);
\draw[black,line width=0.45 mm] (0,0)--(-150:3);

\draw[black,line width=0.45 mm] ([shift=(-180:3cm)]0,0) arc (-180:180:3cm);

\draw[fill] (0:3) circle (0.08);
\draw[fill] (60:3) circle (0.08);
\draw[fill] (120:3) circle (0.08);
\draw[fill] (180:3) circle (0.08);
\draw[fill] (240:3) circle (0.08);
\draw[fill] (300:3) circle (0.08);
\end{tikzpicture} 
\hspace{0.5cm}
\begin{tikzpicture}[scale=1.1,slave]
\node at (0,0) {};

\draw[line width=0.25 mm, dashed,fill=gray!30] (0,0) --  (-150:4.5) arc(-150:-30:4.5) -- cycle;
\draw[line width=0.25 mm, dashed,fill=white] (0,0) --  (-150:3) arc(-150:-30:3) -- cycle;
\node at (-90:3.8) {$\mathcal{S}_{3}^{*}$ };

\draw[line width=0.25 mm, dashed,fill=gray!30] (0,0) --  (-30:4.5) arc(-30:90:4.5) -- cycle;
\draw[line width=0.25 mm, dashed,fill=white] (0,0) --  (-30:3) arc(-30:90:3) -- cycle;
\node at (30:3.8) {$\omega\mathcal{S}_{3}^{*}$ };

\draw[line width=0.25 mm, dashed,fill=gray!30] (0,0) --  (90:4.5) arc(90:210:4.5) -- cycle;
\draw[line width=0.25 mm, dashed,fill=white] (0,0) --  (90:3) arc(90:210:3) -- cycle;
\node at (150:3.8) {$\omega^{2}\mathcal{S}_{3}^{*}$ };

\draw[line width=0.25 mm, dashed,fill=gray!30] (0,0) --  (30:3) arc(30:150:3) -- cycle;
\node at (90:2) {$\mathcal{S}_{3}^{*}$ };

\draw[line width=0.25 mm, dashed,fill=gray!30] (0,0) --  (150:3) arc(150:270:3) -- cycle;
\node at (210:2) {$\omega\mathcal{S}_{3}^{*}$ };

\draw[line width=0.25 mm, dashed,fill=gray!30] (0,0) --  (270:3) arc(270:390:3) -- cycle;
\node at (330:2) {$\omega^{2}\mathcal{S}_{3}^{*}$ };

\draw[black,line width=0.45 mm] (0,0)--(30:3);
\draw[black,line width=0.45 mm] (90:3)--(90:4.5);
\draw[black,line width=0.45 mm] (0,0)--(150:3);
\draw[black,line width=0.45 mm] (-30:3)--(-30:4.5);
\draw[black,line width=0.45 mm] (0,0)--(-90:3);
\draw[black,line width=0.45 mm] (-150:3)--(-150:4.5);

\draw[black,line width=0.45 mm] ([shift=(-180:3cm)]0,0) arc (-180:180:3cm);

\draw[fill] (0:3) circle (0.08);
\draw[fill] (60:3) circle (0.08);
\draw[fill] (120:3) circle (0.08);
\draw[fill] (180:3) circle (0.08);
\draw[fill] (240:3) circle (0.08);
\draw[fill] (300:3) circle (0.08);
\end{tikzpicture}}
\end{center}
\begin{figuretext}\label{fig: s and sA}
The sets $\mathcal{S}_{3},\omega\mathcal{S}_{3},\omega^{2}\mathcal{S}_{3}$ (left), and $\mathcal{S}_{3}^{*},\omega\mathcal{S}_{3}^{*},\omega^{2}\mathcal{S}_{3}^{*}$ (right).
\end{figuretext}
\end{figure}

\begin{assumption}[Absence of solitons]\label{solitonlessassumption}\upshape
Assume that
\begin{itemize}
\item $s_{11},(S^{-1}s)_{11}$ are nonzero on $(\bar{D}_{13} \cup \bar{D}_{14} \cup \bar{D}_{15} \cup \bar{E}_{4} \cup \bar{E}_{5} \cup \bar{E}_{6} \cup \partial \D) \setminus \hat{\mathcal{Q}}$,
\item $s^A_{11}, (S^{T}s^{A})_{11}$ are nonzero on $(\bar{E}_{13} \cup \bar{E}_{14} \cup \bar{E}_{15} \cup \bar{D}_{4} \cup \bar{D}_{5} \cup \bar{D}_{6} \cup \partial \D) \setminus \hat{\mathcal{Q}}$,
\item $s^A_{11}S^{A}_{33}-s^{A}_{31}S^{A}_{13}$ is nonzero on $\omega^{2}\hat{\mathcal{S}}_{3}^{*} \setminus \hat{\mathcal{Q}}$.
\end{itemize}
\end{assumption}
\begin{remark}\label{remark:sym combined with no soliton assumption}
The following symmetries will be established in Propositions \ref{sprop} and \ref{sAprop}: 
\begin{align*}
& s_{11}(k) = s_{11}(\tfrac{1}{\omega^{2}k}), \qquad (S^{-1}s)_{11}(k) = (S^{-1}s)_{11}(\tfrac{1}{\omega^{2}k}), \\
& s_{11}^{A}(k) = s_{11}^{A}(\tfrac{1}{\omega^{2}k}), \qquad (S^{T}s^{A})_{11}(k) = (S^{T}s^{A})_{11}(\tfrac{1}{\omega^{2}k}), \\
& (s^A_{11}S^{A}_{33}-s^{A}_{31}S^{A}_{13})(k) = (s^A_{11}S^{A}_{22}-s^{A}_{21}S^{A}_{12})(\tfrac{1}{\omega^{2}k}). 
\end{align*}
The above relations, combined with Assumption \ref{solitonlessassumption}, imply that
(i) $s_{11}$ and $(S^{-1}s)_{11}$ are nonzero on $\omega^{2}\hat{\mathcal{S}}_{3} \setminus \hat{\mathcal{Q}}$, (ii) $s_{11}^{A}$ and $(S^{T}s^{A})_{11}$ are nonzero on $\omega^{2}\hat{\mathcal{S}}_{3}^{*} \setminus \hat{\mathcal{Q}}$, and (iii) $s^A_{11}S^{A}_{22}-s^{A}_{21}S^{A}_{12}$ is nonzero on $\omega^{2}\hat{\mathcal{S}}_{3}^{*} \setminus \hat{\mathcal{Q}}$.
\end{remark}

\subsubsection{Assumption of generic behavior at $k = \pm 1$}
We further assume that some spectral functions have generic behavior at $k=\pm 1$.

\begin{assumption}[Generic behavior at $k = \pm 1$]\label{originassumption}\upshape
Assume for $k_{\star} =1$ and $k_{\star}=-1$ that
\begin{align*}
& \lim_{k \to k_{\star}} (k-k_{\star}) s_{11}(k) \neq 0, \quad \lim_{k \to k_{\star}} (k-k_{\star}) s_{13}(k) \neq 0, \quad \lim_{k \to k_{\star}} s_{31}(k) \neq 0, \quad \lim_{k \to k_{\star}} s_{33}(k) \neq 0, \\
& \lim_{k \to k_{\star}} (k-k_{\star}) (S^{-1}s)_{11}(k) \neq 0, \quad \lim_{k \to k_{\star}} (k-k_{\star}) (S^{-1}s)_{13}(k) \neq 0, \\
& \lim_{k \to k_{\star}} (S^{-1}s)_{31}(k) \neq 0, \quad \lim_{k \to k_{\star}} (S^{-1}s)_{33}(k) \neq 0, \\
& \lim_{k \to k_{\star}} (k-k_{\star}) s^{A}_{11}(k) \neq 0, \quad \lim_{k \to k_{\star}} (k-k_{\star}) s^{A}_{31}(k) \neq 0, \quad \lim_{k \to k_{\star}} s^{A}_{13}(k) \neq 0, \quad \lim_{k \to k_{\star}} s^{A}_{33}(k) \neq 0, \\
& \lim_{k \to k_{\star}} (k-k_{\star}) (S^{T}s^{A})_{11}(k) \neq 0, \quad \lim_{k \to k_{\star}} (k-k_{\star}) (S^{T}s^{A})_{31}(k) \neq 0, \\ 
& \lim_{k \to k_{\star}} (S^{T}s^{A})_{13}(k) \neq 0, \quad \lim_{k \to k_{\star}} (S^{T}s^{A})_{33}(k) \neq 0, \\
& \lim_{k \to k_{\star}} (k-k_{\star}) (s^{A}_{11}S^{A}_{33}-s^{A}_{31}S^{A}_{13})(k) \neq 0, \quad  \lim_{k \to k_{\star}} (k-k_{\star}) (s^A_{11}S^{A}_{22}-s^{A}_{21}S^{A}_{12})(k) \neq 0, \\
& \lim_{k \to k_{\star}} (k-k_{\star})^{2} (s^{A}_{31}S^{A}_{22}-s^{A}_{21}S^{A}_{32})(k) \neq 0, \quad \lim_{k \to k_{\star}} (k-k_{\star}) (s^{A}_{33}S^{A}_{22}-s^{A}_{23}S^{A}_{32})(k) \neq 0, \\
& \lim_{k \to k_{\star}} (k-k_{\star})^{2} (s^{A}_{31}S^{A}_{11}-s^{A}_{11}S^{A}_{31})(k) \neq 0, \quad \lim_{k \to k_{\star}} (k-k_{\star}) (s^{A}_{33}S^{A}_{11}-s^{A}_{13}S^{A}_{31})(k) \neq 0, \\
& \lim_{k \to k_{\star}} (s^{A}_{23}S^{A}_{11}-s^{A}_{13}S^{A}_{21})(k) \neq 0, \quad \lim_{k \to k_{\star}} (k-k_{\star}) (s^{A}_{22}S^{A}_{11}-s^{A}_{12}S^{A}_{21})(k) \neq 0, \\
& \lim_{k \to k_{\star}} (s^{A}_{23}S^{A}_{33}-s^{A}_{33}S^{A}_{23})(k) \neq 0, \quad \lim_{k \to k_{\star}} (k-k_{\star}) (s^{A}_{22}S^{A}_{33}-s^{A}_{32}S^{A}_{23})(k) \neq 0,
\end{align*}
where the limits are taken within the respective domains of definition.
\end{assumption}

\subsubsection{Statement of the theorem}
We now state our main result on the direct scattering problem for \eqref{badboussinesq} on the half-line.

\begin{theorem}[Direct scattering: properties of $r_{1},\tilde{r}_{1},r_{2},\tilde{r}_{2},\hat{r}_{2},\check{r}_{2},R_{1}, R_{2}, \tilde{R}_{2}$]\label{thm:r1r2}
\, \newline Suppose $u$ is a Schwartz class solution of \eqref{badboussinesq} with existence time $T \in (0, \infty)$, initial data $u_0, u_{1} \in \mathcal{S}(\R_{+})$, and boundary values $\tilde{u}_{0},\tilde{u}_{1}, \tilde{u}_{2}, \tilde{u}_3 \in C^{\infty}([0,T])$ such that Assumptions \ref{solitonlessassumption} and \ref{originassumption} hold.
Then the spectral functions $r_1,\tilde{r}_{1}:(\Gamma_{1'}\cup \Gamma_{10''}\cup \partial \D)\setminus \mathcal{Q} \to \C$, $r_2,\tilde{r}_{2},\hat{r}_{2},\check{r}_{2}: (\Gamma_{1''}\cup \Gamma_{10'}\cup \partial \D)\setminus \mathcal{Q} \to \C$ and $R_{1}, R_{2}, \tilde{R}_{2}$ are well-defined by \eqref{r1r2def}--\eqref{def of R1 R2} and have the following properties:
\begin{enumerate}[$(i)$]
 \item \label{Theorem2.3itemi}
 $r_{1},\tilde{r}_{1},r_{2},\tilde{r}_{2},\hat{r}_{2},\check{r}_{2}$ admit extensions such that $r_1,\tilde{r}_{1} \in C^\infty(\Gamma_{1'}\cup \Gamma_{10''}\cup \partial \D)$\footnote{In particular, $r_1(k)$ has consistent derivatives to all orders as $k \to k_\star$, $k_\star = \pm i$, $k\in \Gamma_{1'}\cup \Gamma_{10''}\cup \partial \D$.}
and $r_{2},\tilde{r}_{2},\hat{r}_{2}$, $\check{r}_{2} \in C^\infty((\Gamma_{1''}\cup \Gamma_{10'}\cup \partial \D) \setminus \{\omega^{2}, -\omega^{2}\})$. The functions $R_{1}, R_{2}, \tilde{R}_{2}$ are smooth on their domains of definition and analytic on the interior of these domains.

\item \label{Theorem2.3itemii}
$r_{1}(\kappa_{j})\neq 0 \neq \tilde{r}_{1}(\kappa_{j})$ for $j=1,\ldots,6$. $r_{2}(k),\tilde{r}_{2}(k),\hat{r}_{2}(k),\check{r}_{2}(k)$ have simple poles at $k=\omega^2$ and $k = -\omega^2$, and simple zeros at $k=\omega$ and $k=-\omega$. Moreover,
\begin{subequations}
\begin{align}
& r_{1}(1) = r_{1}(-1) = \tilde{r}_{1}(1) = \tilde{r}_{1}(-1) = 1, \label{r1at0} \\
& r_{2}(1) = r_{2}(-1) = \tilde{r}_{2}(1) = \tilde{r}_{2}(-1) = \hat{r}_{2}(1) = \hat{r}_{2}(-1) = \check{r}_{2}(1) = \check{r}_{2}(-1) = -1. \label{r2at0}
\end{align}
\end{subequations}

\item \label{Theorem2.3itemiii}
We have
\begin{subequations}
\begin{align}
& r_{1}(k) = \frac{2iu(0,0)}{\sqrt{3}k^{2}} + O(k^{-3}), & & \mbox{as } k \to \infty, \; k \in \Gamma_{10''}, \label{r1 at inf} \\
& \tilde{r}_{1}(k) = -\frac{2iu(0,T)}{\sqrt{3}}k^{2}e^{-\frac{T}{4|k|^{2}}}+O(k^{3} e^{-\frac{T}{4|k|^{2}}}), & & \mbox{as } k \to 0, \; k \in \Gamma_{1'}, \label{rt1 at 0} \\
& r_{2}(k) = -\frac{2iu(0,0)}{\sqrt{3}}k^{2} +  O(k^{3}), & & \mbox{as } k \to 0, \; k \in \Gamma_{10'},  \label{r2 at 0} \\
& \tilde{r}_{2}(k) = \frac{2iu(0,T)}{\sqrt{3}k^{2}}e^{-\frac{T|k|^{2}}{4}} + O(k^{-3} e^{-\frac{T|k|^{2}}{4}}), & & \mbox{as } k \to \infty, \; k \in \Gamma_{1''}, \label{rt2 at inf} \\
& R_{1}(k) = -\frac{2iu(0,0)}{\sqrt{3}}k^{2} + O(k^{3}), & & \mbox{as } k \to 0, \; k \in \Gamma_{3'}, \label{R1 at 0} 
 \\
& R_{2}(k) = -\frac{2iu(0,0)}{\sqrt{3}}k^{2} +  O(k^{3}), & & \mbox{as } k \to 0, \; k \in \Gamma_{12'}, \label{R2 at 0} \\
& \tilde{R}_{2}(k) = -\frac{2iu(0,0)}{\sqrt{3}k^{2}} + O(k^{-3}), & & \mbox{as } k \to \infty, \; k \in \Gamma_{3''}. \label{Rt2 at inf}
\end{align}
\end{subequations}

\item \label{Theorem2.3itemiv}
For all $k \in \partial \D \setminus \{-\omega,\omega\}$, we have
\begin{subequations}\label{relations on the unit circle}
\begin{align}
& r_{1}(\tfrac{1}{\omega k}) + r_{2}(\omega k) + r_{1}(\omega^{2} k) r_{2}(\tfrac{1}{k}) = 0, \label{r1r2 relation on the unit circle} \\
& \tilde{r}_{1}(\tfrac{1}{\omega k}) + \tilde{r}_{2}(\omega k) + \tilde{r}_{1}(\omega^{2} k) \tilde{r}_{2}(\tfrac{1}{k}) = 0, \label{rt1rt2 relation on the unit circle} \\
& r_{2}(k)-\hat{r}_{2}(k) +r_{1}(\omega k)\big( r_{2}(\tfrac{1}{\omega^{2}k}) - \check{r}_{2}(\tfrac{1}{\omega^{2}k}) \big) = 0, \label{zero on the unit circle} \\
& \tilde{r}_{2}(k)-\check{r}_{2}(k) +\tilde{r}_{1}(\omega k)\big( \tilde{r}_{2}(\tfrac{1}{\omega^{2}k}) - \hat{r}_{2}(\tfrac{1}{\omega^{2}k}) \big) = 0, \label{zero tilde on the unit circle} \\
& R_{1}(k) = r_{1}(k)-\tilde{r}_{1}(k), \quad R_{2}(k) = r_{2}(k)-\check{r}_{2}(k), \quad \tilde{R}_{2}(k) = \tilde{r}_{2}(k)-\hat{r}_{2}(k). \label{R1R2 on the unit circle}
\end{align}
In fact, \eqref{r1r2 relation on the unit circle} is also equivalent to any of the following two relations
\begin{align}\label{r1r2 relation on the unit circle new}
r_{2}(k) = \frac{r_{1}( \omega k)r_{1}(\omega^{2} k)-r_{1}(\frac{1}{k})}{1-r_{1}(\omega k)r_{1}(\frac{1}{\omega k})}, \qquad r_{1}(k) = \frac{r_{2}(\omega k)r_{2}(\omega^{2}k)-r_{2}(\frac{1}{k})}{1-r_{2}(\omega k)r_{2}(\frac{1}{\omega k})},
\end{align}
\eqref{rt1rt2 relation on the unit circle} is equivalent to any of the following two relations
\begin{align}\label{rt1rt2 relation on the unit circle new}
\tilde{r}_{2}(k) = \frac{\tilde{r}_{1}( \omega k)\tilde{r}_{1}(\omega^{2} k)-\tilde{r}_{1}(\frac{1}{k})}{1-\tilde{r}_{1}(\omega k)\tilde{r}_{1}(\frac{1}{\omega k})}, \qquad \tilde{r}_{1}(k) = \frac{\tilde{r}_{2}(\omega k)\tilde{r}_{2}(\omega^{2}k)-\tilde{r}_{2}(\frac{1}{k})}{1-\tilde{r}_{2}(\omega k)\tilde{r}_{2}(\frac{1}{\omega k})},
\end{align}
\end{subequations}
and \eqref{zero on the unit circle}--\eqref{zero tilde on the unit circle} can be rewritten (using \eqref{r1r2 relation on the unit circle}--\eqref{rt1rt2 relation on the unit circle}) as
\begin{align*}
& \hat{r}_{2}(k) = \frac{r_{1}(\omega k) \tilde{r}_{1}(\omega^{2}k) -r_{1}(\frac{1}{k})}{1-r_{1}(\omega k)\tilde{r}_{1}(\frac{1}{\omega k})}, \qquad \check{r}_{2}(k) = \frac{\tilde{r}_{1}(\omega k) r_{1}(\omega^{2}k) -\tilde{r}_{1}(\frac{1}{k})}{1-\tilde{r}_{1}(\omega k)r_{1}(\frac{1}{\omega k})}.
\end{align*}


\end{enumerate} 
\end{theorem}


\subsection{The inverse problem}
Our next result shows that the inverse problem of recovering $u$ from the scattering data $r_{1},\tilde{r}_{1},r_{2},\tilde{r}_{2},\hat{r}_{2},\check{r}_{2}, R_{1}, R_{2}, \tilde{R}_{2}$ can be solved using the solution $M(x,t,k)$ to a $3\times 3$ RH problem.

The jump contour $\Gamma$ of this RH problem is given by
\begin{align*}
\Gamma = \cup_{j=1}^{18} \overline{\Gamma_j\cup \Gamma_{j'} \cup \Gamma_{j''}},
\end{align*}
where $\Gamma_{n},\Gamma_{n'},\Gamma_{n''}$, $n=1,\ldots,18$ are the open sets indicated in Figure \ref{fig: Gamma}.

Let $\Gamma_\star = \cup_{j=1}^{6}\{e^{\frac{5\pi i}{12}}\kappa_{j},i\kappa_j,e^{\frac{7\pi i}{12}}\kappa_{j}\}_{j=1}^6  \cup \{0\}$ denote the set of intersection points of $\Gamma$, so that
\begin{align*}
\Gamma\setminus \Gamma_\star = \cup_{j=1}^{18} \Gamma_j\cup \Gamma_{j'} \cup \Gamma_{j''}.
\end{align*}
The jump matrix $v(x,t,k)$ is defined for $k\in \Gamma\setminus (\Gamma_\star \cup \mathcal{Q})$ by 
\begin{align}\label{vdef}
v(x,t,k) = e^{x\hat{\mathcal{L}}+t\hat{\mathcal{Z}}}\tilde{v}(k),
\end{align}
where $\{l_j(k), z_j(k)\}_{j=1}^3$ are defined by (\ref{lmexpressions intro}) and $\tilde{v}(k)$ is given by
\small \begin{align*}
& \tilde{v}_{1''}= \begin{pmatrix}
1-r_{2}(\frac{1}{k})\tilde{r}_{2}(k) & -r_{2}(\frac{1}{k}) & 0 \\
\tilde{r}_{2}(k) & 1 & 0 \\
\tilde{g}(k) & -h(k) & 1
\end{pmatrix}, \; \tilde{v}_{2''} = \begin{pmatrix}
1 & 0 & 0 \\
0 & 1 & 0 \\
0 & R_{1}(\frac{1}{\omega k}) & 1
\end{pmatrix}, \; \tilde{v}_{3''} = \begin{pmatrix}
1 & 0 & 0 \\
-\tilde{R}_{2}(k) & 1 & 0 \\
0 & 0 & 1
\end{pmatrix}, \\
& \tilde{v}_{4''} = \begin{pmatrix}
1 & 0 & 0 \\
0 & 1 & -r_{1}(\omega k) \\
0 & \tilde{r}_{1}(\frac{1}{\omega k}) & 1-r_{1}(\omega k)\tilde{r}_{1}(\frac{1}{\omega k})
\end{pmatrix}, \; \tilde{v}_{5''} = \begin{pmatrix}
1 & 0 & 0 \\
0 & 1 & 0 \\
R_{2}(\frac{1}{\omega^{2}k}) & 0 & 1
\end{pmatrix}, \; \tilde{v}_{6''} = \begin{pmatrix}
1 & 0 & 0 \\
0 & 1 & R_{1}(\omega k) \\
0 & 0 & 1
\end{pmatrix}, \\
& \tilde{v}_{7''} = \begin{pmatrix}
1 & 0 & \tilde{r}_{2}(\omega^{2}k) \\
-h(\omega^{2}k) & 1 & \tilde{g}(\omega^{2}k) \\
-r_{2}(\frac{1}{\omega^{2}k}) & 0 & 1-r_{2}(\frac{1}{\omega^{2}k})\tilde{r}_{2}(\omega^{2}k)
\end{pmatrix}\hspace{-0.05cm}, \; \tilde{v}_{8''} = \begin{pmatrix}
1 & 0 & 0  \\
R_{1}(\frac{1}{k}) & 1 & 0 \\
0 & 0 & 1
\end{pmatrix}\hspace{-0.05cm}, \; \tilde{v}_{9''} = \begin{pmatrix}
1 & 0 & -\tilde{R}_{2}(\omega^{2}k) \\
0 & 1 & 0 \\
0 & 0 & 1
\end{pmatrix}\hspace{-0.05cm}, \\
& \tilde{v}_{10''} = \begin{pmatrix}
1 & -r_{1}(k) & 0 \\
\tilde{r}_{1}(\frac{1}{k}) & 1-r_{1}(k)\tilde{r}_{1}(\frac{1}{k}) & 0 \\
0 & 0 & 1
\end{pmatrix}, \; \tilde{v}_{11''} = \begin{pmatrix}
1 & 0 & 0 \\
0 & 1 & R_{2}(\frac{1}{\omega k}) \\
0 & 0 & 1
\end{pmatrix}, \; \tilde{v}_{12''} = \begin{pmatrix}
1 & R_{1}(k) & 0 \\
0 & 1 & 0 \\
0 & 0 & 1
\end{pmatrix}, \\
& \tilde{v}_{13''} = \begin{pmatrix}
1 & \tilde{g}(\omega k) & -h(\omega k) \\
0 & 1-r_{2}(\frac{1}{\omega k})\tilde{r}_{2}(\omega k) & -r_{2}(\frac{1}{\omega k}) \\
0 & \tilde{r}_{2}(\omega k) & 1
\end{pmatrix}\hspace{-0.05cm}, \; \tilde{v}_{14''} = \begin{pmatrix}
1 & 0 & R_{1}(\frac{1}{\omega^{2}k}) \\
0 & 1 & 0 \\
0 & 0 & 1
\end{pmatrix}\hspace{-0.05cm}, \; \tilde{v}_{15''} = \begin{pmatrix}
1 & 0 & 0 \\
0 & 1 & 0  \\
0 & -\tilde{R}_{2}(\omega k) & 1
\end{pmatrix}\hspace{-0.05cm}, \\
& \tilde{v}_{16''} = \begin{pmatrix}
1-r_{1}(\omega^{2}k)\tilde{r}_{1}(\frac{1}{\omega^{2}k}) & 0 & \tilde{r}_{1}(\frac{1}{\omega^{2}k}) \\
0 & 1 & 0 \\
-r_{1}(\omega^{2}k) & 0 & 1
\end{pmatrix}\hspace{-0.05cm}, \; \tilde{v}_{17''} = \begin{pmatrix}
1 & R_{2}(\frac{1}{k}) & 0 \\
0 & 1 & 0 \\
0 & 0 & 1
\end{pmatrix}\hspace{-0.05cm}, \; \tilde{v}_{18''} = \begin{pmatrix}
1 & 0 & 0 \\
0 & 1 & 0 \\
R_{1}(\omega^{2}k) & 0 & 1
\end{pmatrix}\hspace{-0.05cm}, \\
& \tilde{v}_{1'} = \begin{pmatrix}
1 & -\tilde{r}_{1}(k) & 0 \\
r_{1}(\frac{1}{k}) & 1-r_{1}(\frac{1}{k})\tilde{r}_{1}(k) & 0 \\
0 & 0 & 1
\end{pmatrix}, \; \tilde{v}_{2'} = \begin{pmatrix}
1 & 0 & 0 \\
0 & 1 & \tilde{R}_{2}(\frac{1}{\omega k}) \\
0 & 0 & 1
\end{pmatrix}, \; \tilde{v}_{3'} = \begin{pmatrix}
1 & -R_{1}(k) & 0 \\
0 & 1 & 0 \\
0 & 0 & 1
\end{pmatrix}, \\
& \tilde{v}_{4'} = \begin{pmatrix}
1 & g(\omega k) & -\tilde{h}(\omega k) \\
0 & 1-r_{2}(\omega k) \tilde{r}_{2}(\frac{1}{\omega k}) & -\tilde{r}_{2}(\frac{1}{\omega k}) \\
0 & r_{2}(\omega k) & 1
\end{pmatrix}, \; \tilde{v}_{5'} = \begin{pmatrix}
1 & 0 & -R_{1}(\frac{1}{\omega^{2}k}) \\
0 & 1 & 0 \\
0 & 0 & 1
\end{pmatrix}, \; \tilde{v}_{6'} = \begin{pmatrix}
1 & 0 & 0 \\
0 & 1 & 0 \\
0 & -R_{2}(\omega k) & 1 
\end{pmatrix}, \\
& \tilde{v}_{7'} = \begin{pmatrix}
1 - r_{1}(\frac{1}{\omega^{2}k})\tilde{r}_{1}(\omega^{2}k) & 0 & r_{1}(\frac{1}{\omega^{2}k}) \\
0 & 1 & 0 \\
-\tilde{r}_{1}(\omega^{2}k) & 0 & 1 
\end{pmatrix}, \; \tilde{v}_{8'} = \begin{pmatrix}
1 & \tilde{R}_{2}(\frac{1}{k}) & 0 \\
0 & 1 & 0 \\
0 & 0 & 1
\end{pmatrix}, \; \tilde{v}_{9'} = \begin{pmatrix}
1 & 0 & 0 \\
0 & 1 & 0 \\
-R_{1}(\omega^{2}k) & 0 & 1
\end{pmatrix}, \\
& \tilde{v}_{10'} = \begin{pmatrix}
1 - r_{2}(k)\tilde{r}_{2}(\frac{1}{k}) & -\tilde{r}_{2}(\frac{1}{k}) & 0 \\
r_{2}(k) & 1 & 0 \\
g(k) & -\tilde{h}(k) & 1
\end{pmatrix}, \; \tilde{v}_{11'} = \begin{pmatrix}
1 & 0 & 0 \\
0 & 1 & 0 \\
0 & - R_{1}(\frac{1}{\omega k}) & 1
\end{pmatrix}, \; \tilde{v}_{12'} = \begin{pmatrix}
1 & 0 & 0 \\
-R_{2}(k) & 1 & 0 \\
0 & 0 & 1
\end{pmatrix}, \\
& \tilde{v}_{13'} = \begin{pmatrix}
1 & 0 & 0 \\
0 & 1 & -\tilde{r}_{1}(\omega k) \\
0 & r_{1}(\frac{1}{\omega k}) & 1-r_{1}(\frac{1}{\omega k})\tilde{r}_{1}(\omega k)
\end{pmatrix}, \; \tilde{v}_{14'} = \begin{pmatrix}
1 & 0 & 0 \\
0 & 1 & 0 \\
\tilde{R}_{2}(\frac{1}{\omega^{2}k}) & 0 & 1
\end{pmatrix}, \; \tilde{v}_{15'} = \begin{pmatrix}
1 & 0 & 0 \\
0 & 1 & -R_{1}(\omega k) \\
0 & 0 & 1
\end{pmatrix}, \\
& \tilde{v}_{16'} = \hspace{-0.05cm}\begin{pmatrix}
1 & 0 & r_{2}(\omega^{2}k) \\
-\tilde{h}(\omega^{2}k) & 1 & g(\omega^{2}k) \\
- \tilde{r}_{2}(\frac{1}{\omega^{2}k}) & 0 & 1-r_{2}(\omega^{2}k)\tilde{r}_{2}(\frac{1}{\omega^{2}k})
\end{pmatrix}\hspace{-0.05cm}, \; \tilde{v}_{17'} = \hspace{-0.05cm}\begin{pmatrix}
1 & 0 & 0 \\
-R_{1}(\frac{1}{k}) & 1 & 0 \\
0 & 0 & 1
\end{pmatrix}\hspace{-0.05cm}, \; \tilde{v}_{18'} = \hspace{-0.05cm}\begin{pmatrix}
1 & 0 & \hspace{-0.1cm}-R_{2}(\omega^{2}k) \\
0 & 1 & \hspace{-0.1cm}0 \\
0 & 0 & \hspace{-0.1cm}1
\end{pmatrix}\hspace{-0.05cm}, \\
& \tilde{v}_{1} = \begin{pmatrix}
1 & -\tilde{r}_{1}(k) & \check{r}_{2}(\omega^{2}k) \\
-\tilde{r}_{2}(k) & 1+\tilde{r}_{1}(k)\tilde{r}_{2}(k) & \hat{r}_{2}(\frac{1}{\omega k})-\check{r}_{2}(\omega^{2}k)\tilde{r}_{2}(k) \\
-\tilde{r}_{2}(\frac{1}{\omega^{2}k}) - r_{1}(\frac{1}{\omega k})\tilde{r}_{2}(k) & r_{1}(\frac{1}{\omega k}) + \tilde{r}_{1}(k) \big( \tilde{r}_{2}(\frac{1}{\omega^{2}k}) + r_{1}(\frac{1}{\omega k})\tilde{r}_{2}(k) \big) & f_{4}(\omega^{2}k)
\end{pmatrix}, \\
& \tilde{v}_{2} = \begin{pmatrix}
1 & -\tilde{r}_{1}(k) & \tilde{r}_{2}(\omega^{2}k) \\
-\tilde{r}_{2}(k) & 1+\tilde{r}_{1}(k)\tilde{r}_{2}(k) & \tilde{r}_{2}(\frac{1}{\omega k})-\tilde{r}_{2}(\omega^{2}k)\tilde{r}_{2}(k) \\
-\tilde{r}_{2}(\frac{1}{\omega^{2}k}) - \tilde{r}_{1}(\frac{1}{\omega k})\tilde{r}_{2}(k) & \tilde{r}_{1}(\frac{1}{\omega k}) + \tilde{r}_{1}(k) \big( \tilde{r}_{2}(\frac{1}{\omega^{2}k}) + \tilde{r}_{1}(\frac{1}{\omega k})\tilde{r}_{2}(k) \big) & f_{2}(\omega^{2}k)
\end{pmatrix}, \\
& \tilde{v}_{3} = \begin{pmatrix}
1 & -r_{1}(k) & \tilde{r}_{2}(\omega^{2}k) \\
-\hat{r}_{2}(k) & 1+r_{1}(k)\hat{r}_{2}(k) & \tilde{r}_{2}(\frac{1}{\omega k})-\tilde{r}_{2}(\omega^{2}k)\hat{r}_{2}(k) \\
-\check{r}_{2}(\frac{1}{\omega^{2}k}) - \tilde{r}_{1}(\frac{1}{\omega k})\hat{r}_{2}(k) & \tilde{r}_{1}(\frac{1}{\omega k}) + r_{1}(k) \big( \check{r}_{2}(\frac{1}{\omega^{2}k}) + \tilde{r}_{1}(\frac{1}{\omega k})\hat{r}_{2}(k) \big) & f_{2}(\omega^{2}k)
\end{pmatrix}, \\
& \tilde{v}_{4} = \hspace{-0.05cm} \begin{pmatrix}
f_{3}(k) & -r_{2}(\frac{1}{k})\hspace{-0.03cm}-\hspace{-0.03cm}\tilde{r}_{1}(\frac{1}{\omega^{2}k})r_{2}(\omega k) & r_{1}(\omega k)r_{2}(\frac{1}{k})\hspace{-0.03cm}+\hspace{-0.03cm}\tilde{r}_{1}(\frac{1}{\omega^{2}k})\big(1\hspace{-0.03cm}+\hspace{-0.03cm}r_{1}(\omega k)r_{2}(\omega k)\big) \\
\hat{r}_{2}(k) & 1 & -r_{1}(\omega k) \\
\check{r}_{2}(\frac{1}{\omega^{2}k})\hspace{-0.03cm}-\hspace{-0.03cm}r_{2}(\omega k)\hat{r}_{2}(k) & -r_{2}(\omega k) & 1+r_{1}(\omega k)r_{2}(\omega k)
\end{pmatrix}\hspace{-0.05cm}, \\
& \tilde{v}_{5} = \hspace{-0.05cm}\begin{pmatrix}
f_{1}(k) & -r_{2}(\frac{1}{k})\hspace{-0.03cm}-\hspace{-0.03cm}r_{1}(\frac{1}{\omega^{2}k})r_{2}(\omega k) & r_{1}(\omega k)r_{2}(\frac{1}{k})\hspace{-0.03cm}+\hspace{-0.03cm}r_{1}(\frac{1}{\omega^{2}k})\big(1\hspace{-0.03cm}+\hspace{-0.03cm}r_{1}(\omega k)r_{2}(\omega k)\big) \\
r_{2}(k) & 1 & -r_{1}(\omega k) \\
r_{2}(\frac{1}{\omega^{2}k})\hspace{-0.03cm}-\hspace{-0.03cm}r_{2}(\omega k)r_{2}(k) & -r_{2}(\omega k) & 1+r_{1}(\omega k)r_{2}(\omega k)
\end{pmatrix}\hspace{-0.05cm}, \\
& \tilde{v}_{6} = \hspace{-0.05cm}\begin{pmatrix}
f_{1}(k) & -\hat{r}_{2}(\frac{1}{k})\hspace{-0.03cm}-\hspace{-0.03cm}r_{1}(\frac{1}{\omega^{2}k})\check{r}_{2}(\omega k) & \tilde{r}_{1}(\omega k)\hat{r}_{2}(\frac{1}{k})\hspace{-0.03cm}+\hspace{-0.03cm}r_{1}(\frac{1}{\omega^{2}k})\big(1\hspace{-0.03cm}+\hspace{-0.03cm}\tilde{r}_{1}(\omega k)\check{r}_{2}(\omega k)\big) \\
r_{2}(k) & 1 & -\tilde{r}_{1}(\omega k) \\
r_{2}(\frac{1}{\omega^{2}k})\hspace{-0.03cm}-\hspace{-0.03cm}\check{r}_{2}(\omega k)r_{2}(k) & -\check{r}_{2}(\omega k) & 1+\tilde{r}_{1}(\omega k)\check{r}_{2}(\omega k)
\end{pmatrix}\hspace{-0.05cm}, \\
& \tilde{v}_{7} = \hspace{-0.06cm} \begin{pmatrix}
1+\tilde{r}_{1}(\omega^{2}k)\tilde{r}_{2}(\omega^{2}k) & \hspace{-0.16cm} \hat{r}_{2}(\frac{1}{k})\hspace{-0.04cm}-\hspace{-0.04cm}\check{r}_{2}(\omega k)\tilde{r}_{2}(\omega^{2}k) & \hspace{-0.16cm} -\tilde{r}_{2}(\omega^{2}k) \\
\tilde{r}_{1}(\omega^{2}k)\tilde{r}_{2}(\frac{1}{\omega k}) \hspace{-0.04cm}+\hspace{-0.04cm} r_{1}(\frac{1}{k})\big( 1\hspace{-0.04cm}+\hspace{-0.04cm}\tilde{r}_{1}(\omega^{2}k)\tilde{r}_{2}(\omega^{2}k) \big) & \hspace{-0.16cm} f_{4}(\omega k) & \hspace{-0.16cm} -\tilde{r}_{2}( \frac{1}{\omega k})\hspace{-0.04cm}-\hspace{-0.04cm} r_{1}(\frac{1}{k})\tilde{r}_{2}(\omega^{2}k) \\
-\tilde{r}_{1}(\omega^{2}k) & \hspace{-0.16cm} \check{r}_{2}(\omega k) & \hspace{-0.16cm} 1
\end{pmatrix}\hspace{-0.06cm} , \\
& \tilde{v}_{8} = \hspace{-0.06cm} \begin{pmatrix}
1+\tilde{r}_{1}(\omega^{2}k)\tilde{r}_{2}(\omega^{2}k) & \hspace{-0.16cm} \tilde{r}_{2}(\frac{1}{k})\hspace{-0.04cm}-\hspace{-0.04cm}\tilde{r}_{2}(\omega k)\tilde{r}_{2}(\omega^{2}k) & \hspace{-0.16cm} -\tilde{r}_{2}(\omega^{2}k) \\
\tilde{r}_{1}(\omega^{2}k)\tilde{r}_{2}(\frac{1}{\omega k}) \hspace{-0.04cm}+\hspace{-0.04cm} \tilde{r}_{1}(\frac{1}{k})\big( 1\hspace{-0.04cm}+\hspace{-0.04cm}\tilde{r}_{1}(\omega^{2}k)\tilde{r}_{2}(\omega^{2}k) \big) & \hspace{-0.16cm} f_{2}(\omega k) & \hspace{-0.16cm} -\tilde{r}_{2}( \frac{1}{\omega k}) \hspace{-0.04cm}-\hspace{-0.04cm} \tilde{r}_{1}(\frac{1}{k})\tilde{r}_{2}(\omega^{2}k) \\
-\tilde{r}_{1}(\omega^{2}k) & \hspace{-0.16cm} \tilde{r}_{2}(\omega k) & \hspace{-0.16cm} 1
\end{pmatrix}\hspace{-0.06cm} , \\
& \tilde{v}_{9} = \hspace{-0.06cm} \begin{pmatrix}
1+r_{1}(\omega^{2}k)\hat{r}_{2}(\omega^{2}k) & \hspace{-0.16cm} \tilde{r}_{2}(\frac{1}{k})\hspace{-0.04cm}-\hspace{-0.04cm}\tilde{r}_{2}(\omega k)\hat{r}_{2}(\omega^{2}k) & \hspace{-0.16cm} -\hat{r}_{2}(\omega^{2}k) \\
r_{1}(\omega^{2}k)\check{r}_{2}(\frac{1}{\omega k}) \hspace{-0.04cm}+\hspace{-0.04cm} \tilde{r}_{1}(\frac{1}{k})\big( 1\hspace{-0.04cm}+\hspace{-0.04cm}r_{1}(\omega^{2}k)\hat{r}_{2}(\omega^{2}k) \big) & \hspace{-0.16cm} f_{2}(\omega k) & \hspace{-0.1cm} -\check{r}_{2}( \frac{1}{\omega k})\hspace{-0.04cm}-\hspace{-0.04cm}\tilde{r}_{1}(\frac{1}{k})\hat{r}_{2}(\omega^{2}k) \\
-r_{1}(\omega^{2}k) & \hspace{-0.16cm} \tilde{r}_{2}(\omega k) & \hspace{-0.16cm} 1
\end{pmatrix}\hspace{-0.06cm} , \\
& \tilde{v}_{10} = \begin{pmatrix}
1 & -r_{1}(k) & \hat{r}_{2}(\omega^{2}k) \\
-r_{2}(k) & 1+r_{1}(k)r_{2}(k) & \check{r}_{2}(\frac{1}{\omega k})-r_{2}(k)\hat{r}_{2}(\omega^{2}k) \\
-r_{2}(\frac{1}{\omega^{2}k})-\tilde{r}_{1}(\frac{1}{\omega k})r_{2}(k) & r_{1}(k)r_{2}(\frac{1}{\omega^{2}k}) + \tilde{r}_{1}(\frac{1}{\omega k})\big( 1+r_{1}(k)r_{2}(k) \big) & f_{3}(\omega^{2}k)
\end{pmatrix}, \\
& \tilde{v}_{11} = \begin{pmatrix}
1 & -r_{1}(k) & r_{2}(\omega^{2}k) \\
-r_{2}(k) & 1+r_{1}(k)r_{2}(k) & r_{2}(\frac{1}{\omega k})-r_{2}(k)r_{2}(\omega^{2}k) \\
-r_{2}(\frac{1}{\omega^{2}k})-r_{1}(\frac{1}{\omega k})r_{2}(k) & r_{1}(k)r_{2}(\frac{1}{\omega^{2}k}) + r_{1}(\frac{1}{\omega k})\big( 1+r_{1}(k)r_{2}(k) \big) & f_{1}(\omega^{2}k)
\end{pmatrix}, \\
& \tilde{v}_{12} = \begin{pmatrix}
1 & -\tilde{r}_{1}(k) & r_{2}(\omega^{2}k) \\
-\check{r}_{2}(k) & 1+\tilde{r}_{1}(k)\check{r}_{2}(k) & r_{2}(\frac{1}{\omega k})-\check{r}_{2}(k)r_{2}(\omega^{2}k) \\
-\hat{r}_{2}(\frac{1}{\omega^{2}k})-r_{1}(\frac{1}{\omega k})\check{r}_{2}(k) & \tilde{r}_{1}(k)\hat{r}_{2}(\frac{1}{\omega^{2}k}) + r_{1}(\frac{1}{\omega k})\big( 1+\tilde{r}_{1}(k)\check{r}_{2}(k) \big) & f_{1}(\omega^{2}k)
\end{pmatrix}, \\
& \tilde{v}_{13} = \hspace{-0.06cm} \begin{pmatrix}
f_{4}(k) & \hspace{-0.1cm} -\tilde{r}_{2}(\frac{1}{k})\hspace{-0.04cm} - \hspace{-0.04cm} r_{1}(\frac{1}{\omega^{2}k})\tilde{r}_{2}(\omega k) & \tilde{r}_{1}(\omega k)\tilde{r}_{2}(\frac{1}{k})\hspace{-0.04cm}+\hspace{-0.04cm}r_{1}(\frac{1}{\omega^{2}k}) \big( 1\hspace{-0.04cm}+\hspace{-0.04cm}\tilde{r}_{1}(\omega k)\tilde{r}_{2}(\omega k) \big) \\
\check{r}_{2}(k) & \hspace{-0.1cm} 1 & -\tilde{r}_{1}(\omega k) \\
\hat{r}_{2}(\frac{1}{\omega^{2}k})\hspace{-0.04cm}-\hspace{-0.04cm}\check{r}_{2}(k)\tilde{r}_{2}(\omega k) & \hspace{-0.1cm} -\tilde{r}_{2}(\omega k) & 1+\tilde{r}_{1}(\omega k)\tilde{r}_{2}(\omega k)
\end{pmatrix}\hspace{-0.06cm}, \\
& \tilde{v}_{14} = \hspace{-0.06cm} \begin{pmatrix}
f_{2}(k) & \hspace{-0.1cm} -\tilde{r}_{2}(\frac{1}{k})\hspace{-0.04cm}-\hspace{-0.04cm}\tilde{r}_{1}(\frac{1}{\omega^{2}k})\tilde{r}_{2}(\omega k) & \tilde{r}_{1}(\omega k)\tilde{r}_{2}(\frac{1}{k})\hspace{-0.04cm}+\hspace{-0.04cm}\tilde{r}_{1}(\frac{1}{\omega^{2}k}) \big( 1\hspace{-0.04cm}+\hspace{-0.04cm}\tilde{r}_{1}(\omega k)\tilde{r}_{2}(\omega k) \big) \\
\tilde{r}_{2}(k) & \hspace{-0.1cm} 1 & -\tilde{r}_{1}(\omega k) \\
\tilde{r}_{2}(\frac{1}{\omega^{2}k})\hspace{-0.04cm}-\hspace{-0.04cm}\tilde{r}_{2}(k)\tilde{r}_{2}(\omega k) & \hspace{-0.1cm} -\tilde{r}_{2}(\omega k) & 1+\tilde{r}_{1}(\omega k)\tilde{r}_{2}(\omega k)
\end{pmatrix}\hspace{-0.06cm}, \\
& \tilde{v}_{15} = \hspace{-0.06cm} \begin{pmatrix}
f_{2}(k) & \hspace{-0.1cm} -\check{r}_{2}(\frac{1}{k})\hspace{-0.04cm}-\hspace{-0.04cm}\tilde{r}_{1}(\frac{1}{\omega^{2}k})\hat{r}_{2}(\omega k) & r_{1}(\omega k)\check{r}_{2}(\frac{1}{k})\hspace{-0.04cm}+\hspace{-0.04cm}\tilde{r}_{1}(\frac{1}{\omega^{2}k}) \big( 1+r_{1}(\omega k)\hat{r}_{2}(\omega k) \big) \\
\tilde{r}_{2}(k) & \hspace{-0.1cm} 1 & -r_{1}(\omega k) \\
\tilde{r}_{2}(\frac{1}{\omega^{2}k})\hspace{-0.04cm}-\hspace{-0.04cm}\tilde{r}_{2}(k)\hat{r}_{2}(\omega k) & \hspace{-0.1cm} -\hat{r}_{2}(\omega k) & 1+r_{1}(\omega k)\hat{r}_{2}(\omega k)
\end{pmatrix}\hspace{-0.06cm}, \\
& \tilde{v}_{16} = \hspace{-0.06cm} \begin{pmatrix}
1+r_{1}(\omega^{2}k)r_{2}(\omega^{2}k) & \check{r}_{2}(\frac{1}{k})\hspace{-0.04cm}-\hspace{-0.04cm}r_{2}(\omega^{2}k)\hat{r}_{2}(\omega k) & -r_{2}(\omega^{2}k) \\
\tilde{r}_{1}(\frac{1}{k})\hspace{-0.04cm}+\hspace{-0.04cm}r_{1}(\omega^{2}k)\big( r_{2}(\frac{1}{\omega k}) \hspace{-0.04cm}+\hspace{-0.04cm} \tilde{r}_{1}(\frac{1}{k})r_{2}(\omega^{2}k) \big) & f_{3}(\omega k) & -r_{2}(\frac{1}{\omega k}) \hspace{-0.04cm}-\hspace{-0.04cm} \tilde{r}_{1}(\frac{1}{k})r_{2}(\omega^{2}k) \\[0.05cm]
-r_{1}(\omega^{2}k) & \hat{r}_{2}(\omega k) & 1
\end{pmatrix}\hspace{-0.06cm}, \\
& \tilde{v}_{17} = \hspace{-0.06cm} \begin{pmatrix}
1+r_{1}(\omega^{2}k)r_{2}(\omega^{2}k) & r_{2}(\frac{1}{k})\hspace{-0.04cm}-\hspace{-0.04cm}r_{2}(\omega^{2}k)r_{2}(\omega k) & -r_{2}(\omega^{2}k) \\
r_{1}(\frac{1}{k})\hspace{-0.04cm}+\hspace{-0.04cm}r_{1}(\omega^{2}k)\big( r_{2}(\frac{1}{\omega k}) \hspace{-0.04cm}+\hspace{-0.04cm} r_{1}(\frac{1}{k})r_{2}(\omega^{2}k) \big) & f_{1}(\omega k) & -r_{2}(\frac{1}{\omega k}) \hspace{-0.04cm}-\hspace{-0.04cm} r_{1}(\frac{1}{k})r_{2}(\omega^{2}k) \\[0.05cm]
-r_{1}(\omega^{2}k) & r_{2}(\omega k) & 1
\end{pmatrix}\hspace{-0.06cm}, \\
& \tilde{v}_{18} = \hspace{-0.06cm}\begin{pmatrix}
1+\tilde{r}_{1}(\omega^{2}k)\check{r}_{2}(\omega^{2}k) & r_{2}(\frac{1}{k})\hspace{-0.04cm}-\hspace{-0.04cm}\check{r}_{2}(\omega^{2}k)r_{2}(\omega k) & -\check{r}_{2}(\omega^{2}k) \\
r_{1}(\frac{1}{k})\hspace{-0.04cm}+\hspace{-0.04cm}\tilde{r}_{1}(\omega^{2}k)\big( \hat{r}_{2}(\frac{1}{\omega k}) \hspace{-0.04cm}+\hspace{-0.04cm} r_{1}(\frac{1}{k})\check{r}_{2}(\omega^{2}k) \big) & f_{1}(\omega k) & -\hat{r}_{2}(\frac{1}{\omega k}) \hspace{-0.04cm}-\hspace{-0.04cm} r_{1}(\frac{1}{k})\check{r}_{2}(\omega^{2}k) \\[0.05cm]
-\tilde{r}_{1}(\omega^{2}k) & r_{2}(\omega k) & 1
\end{pmatrix}\hspace{-0.06cm},
\end{align*} \normalsize
where 
\begin{subequations}\label{def of h}
\begin{align}
& h(k) := -R_{1}(\omega^{2}k)r_{2}(\tfrac{1}{k}), & & g(k) := -r_{2}(k) \big( R_{1}(\tfrac{1}{\omega k}) + R_{1}(\omega^{2}k)\tilde{r}_{2}(\tfrac{1}{k}) \big), \\
& \tilde{h}(k) := R_{1}(\omega^{2}k)\tilde{r}_{2}(\tfrac{1}{k}), & & \tilde{g}(k) := \tilde{r}_{2}(k) \big( R_{1}(\tfrac{1}{\omega k}) + R_{1}(\omega^{2}k)r_{2}(\tfrac{1}{k}) \big).
\end{align}
\end{subequations}
where $\tilde{v}_j$ denotes the restriction of $\tilde{v}$ to $\Gamma_{j}$, and where $\{f_{j}(k)\}_{j=1}^{6}$ are defined for $k \in \partial \D$ by
\begin{subequations}\label{def of f}
\begin{align}
f_{1}(k) & = 1 + r_{1}(\tfrac{1}{\omega^{2}k})r_{2}(\tfrac{1}{\omega^{2}k})+r_{1}(k)r_{2}(k) \\
& = 1 + r_{1}(\tfrac{1}{\omega^{2}k})r_{2}(\tfrac{1}{\omega^{2}k})-r_{2}(k) \big( r_{2}(\tfrac{1}{k}) + r_{1}(\tfrac{1}{\omega^{2}k})r_{2}(\omega k) \big) \\
& = 1 + r_{1}(\tfrac{1}{\omega^{2}k})r_{2}(\tfrac{1}{\omega^{2}k})-r_{2}(k) \big( \hat{r}_{2}(\tfrac{1}{k}) + r_{1}(\tfrac{1}{\omega^{2}k})\check{r}_{2}(\omega k) \big), \\
f_{2}(k) & = 1 + \tilde{r}_{1}(\tfrac{1}{\omega^{2}k})\tilde{r}_{2}(\tfrac{1}{\omega^{2}k})+\tilde{r}_{1}(k)\tilde{r}_{2}(k) \\
& = 1 + \tilde{r}_{1}(\tfrac{1}{\omega^{2}k})\tilde{r}_{2}(\tfrac{1}{\omega^{2}k})-\tilde{r}_{2}(k) \big( \tilde{r}_{2}(\tfrac{1}{k}) + \tilde{r}_{1}(\tfrac{1}{\omega^{2}k})\tilde{r}_{2}(\omega k) \big) \\
& = 1 + \tilde{r}_{1}(\tfrac{1}{\omega^{2}k})\tilde{r}_{2}(\tfrac{1}{\omega^{2}k})-\tilde{r}_{2}(k) \big( \check{r}_{2}(\tfrac{1}{k}) + \tilde{r}_{1}(\tfrac{1}{\omega^{2}k})\hat{r}_{2}(\omega k) \big), \\
f_{3}(k) & = 1 + \tilde{r}_{1}(\tfrac{1}{\omega^{2}k})\check{r}_{2}(\tfrac{1}{\omega^{2}k})-\hat{r}_{2}(k) \big( r_{2}(\tfrac{1}{k}) + \tilde{r}_{1}(\tfrac{1}{\omega^{2}k})r_{2}(\omega k) \big), \\
f_{4}(k) & = 1 + r_{1}(\tfrac{1}{\omega^{2}k})\hat{r}_{2}(\tfrac{1}{\omega^{2}k})-\check{r}_{2}(k) \big( \tilde{r}_{2}(\tfrac{1}{k}) + r_{1}(\tfrac{1}{\omega^{2}k})\tilde{r}_{2}(\omega k) \big).
\end{align}
\end{subequations}

A long but straightforward computation shows that $v$ satisfies the symmetries
\begin{align}\label{vsymm}
v(x,t,k) = \mathcal{A} v(x,t,\omega k)\mathcal{A}^{-1}
 = \mathcal{B} v(x,t, k^{-1})^{-1}\mathcal{B}, \qquad k \in \Gamma\setminus (\Gamma_\star \cup \mathcal{Q}),
\end{align}
where $\mathcal{A}$ and $\mathcal{B}$ are the matrices defined by 
\begin{align}\label{def of Acal and Bcal}
\mathcal{A} := \begin{pmatrix}
0 & 0 & 1 \\
1 & 0 & 0 \\
0 & 1 & 0
\end{pmatrix} \qquad \mbox{ and } \qquad \mathcal{B} := \begin{pmatrix}
0 & 1 & 0 \\
1 & 0 & 0 \\
0 & 0 & 1
\end{pmatrix}.
\end{align}

\begin{RHproblem}[RH problem for $M$]\label{RH problem for M}
Find $M(x,t,k)$ with the following properties:
\begin{enumerate}[(a)]
\item $M(x,t,\cdot) : \mathbb{C}\setminus \Gamma \to \mathbb{C}^{3 \times 3}$ is analytic.

\item The limits of $M(x,t,k)$ as $k$ approaches $\Gamma\setminus (\Gamma_\star \cup \mathcal{Q})$ from the left and right (with respect to the orientation of $\Gamma\setminus (\Gamma_\star \cup \mathcal{Q})$ shown in Figure \ref{fig: Gamma}) exist, are continuous on $\Gamma\setminus (\Gamma_\star \cup \mathcal{Q})$, and are denoted by $M_{+}(x,t,k)$ and $M_{-}(x,t,k)$, respectively. Moreover, they satisfy
\begin{align}\label{Mjumpcondition}
& M_{+}(x,t,k) = M_{-}(x,t,k)v(x,t,k), \qquad k \in \Gamma \setminus (\Gamma_\star \cup \mathcal{Q}),
\end{align}
where $v$ is defined by \eqref{vdef}.

\item As $k \to \infty$, 
\begin{align}\label{asymp for M at infty in RH def}
M(x,t,k) = I + \frac{M^{(1)}(x,t)}{k} + \frac{M^{(2)}(x,t)}{k^{2}} + O\bigg(\frac{1}{k^3}\bigg),
\end{align}
where the matrices $M^{(1)}$ and $M^{(2)}$ depend on $x$ and $t$ but not on $k$, and satisfy
\begin{align}\label{singRHMatinftyb}
M_{12}^{(1)} = M_{13}^{(1)} = M_{12}^{(2)} + M_{21}^{(2)} = 0.
\end{align}

\item There exist matrices $\{\mathcal{M}_{14}^{(l)}(x,t),\widetilde{\mathcal{M}}_{5}^{(l)}(x,t)\}_{l=-1}^{+\infty}$ depending on $x$ and $t$ but not on $k$ such that, for any $N \geq -1$,
\begin{subequations}\label{singRHMat0}
\begin{align}
& M(x,t,k) = \sum_{l=-1}^{N} \mathcal{M}_{14}^{(l)}(x,t)(k-1)^{l} + O((k-1)^{N+1}) \qquad \text{as}\ k \to 1, \ k \in \bar{D}_{14}, \\
& M(x,t,k) = \sum_{l=-1}^{N} \widetilde{\mathcal{M}}_{5}^{(l)}(x,t)(k+1)^{l} + O((k+1)^{N+1}) \qquad \text{as}\ k \to -1, \ k \in \bar{E}_{5}.
\end{align}
\end{subequations}
Furthermore, there exist scalar functions $\alpha, \beta, \gamma, \tilde{\alpha}, \tilde{\beta}, \tilde{\gamma}$ depending on $x$ and $t$, but not on $k$, such that
\begin{align}\nonumber
& \mathcal{M}_{14}^{(-1)}(x,t) = \begin{pmatrix}
\alpha(x,t) & 0 & \beta(x,t) \\
-\alpha(x,t) & 0 & -\beta(x,t) \\
0 & 0 & 0
\end{pmatrix}, & & \mathcal{M}_{14}^{(0)}(x,t) = \begin{pmatrix}
\star & \gamma(x,t) & \star \\
\star & -\gamma(x,t) & \star \\
\star & 0 & \star
\end{pmatrix}, 
	\\ \label{mathcalMcoefficients}
& \widetilde{\mathcal{M}}_{5}^{(-1)}(x,t) = \begin{pmatrix}
\tilde{\alpha}(x,t) & 0 & \tilde{\beta}(x,t) \\
-\tilde{\alpha}(x,t) & 0 & -\tilde{\beta}(x,t) \\
0 & 0 & 0
\end{pmatrix}, & & \widetilde{\mathcal{M}}_{5}^{(0)}(x,t) = \begin{pmatrix}
\star & \tilde{\gamma}(x,t) & \star \\
\star & -\tilde{\gamma}(x,t) & \star \\
\star & 0 & \star
\end{pmatrix}.
\end{align}
Here $\star$ denotes an entry whose value is irrelevant for us.

\item $M$ satisfies the symmetries 
\begin{align}\label{symmetry of M}
M(x,t, k) = \mathcal{A} M(x,t,\omega k)\mathcal{A}^{-1} = \mathcal{B} M(x,t,\tfrac{1}{k})\mathcal{B}, \qquad k \in \mathbb{C}\setminus \Gamma.
\end{align}
\item $M(x,t,k) = O(1)$ as $k\to k_{\star} \in \Gamma_\star$.
\end{enumerate}
\end{RHproblem}

We prove in Appendix \ref{Appendix:uniqueness} that the solution to RH problem \ref{RH problem for M} is unique whenever it exists. The existence of a solution to RH problem \ref{RH problem for M} follows if u exists; see Theorem \ref{thm:inverse sca}.

Our second theorem states that the solution $u(x,t)$ of the Boussinesq equation \eqref{badboussinesq} on the half-line domain \eqref{halflinedomain} with initial-boundary values \eqref{initial data u}-\eqref{boundaryvalues u} can be recovered from the unique solution $M$ of RH problem \ref{RH problem for M}.

\begin{theorem}[Solution of (\ref{badboussinesq}) via inverse scattering]\label{thm:inverse sca}
Suppose $u$ is a Schwartz class solution of \eqref{badboussinesq} with existence time $T \in (0, \infty)$, initial data $u_0, u_{1} \in \mathcal{S}(\R_{+})$, and boundary values $\tilde{u}_{0},\tilde{u}_{1}, \tilde{u}_{2}, \tilde{u}_3 \in C^{\infty}([0,T])$ such that Assumptions \ref{solitonlessassumption} and \ref{originassumption} hold. Define the functions $r_1, \tilde{r}_1, r_{2}, \tilde{r}_{2}, \hat{r}_{2}, \check{r}_{2},R_{1}, R_{2}, \tilde{R}_{2}$ in terms of $u_0, u_1, \tilde{u}_{0},\tilde{u}_{1}, \tilde{u}_{2}, \tilde{u}_3$ by \eqref{r1r2def}--\eqref{def of R1 R2}.
Then RH problem \ref{RH problem for M} has a unique solution $M(x,t,k)$ for each $(x,t) \in \R_{+} \times [0,T]$ and the formulas
\begin{align}\label{recoveruv intro}
 \displaystyle{u(x,t) = -i\sqrt{3}\frac{\partial}{\partial x}\lim_{k\to \infty}k\big[(M(x,t,k))_{33} - 1\big] = \frac{1-\omega}{2} \lim_{k\to \infty}k^{2}(M(x,t,k))_{32},}
\end{align}
expressing $u(x,t)$ in terms of $M(x,t,k)$ are valid for all $(x,t) \in \R_{+} \times [0,T]$.
\end{theorem}

\subsection{An equivalent system}

In \cite{Z1974}, Zakharov established a Lax pair for the system
\begin{align}\label{boussinesqsystem}
& \begin{cases}
 v_{t} = u_{x} + (u^2)_{x} + u_{xxx},
 \\
 u_t = v_x.
\end{cases}
\end{align}
The goal of this short subsection is to show that all results presented above on \eqref{badboussinesq} can be reformulated as results for \eqref{boussinesqsystem}. 

\begin{definition}\label{Schwartzsolutiondef}\upshape
We say that $\{u(x,t), v(x,t)\}$ is a {\it Schwartz class solution of \eqref{boussinesqsystem} with existence time $T\in (0,\infty)$, initial data $u_0, v_0 \in \mathcal{S}(\R_{+})$ and boundary values $\tilde{u}_{0},\tilde{u}_{1}, \tilde{u}_{2}, \tilde{v}_0 \in C^{\infty}([0,T])$} if
\begin{enumerate}[$(i)$] 
  \item $u,v$ are smooth real-valued functions of $(x,t) \in \R_{+} \times [0,T]$.

\item $u,v$ satisfy \eqref{boussinesqsystem} for $(x,t) \in \R_{+} \times [0,T]$ and 
\begin{align}
& u(x,0) = u_0(x), \quad v(x,0) = v_0(x), & & x \in \R_{+}, \label{initial data u v} \\
& u(0,t) = \tilde{u}_{0}(t), \quad u_{x}(0,t) = \tilde{u}_{1}(t), \quad u_{xx}(0,t) = \tilde{u}_{2}(t), \quad v(0,t) = \tilde{v}_{0}(t), & & t \in [0,T]. \label{boundaryvalues u v}
\end{align}

  \item $u,v$ have rapid decay as $x \to +\infty$ in the sense that, for each integer $N \geq 1$,
\begin{align}\label{rapiddecay u v}
\sup_{\substack{x \in [0,\infty) \\ t \in [0, T]}} \sum_{i =0}^N (1+|x|)^N(|\partial_x^i u(x,t)| + |\partial_x^i v(x,t)| ) < \infty.
\end{align}
\end{enumerate} 
\end{definition}

The following lemma establishes that the initial-boundary value problems for \eqref{badboussinesq} and \eqref{boussinesqsystem} are equivalent. (In particular, Theorems \ref{thm:r1r2} and \ref{thm:inverse sca} on \eqref{badboussinesq} can be reformulated as results for \eqref{boussinesqsystem}; however, for conciseness, we have chosen not to do so.) 

\begin{lemma}\label{lemma:equivalence}
If $\{u,v\}$ is a Schwartz class solution of \eqref{boussinesqsystem} with existence time $T\in (0, \infty)$, initial data $u_0, v_0 \in \mathcal{S}(\R_{+})$ and boundary values $\tilde{u}_{0},\tilde{u}_{1}, \tilde{u}_{2}, \tilde{v}_0 \in C^{\infty}([0,T])$, then $u$ is Schwartz class solution of \eqref{badboussinesq} with existence time $T$, initial data $u_0, u_1 \in \mathcal{S}(\R_{+})$ and boundary values $\tilde{u}_{0},\tilde{u}_{1}, \tilde{u}_{2}, \tilde{u}_{3} \in C^{\infty}([0,T])$, where 
\begin{align}\label{def of u1 ut3}
u_1(x) := u_t(x,0) = v_{0x}(x) \quad \mbox{ and } \quad \tilde{u}_{3}(t) := u_{xxx}(0,t) = (\tilde{v}_{0t}-\tilde{u}_{1}-2\tilde{u}_{0}\tilde{u}_{1})(t).
\end{align}

Conversely, suppose $u$ is Schwartz class solution of \eqref{badboussinesq} with existence time $T\in (0, \infty)$, initial data $u_0, u_{1} \in \mathcal{S}(\R_{+})$ and boundary values $\tilde{u}_{0},\tilde{u}_{1}, \tilde{u}_{2}, \tilde{u}_{3} \in C^{\infty}([0,T])$. Let 
\begin{align}
& v(x,t) = \int_{+\infty}^x u_t(x', t) dx', \qquad v_0(x) = \int_{+\infty}^x u_1(x') dx', \label{vxtdef} \\
& \tilde{v}_{0}(t) = \int_{+\infty}^0 u_1(x') dx' + \int_{0}^{t}(\tilde{u}_{1}+2\tilde{u}_{0}\tilde{u}_{1}+\tilde{u}_{3})(t')dt'. \nonumber
\end{align}
Then $\{u,v\}$ is a Schwartz class solution of \eqref{boussinesqsystem} with existence time $T$, initial data $u_0, v_0 \in \mathcal{S}(\R_{+})$ and boundary values $\tilde{u}_{0},\tilde{u}_{1}, \tilde{u}_{2}, \tilde{v}_0 \in C^{\infty}([0,T])$.
\end{lemma}
\begin{proof}
See Appendix \ref{Appendix:equivalence}.
\end{proof}

\subsection{Notation}\label{notationsubsec}
The following notation will be used throughout the article.

\begin{enumerate}[$-$]
\item $C>0$ and $c>0$ denote generic constants that may change within a computation.

\item $[A]_j$ denotes the $j$th-column of a matrix $A$.

\item If $A$ is an $n \times m$ matrix, $|A| \ge 0$ is defined by $|A|^2=\Sigma_{i,j}|A_{ij}|^2$. Note that $|A + B| \leq |A| + |B|$ and $|AB| \leq |A| |B|$.

\item $\D := \{k \in \C \, | \, |k| < 1\}$ and $\partial \D := \{k \in \C \, | \, |k| = 1\}$. 



\item $\R_{+}:=[0,+\infty)$ and $\mathcal{S}(\R_{+})$ is the Schwartz space of all smooth functions $f$ on $[0,+\infty)$ such that $f$ and all its derivatives have rapid decay as $x \to + \infty$. 

\item $\kappa_{j} = e^{\frac{\pi i(j-1)}{3}}$, $j=1,\ldots,6$, are the sixth roots of unity, see Figure \ref{fig: Gamma} (right).

\item $\mathcal{Q} := \{\kappa_{j}\}_{j=1}^{6}$ and $\hat{\mathcal{Q}} := \mathcal{Q} \cup \{0\}$.

\item $D_n$, $E_{n}$, $n = 1, \dots, 18$, are the open subsets of $\C$ shown in Figure \ref{fig: Gamma} (right).

\item $\Gamma = \cup_{j=1}^{18} \overline{\Gamma_j \cup \Gamma_{j'} \cup \Gamma_{j''}}$ is the contour shown and oriented as in Figure \ref{fig: Gamma} (left).


\end{enumerate}

\section{Spectral analysis}\nequation\label{laxsec}

Let $T>0$ be fixed. Throughout the rest of the paper, we suppose that $u_0, v_{0} \in \mathcal{S}(\R_{+})$, $\tilde{u}_{0},\tilde{u}_{1}, \tilde{u}_{2}, \tilde{v}_{0} \in C^{\infty}([0,T])$ and that $\{u,v\}$ is a Schwartz class solution of \eqref{boussinesqsystem} with existence time $T$, initial data $u_0, v_{0}$ and boundary values $\tilde{u}_{0},\tilde{u}_{1}, \tilde{u}_{2}, \tilde{v}_{0}$. 

 By \cite{Z1974}, the system (\ref{boussinesqsystem}) is the compatibility condition of the Lax pair
\begin{equation}\label{Xlax}
\begin{cases}
\mu_{x} - [\mathcal{L},\mu] = \mathsf{U} \mu, \\
\mu_{t} - [\mathcal{Z},\mu] = \mathsf{V} \mu,
\end{cases}
\end{equation}
where $\mathcal{L} = \mbox{diag}(l_{1},l_{2},l_{3})$, $\mathcal{Z} = \mbox{diag}(z_{1},z_{2},z_{3})$, 
\begin{align*}
& \mathsf{U} = L - \mathcal{L}, \qquad \mathsf{V} = Z - \mathcal{Z}, \qquad L = P^{-1}\tilde{L}P, \qquad Z = P^{-1}\tilde{Z}P, \\
& \tilde{L} = \begin{pmatrix}
0 & 1 & 0 \\
0 & 0 & 1 \\
\frac{\lambda}{12 i \sqrt{3}}-\frac{u_{x}}{4}-\frac{iv}{4\sqrt{3}} & -\frac{1+2u}{4} & 0
\end{pmatrix}, \quad
\tilde{Z} = \begin{pmatrix}
-i \frac{1+2u}{2\sqrt{3}} & 0 & -i \sqrt{3} \\
- \frac{\lambda}{12} - i \frac{u_{x}}{4\sqrt{3}} - \frac{v}{4} & i \frac{1+2u}{4\sqrt{3}} & 0 \\
-i \frac{u_{xx}}{4\sqrt{3}}-\frac{v_{x}}{4} & - \frac{\lambda}{12} + \frac{iu_{x}}{4\sqrt{3}}-\frac{v}{4} & i \frac{1+2u}{4\sqrt{3}}
\end{pmatrix},
\end{align*}
$\lambda = \frac{k^3 + k^{-3}}{2}$, $P$ is defined in \eqref{Pdef intro}, and $\mu(x,t,k)$ is a $3\times 3$-matrix valued function depending on $x \geq 0$, $t\in [0,T]$ and $k\in \C\setminus \hat{\mathcal{Q}}$. The adjoint system of \eqref{Xlax} is 
\begin{equation}\label{Xlax adjoint}
\begin{cases}
\mu_{x}^{A} + [\mathcal{L},\mu^{A}] = -\mathsf{U}^{T} \mu^{A}, \\
\mu_{t}^{A} + [\mathcal{Z},\mu^{A}] = -\mathsf{V}^{T} \mu^{A},
\end{cases}
\end{equation}
in the sense that if $\mu(x,t,k)$ is a $3\times 3$ invertible matrix valued function satisfying \eqref{Xlax}, then $\mu^{A} := (\mu^{-1})^{T}$ satisfies \eqref{Xlax adjoint}.

We write the systems (\ref{Xlax}) and \eqref{Xlax adjoint} in differential forms as
\begin{equation}\label{mulaxdiffform}
d\left(e^{-\hat{\mathcal{L}}x - \hat{\mathcal{Z}}t} \mu \right) = W, \qquad d\left(e^{\hat{\mathcal{L}}x + \hat{\mathcal{Z}}t} \mu^{A} \right) = W^{A},
\end{equation}
where $W(x,t,k)$ and $W^{A}(x,t,k)$ are the closed one-forms defined by
\begin{equation}\label{Wdef}  
W = e^{-\hat{\mathcal{L}}x - \hat{\mathcal{Z}}t}[(\mathsf{U} dx + \mathsf{V} dt) \mu], \qquad W^{A} = -e^{\hat{\mathcal{L}}x + \hat{\mathcal{Z}}t}[(\mathsf{U}^{T} dx + \mathsf{V}^{T} dt) \mu^{A}].
\end{equation}  

\subsection{Preliminaries}

It will be convenient for us to introduce the following 36 open, pairwisely disjoint subsets $\{D_n\}_1^{18}$ and $\{E_n\}_1^{18}$ of the complex plane by (see also Figure \ref{fig: Gamma} (right))

\begin{align}
D_{1} & = \{k \in \C \,|\, \text{Re}\,l_3 < \text{Re}\,l_2 < \text{Re}\,l_1 \text{  and  } 
\text{Re}\,z_1 < \text{Re}\,z_3 < \text{Re}\,z_2 \}, \nonumber \\
D_{2} & = \{k \in \C \,|\, \text{Re}\,l_3 < \text{Re}\,l_2 < \text{Re}\,l_1 \text{  and  } 
\text{Re}\,z_1 < \text{Re}\,z_2 < \text{Re}\,z_3 \}, \nonumber \\
D_{3} & = \{k \in \C \,|\, \text{Re}\,l_3 < \text{Re}\,l_2 < \text{Re}\,l_1 \text{  and  } 
\text{Re}\,z_2 < \text{Re}\,z_1 < \text{Re}\,z_3 \}, \nonumber \\
D_{4} & = \{k \in \C \,|\, \text{Re}\,l_2 < \text{Re}\,l_3 < \text{Re}\,l_1 \text{  and  } 
\text{Re}\,z_2 < \text{Re}\,z_1 < \text{Re}\,z_3 \}, \nonumber \\
D_{5} & = \{k \in \C \,|\, \text{Re}\,l_2 < \text{Re}\,l_3 < \text{Re}\,l_1 \text{  and  } 
\text{Re}\,z_2 < \text{Re}\,z_3 < \text{Re}\,z_1 \}, \nonumber \\
D_{6} & = \{k \in \C \,|\, \text{Re}\,l_2 < \text{Re}\,l_3 < \text{Re}\,l_1 \text{  and  } 
\text{Re}\,z_3 < \text{Re}\,z_2 < \text{Re}\,z_1 \}, \nonumber \\
D_{7} & = \{k \in \C \,|\, \text{Re}\,l_2 < \text{Re}\,l_1 < \text{Re}\,l_3 \text{  and  } 
\text{Re}\,z_3 < \text{Re}\,z_2 < \text{Re}\,z_1 \}, \nonumber \\
D_{8} & = \{k \in \C \,|\, \text{Re}\,l_2 < \text{Re}\,l_1 < \text{Re}\,l_3 \text{  and  } 
\text{Re}\,z_3 < \text{Re}\,z_1 < \text{Re}\,z_2 \}, \nonumber \\
D_{9} & = \{k \in \C \,|\, \text{Re}\,l_2 < \text{Re}\,l_1 < \text{Re}\,l_3 \text{  and  } 
\text{Re}\,z_1 < \text{Re}\,z_3 < \text{Re}\,z_2 \}, \nonumber \\
D_{10} & = \{k \in \C \,|\, \text{Re}\,l_1 < \text{Re}\,l_2 < \text{Re}\,l_3 \text{  and  } 
\text{Re}\,z_1 < \text{Re}\,z_3 < \text{Re}\,z_2 \}, \nonumber \\
D_{11} & = \{k \in \C \,|\, \text{Re}\,l_1 < \text{Re}\,l_2 < \text{Re}\,l_3 \text{  and  } 
\text{Re}\,z_1 < \text{Re}\,z_2 < \text{Re}\,z_3 \}, \nonumber \\
D_{12} & = \{k \in \C \,|\, \text{Re}\,l_1 < \text{Re}\,l_2 < \text{Re}\,l_3 \text{  and  } 
\text{Re}\,z_2 < \text{Re}\,z_1 < \text{Re}\,z_3 \}, \nonumber \\
D_{13} & = \{k \in \C \,|\, \text{Re}\,l_1 < \text{Re}\,l_3 < \text{Re}\,l_2 \text{  and  } 
\text{Re}\,z_2 < \text{Re}\,z_1 < \text{Re}\,z_3 \}, \nonumber \\
D_{14} & = \{k \in \C \,|\, \text{Re}\,l_1 < \text{Re}\,l_3 < \text{Re}\,l_2 \text{  and  } 
\text{Re}\,z_2 < \text{Re}\,z_3 < \text{Re}\,z_1 \}, \nonumber \\
D_{15} & = \{k \in \C \,|\, \text{Re}\,l_1 < \text{Re}\,l_3 < \text{Re}\,l_2 \text{  and  } 
\text{Re}\,z_3 < \text{Re}\,z_2 < \text{Re}\,z_1 \}, \nonumber \\
D_{16} & = \{k \in \C \,|\, \text{Re}\,l_3 < \text{Re}\,l_1 < \text{Re}\,l_2 \text{  and  } 
\text{Re}\,z_3 < \text{Re}\,z_2 < \text{Re}\,z_1 \}, \nonumber \\
D_{17} & = \{k \in \C \,|\, \text{Re}\,l_3 < \text{Re}\,l_1 < \text{Re}\,l_2 \text{  and  } 
\text{Re}\,z_3 < \text{Re}\,z_1 < \text{Re}\,z_2 \}, \nonumber \\
D_{18} & = \{k \in \C \,|\, \text{Re}\,l_3 < \text{Re}\,l_1 < \text{Re}\,l_2 \text{  and  } 
\text{Re}\,z_1 < \text{Re}\,z_3 < \text{Re}\,z_2 \}, \nonumber \\
E_{1} & = \{k \in \C \,|\, \text{Re}\,l_1 < \text{Re}\,l_2 < \text{Re}\,l_3 \text{  and  } 
\text{Re}\,z_2 < \text{Re}\,z_3 < \text{Re}\,z_1 \}, \nonumber \\
E_{2} & = \{k \in \C \,|\, \text{Re}\,l_1 < \text{Re}\,l_2 < \text{Re}\,l_3 \text{  and  } 
\text{Re}\,z_3 < \text{Re}\,z_2 < \text{Re}\,z_1 \}, \nonumber \\
E_{3} & = \{k \in \C \,|\, \text{Re}\,l_1 < \text{Re}\,l_2 < \text{Re}\,l_3 \text{  and  } 
\text{Re}\,z_3 < \text{Re}\,z_1 < \text{Re}\,z_2 \}, \nonumber \\
E_{4} & = \{k \in \C \,|\, \text{Re}\,l_1 < \text{Re}\,l_3 < \text{Re}\,l_2 \text{  and  } 
\text{Re}\,z_3 < \text{Re}\,z_1 < \text{Re}\,z_2 \}, \nonumber \\
E_{5} & = \{k \in \C \,|\, \text{Re}\,l_1 < \text{Re}\,l_3 < \text{Re}\,l_2 \text{  and  } 
\text{Re}\,z_1 < \text{Re}\,z_3 < \text{Re}\,z_2 \}, \nonumber \\
E_{6} & = \{k \in \C \,|\, \text{Re}\,l_1 < \text{Re}\,l_3 < \text{Re}\,l_2 \text{  and  } 
\text{Re}\,z_1 < \text{Re}\,z_2 < \text{Re}\,z_3 \}, \nonumber \\
E_{7} & = \{k \in \C \,|\, \text{Re}\,l_3 < \text{Re}\,l_1 < \text{Re}\,l_2 \text{  and  } 
\text{Re}\,z_1 < \text{Re}\,z_2 < \text{Re}\,z_3 \}, \nonumber \\
E_{8} & = \{k \in \C \,|\, \text{Re}\,l_3 < \text{Re}\,l_1 < \text{Re}\,l_2 \text{  and  } 
\text{Re}\,z_2 < \text{Re}\,z_1 < \text{Re}\,z_3 \}, \nonumber \\
E_{9} & = \{k \in \C \,|\, \text{Re}\,l_3 < \text{Re}\,l_1 < \text{Re}\,l_2 \text{  and  } 
\text{Re}\,z_2 < \text{Re}\,z_3 < \text{Re}\,z_1 \}, \nonumber \\
E_{10} & = \{k \in \C \,|\, \text{Re}\,l_3 < \text{Re}\,l_2 < \text{Re}\,l_1 \text{  and  } 
\text{Re}\,z_2 < \text{Re}\,z_3 < \text{Re}\,z_1 \}, \nonumber \\
E_{11} & = \{k \in \C \,|\, \text{Re}\,l_3 < \text{Re}\,l_2 < \text{Re}\,l_1 \text{  and  } 
\text{Re}\,z_3 < \text{Re}\,z_2 < \text{Re}\,z_1 \}, \nonumber \\
E_{12} & = \{k \in \C \,|\, \text{Re}\,l_3 < \text{Re}\,l_2 < \text{Re}\,l_1 \text{  and  } 
\text{Re}\,z_3 < \text{Re}\,z_1 < \text{Re}\,z_2 \}, \nonumber \\
E_{13} & = \{k \in \C \,|\, \text{Re}\,l_2 < \text{Re}\,l_3 < \text{Re}\,l_1 \text{  and  } 
\text{Re}\,z_3 < \text{Re}\,z_1 < \text{Re}\,z_2 \}, \nonumber \\
E_{14} & = \{k \in \C \,|\, \text{Re}\,l_2 < \text{Re}\,l_3 < \text{Re}\,l_1 \text{  and  } 
\text{Re}\,z_1 < \text{Re}\,z_3 < \text{Re}\,z_2 \}, \nonumber \\
E_{15} & = \{k \in \C \,|\, \text{Re}\,l_2 < \text{Re}\,l_3 < \text{Re}\,l_1 \text{  and  } 
\text{Re}\,z_1 < \text{Re}\,z_2 < \text{Re}\,z_3 \}, \nonumber \\
E_{16} & = \{k \in \C \,|\, \text{Re}\,l_2 < \text{Re}\,l_1 < \text{Re}\,l_3 \text{  and  } 
\text{Re}\,z_1 < \text{Re}\,z_2 < \text{Re}\,z_3 \}, \nonumber \\
E_{17} & = \{k \in \C \,|\, \text{Re}\,l_2 < \text{Re}\,l_1 < \text{Re}\,l_3 \text{  and  } 
\text{Re}\,z_2 < \text{Re}\,z_1 < \text{Re}\,z_3 \}, \nonumber \\
E_{18} & = \{k \in \C \,|\, \text{Re}\,l_2 < \text{Re}\,l_1 < \text{Re}\,l_3 \text{  and  } 
\text{Re}\,z_2 < \text{Re}\,z_3 < \text{Re}\,z_1 \}, \label{def of Dj and Ej}
\end{align}
and let $\mathcal{S}_{j}$, $j=1,2,3$ be defined by (see also Figure \ref{fig: S1S2S3})
\begin{figure}
\begin{center}
\scalebox{0.5}{\begin{tikzpicture}[scale=1.1, master]
\node at (0,0) {};

\draw[line width=0.025 mm, dashed,fill=gray!30] (0,0) --  (-105:4.5) arc(-105:-150:4.5) -- cycle;
\draw[line width=0.25 mm, dashed,fill=white] (0,0) --  (-105:3) arc(-105:-150:3) -- cycle;
\node at (-130:3.8) {$\mathcal{S}_{1}$ };

\draw[line width=0.025 mm, dashed,fill=gray!30] (0,0) --  (-30:4.5) arc(-30:15:4.5) -- cycle;
\draw[line width=0.25 mm, dashed,fill=white] (0,0) --  (-30:3) arc(-30:15:3) -- cycle;
\node at (-5:3.8) {$\omega\mathcal{S}_{1}$ };

\draw[line width=0.025 mm, dashed,fill=gray!30] (0,0) --  (90:4.5) arc(90:135:4.5) -- cycle;
\draw[line width=0.25 mm, dashed,fill=white] (0,0) --  (90:3) arc(90:135:3) -- cycle;
\node at (115:3.8) {$\omega^{2}\mathcal{S}_{1}$ };

\draw[line width=0.25 mm, dashed,fill=gray!30] (0,0) --  (105:3) arc(105:150:3) -- cycle;
\node at (130:2) {$\mathcal{S}_{1}$ };

\draw[line width=0.25 mm, dashed,fill=gray!30] (0,0) --  (-135:3) arc(-135:-90:3) -- cycle;
\node at (-110:2) {$\omega \mathcal{S}_{1}$ };

\draw[line width=0.25 mm, dashed,fill=gray!30] (0,0) --  (-15:3) arc(-15:30:3) -- cycle;
\node at (10:2) {$\omega^{2} \mathcal{S}_{1}$ };

\draw[black,line width=0.45 mm] (15:3)--(15:4.5);
\draw[black,line width=0.45 mm] (0,0)--(30:4.5);
\draw[black,line width=0.45 mm] (0,0)--(45:4.5);
\draw[black,line width=0.45 mm] (0,0)--(75:4.5);
\draw[black,line width=0.45 mm] (0,0)--(90:4.5);
\draw[black,line width=0.45 mm] (0,0)--(105:3);
\draw[black,line width=0.45 mm] (135:3)--(135:4.5);
\draw[black,line width=0.45 mm] (0,0)--(150:4.5);
\draw[black,line width=0.45 mm] (0,0)--(165:4.5);
\draw[black,line width=0.45 mm] (0,0)--(-15:3);
\draw[black,line width=0.45 mm] (0,0)--(-30:4.5);
\draw[black,line width=0.45 mm] (0,0)--(-45:4.5);
\draw[black,line width=0.45 mm] (0,0)--(-75:4.5);
\draw[black,line width=0.45 mm] (0,0)--(-90:4.5);
\draw[black,line width=0.45 mm] (-105:3)--(-105:4.5);
\draw[black,line width=0.45 mm] (0,0)--(-135:3);
\draw[black,line width=0.45 mm] (0,0)--(-150:4.5);
\draw[black,line width=0.45 mm] (0,0)--(-165:4.5);

\draw[black,line width=0.45 mm] ([shift=(-180:3cm)]0,0) arc (-180:180:3cm);

\draw[fill] (0:3) circle (0.08);
\draw[fill] (60:3) circle (0.08);
\draw[fill] (120:3) circle (0.08);
\draw[fill] (180:3) circle (0.08);
\draw[fill] (240:3) circle (0.08);
\draw[fill] (300:3) circle (0.08);

\end{tikzpicture} \hspace{0.05cm} \begin{tikzpicture}[scale=1.1,slave]
\node at (0,0) {};

\draw[line width=0.025 mm, dashed,fill=gray!30] (0,0) --  (-75:4.5) arc(-75:-30:4.5) -- cycle;
\draw[line width=0.25 mm, dashed,fill=white] (0,0) --  (-75:3) arc(-75:-30:3) -- cycle;
\node at (-55:3.8) {$\mathcal{S}_{2}$ };

\draw[line width=0.025 mm, dashed,fill=gray!30] (0,0) --  (45:4.5) arc(45:90:4.5) -- cycle;
\draw[line width=0.25 mm, dashed,fill=white] (0,0) --  (45:3) arc(45:90:3) -- cycle;
\node at (65:3.8) {$\omega\mathcal{S}_{2}$ };

\draw[line width=0.025 mm, dashed,fill=gray!30] (0,0) --  (165:4.5) arc(165:210:4.5) -- cycle;
\draw[line width=0.25 mm, dashed,fill=white] (0,0) --  (165:3) arc(165:210:3) -- cycle;
\node at (185:3.8) {$\omega^{2}\mathcal{S}_{2}$ };

\draw[line width=0.25 mm, dashed,fill=gray!30] (0,0) --  (30:3) arc(30:75:3) -- cycle;
\node at (55:2) {$\mathcal{S}_{2}$ };

\draw[line width=0.25 mm, dashed,fill=gray!30] (0,0) --  (150:3) arc(150:195:3) -- cycle;
\node at (175:2) {$\omega \mathcal{S}_{2}$ };

\draw[line width=0.25 mm, dashed,fill=gray!30] (0,0) --  (270:3) arc(270:315:3) -- cycle;
\node at (295:2) {$\omega^{2} \mathcal{S}_{2}$ };

\draw[black,line width=0.45 mm] (0,0)--(15:4.5);
\draw[black,line width=0.45 mm] (0,0)--(30:4.5);
\draw[black,line width=0.45 mm] (45:3)--(45:4.5);
\draw[black,line width=0.45 mm] (0,0)--(75:3);
\draw[black,line width=0.45 mm] (0,0)--(90:4.5);
\draw[black,line width=0.45 mm] (0,0)--(105:4.5);
\draw[black,line width=0.45 mm] (0,0)--(135:4.5);
\draw[black,line width=0.45 mm] (0,0)--(150:4.5);
\draw[black,line width=0.45 mm] (165:3)--(165:4.5);
\draw[black,line width=0.45 mm] (0,0)--(-15:4.5);
\draw[black,line width=0.45 mm] (0,0)--(-30:4.5);
\draw[black,line width=0.45 mm] (0,0)--(-45:3);
\draw[black,line width=0.45 mm] (-75:3)--(-75:4.5);
\draw[black,line width=0.45 mm] (0,0)--(-90:4.5);
\draw[black,line width=0.45 mm] (0,0)--(-105:4.5);
\draw[black,line width=0.45 mm] (0,0)--(-135:4.5);
\draw[black,line width=0.45 mm] (0,0)--(-150:4.5);
\draw[black,line width=0.45 mm] (0,0)--(-165:3);

\draw[black,line width=0.45 mm] ([shift=(-180:3cm)]0,0) arc (-180:180:3cm);

\draw[fill] (0:3) circle (0.08);
\draw[fill] (60:3) circle (0.08);
\draw[fill] (120:3) circle (0.08);
\draw[fill] (180:3) circle (0.08);
\draw[fill] (240:3) circle (0.08);
\draw[fill] (300:3) circle (0.08);

\end{tikzpicture} \hspace{0.05cm}  \begin{tikzpicture}[scale=1.1,slave]
\node at (0,0) {};

\draw[line width=0.025 mm, dashed,fill=gray!30] (0,0) --  (30:4.5) arc(30:150:4.5) -- cycle;
\draw[line width=0.25 mm, dashed,fill=white] (0,0) --  (30:3) arc(30:150:3) -- cycle;
\node at (90:3.8) {$\mathcal{S}_{3}$ };

\draw[line width=0.025 mm, dashed,fill=gray!30] (0,0) --  (150:4.5) arc(150:270:4.5) -- cycle;
\draw[line width=0.25 mm, dashed,fill=white] (0,0) --  (150:3) arc(150:270:3) -- cycle;
\node at (210:3.8) {$\omega\mathcal{S}_{3}$ };

\draw[line width=0.025 mm, dashed,fill=gray!30] (0,0) --  (270:4.5) arc(270:390:4.5) -- cycle;
\draw[line width=0.25 mm, dashed,fill=white] (0,0) --  (270:3) arc(270:390:3) -- cycle;
\node at (330:3.8) {$\omega^{2}\mathcal{S}_{3}$ };

\draw[line width=0.25 mm, dashed,fill=gray!30] (0,0) --  (-150:3) arc(-150:-30:3) -- cycle;
\node at (-90:2) {$\mathcal{S}_{3}$ };

\draw[line width=0.25 mm, dashed,fill=gray!30] (0,0) --  (-30:3) arc(-30:90:3) -- cycle;
\node at (30:2) {$\omega\mathcal{S}_{3}$ };

\draw[line width=0.25 mm, dashed,fill=gray!30] (0,0) --  (90:3) arc(90:210:3) -- cycle;
\node at (150:2) {$\omega^{2}\mathcal{S}_{3}$ };

\draw[black,line width=0.45 mm] (30:3)--(30:4.5);
\draw[black,line width=0.45 mm] (0,0)--(90:3);
\draw[black,line width=0.45 mm] (150:3)--(150:4.5);
\draw[black,line width=0.45 mm] (0,0)--(-30:3);
\draw[black,line width=0.45 mm] (-90:3)--(-90:4.5);
\draw[black,line width=0.45 mm] (0,0)--(-150:3);

\draw[black,line width=0.45 mm] ([shift=(-180:3cm)]0,0) arc (-180:180:3cm);

\draw[fill] (0:3) circle (0.08);
\draw[fill] (60:3) circle (0.08);
\draw[fill] (120:3) circle (0.08);
\draw[fill] (180:3) circle (0.08);
\draw[fill] (240:3) circle (0.08);
\draw[fill] (300:3) circle (0.08);
\end{tikzpicture}}
\\[0.3cm]
\scalebox{0.5}{\begin{tikzpicture}[scale=1.1, master]
\node at (0,0) {};

\draw[line width=0.25 mm, dashed,fill=gray!30] (0,0) --  (105:4.5) arc(105:150:4.5) -- cycle;
\draw[line width=0.25 mm, dashed,fill=white] (0,0) --  (105:3) arc(105:150:3) -- cycle;
\node at (130:3.8) {$\mathcal{S}_{1}^{*}$ };

\draw[line width=0.25 mm, dashed,fill=gray!30] (0,0) --  (-135:4.5) arc(-135:-90:4.5) -- cycle;
\draw[line width=0.25 mm, dashed,fill=white] (0,0) --  (-135:3) arc(-135:-90:3) -- cycle;
\node at (-110:3.8) {$\omega \mathcal{S}_{1}^{*}$ };

\draw[line width=0.25 mm, dashed,fill=gray!30] (0,0) --  (-15:4.5) arc(-15:30:4.5) -- cycle;
\draw[line width=0.25 mm, dashed,fill=white] (0,0) --  (-15:3) arc(-15:30:3) -- cycle;
\node at (10:3.8) {$\omega^{2} \mathcal{S}_{1}^{*}$ };

\draw[line width=0.25 mm, dashed,fill=gray!30] (0,0) --  (-105:3) arc(-105:-150:3) -- cycle;
\node at (-130:2) {$\mathcal{S}_{1}^{*}$ };

\draw[line width=0.25 mm, dashed,fill=gray!30] (0,0) --  (-30:3) arc(-30:15:3) -- cycle;
\node at (-5:2) {$\omega\mathcal{S}_{1}^{*}$ };

\draw[line width=0.25 mm, dashed,fill=gray!30] (0,0) --  (90:3) arc(90:135:3) -- cycle;
\node at (115:2) {$\omega^{2}\mathcal{S}_{1}^{*}$ };

\draw[black,line width=0.45 mm] (0,0)--(15:3);
\draw[black,line width=0.45 mm] (0,0)--(30:4.5);
\draw[black,line width=0.45 mm] (0,0)--(45:4.5);
\draw[black,line width=0.45 mm] (0,0)--(75:4.5);
\draw[black,line width=0.45 mm] (0,0)--(90:4.5);
\draw[black,line width=0.45 mm] (105:3)--(105:4.5);
\draw[black,line width=0.45 mm] (0,0)--(135:3);
\draw[black,line width=0.45 mm] (0,0)--(150:4.5);
\draw[black,line width=0.45 mm] (0,0)--(165:4.5);
\draw[black,line width=0.45 mm] (-15:3)--(-15:4.5);
\draw[black,line width=0.45 mm] (0,0)--(-30:4.5);
\draw[black,line width=0.45 mm] (0,0)--(-45:4.5);
\draw[black,line width=0.45 mm] (0,0)--(-75:4.5);
\draw[black,line width=0.45 mm] (0,0)--(-90:4.5);
\draw[black,line width=0.45 mm] (0,0)--(-105:3);
\draw[black,line width=0.45 mm] (-135:3)--(-135:4.5);
\draw[black,line width=0.45 mm] (0,0)--(-150:4.5);
\draw[black,line width=0.45 mm] (0,0)--(-165:4.5);

\draw[black,line width=0.45 mm] ([shift=(-180:3cm)]0,0) arc (-180:180:3cm);

\draw[fill] (0:3) circle (0.08);
\draw[fill] (60:3) circle (0.08);
\draw[fill] (120:3) circle (0.08);
\draw[fill] (180:3) circle (0.08);
\draw[fill] (240:3) circle (0.08);
\draw[fill] (300:3) circle (0.08);

\end{tikzpicture} \hspace{0.05cm} \begin{tikzpicture}[scale=1.1,slave]
\node at (0,0) {};

\draw[line width=0.25 mm, dashed,fill=gray!30] (0,0) --  (30:4.5) arc(30:75:4.5) -- cycle;
\draw[line width=0.25 mm, dashed,fill=white] (0,0) --  (30:3) arc(30:75:3) -- cycle;
\node at (55:3.8) {$\mathcal{S}_{2}^{*}$ };

\draw[line width=0.25 mm, dashed,fill=gray!30] (0,0) --  (150:4.5) arc(150:195:4.5) -- cycle;
\draw[line width=0.25 mm, dashed,fill=white] (0,0) --  (150:3) arc(150:195:3) -- cycle;
\node at (175:3.8) {$\omega \mathcal{S}_{2}^{*}$ };

\draw[line width=0.25 mm, dashed,fill=gray!30] (0,0) --  (270:4.5) arc(270:315:4.5) -- cycle;
\draw[line width=0.25 mm, dashed,fill=white] (0,0) --  (270:3) arc(270:315:3) -- cycle;
\node at (295:3.8) {$\omega^{2} \mathcal{S}_{2}^{*}$ };

\draw[line width=0.25 mm, dashed,fill=gray!30] (0,0) --  (-75:3) arc(-75:-30:3) -- cycle;
\node at (-55:2) {$\mathcal{S}_{2}^{*}$ };

\draw[line width=0.25 mm, dashed,fill=gray!30] (0,0) --  (45:3) arc(45:90:3) -- cycle;
\node at (65:2) {$\omega\mathcal{S}_{2}^{*}$ };

\draw[line width=0.25 mm, dashed,fill=gray!30] (0,0) --  (165:3) arc(165:210:3) -- cycle;
\node at (185:2) {$\omega^{2}\mathcal{S}_{2}^{*}$ };

\draw[black,line width=0.45 mm] (0,0)--(15:4.5);
\draw[black,line width=0.45 mm] (0,0)--(30:4.5);
\draw[black,line width=0.45 mm] (0,0)--(45:3);
\draw[black,line width=0.45 mm] (75:3)--(75:4.5);
\draw[black,line width=0.45 mm] (0,0)--(90:4.5);
\draw[black,line width=0.45 mm] (0,0)--(105:4.5);
\draw[black,line width=0.45 mm] (0,0)--(135:4.5);
\draw[black,line width=0.45 mm] (0,0)--(150:4.5);
\draw[black,line width=0.45 mm] (0,0)--(165:3);
\draw[black,line width=0.45 mm] (0,0)--(-15:4.5);
\draw[black,line width=0.45 mm] (0,0)--(-30:4.5);
\draw[black,line width=0.45 mm] (-45:3)--(-45:4.5);
\draw[black,line width=0.45 mm] (0,0)--(-75:3);
\draw[black,line width=0.45 mm] (0,0)--(-90:4.5);
\draw[black,line width=0.45 mm] (0,0)--(-105:4.5);
\draw[black,line width=0.45 mm] (0,0)--(-135:4.5);
\draw[black,line width=0.45 mm] (0,0)--(-150:4.5);
\draw[black,line width=0.45 mm] (-165:3)--(-165:4.5);

\draw[black,line width=0.45 mm] ([shift=(-180:3cm)]0,0) arc (-180:180:3cm);

\draw[fill] (0:3) circle (0.08);
\draw[fill] (60:3) circle (0.08);
\draw[fill] (120:3) circle (0.08);
\draw[fill] (180:3) circle (0.08);
\draw[fill] (240:3) circle (0.08);
\draw[fill] (300:3) circle (0.08);

\end{tikzpicture} \hspace{0.05cm}  \begin{tikzpicture}[scale=1.1,slave]
\node at (0,0) {};

\draw[line width=0.25 mm, dashed,fill=gray!30] (0,0) --  (-150:4.5) arc(-150:-30:4.5) -- cycle;
\draw[line width=0.25 mm, dashed,fill=white] (0,0) --  (-150:3) arc(-150:-30:3) -- cycle;
\node at (-90:3.8) {$\mathcal{S}_{3}^{*}$ };

\draw[line width=0.25 mm, dashed,fill=gray!30] (0,0) --  (-30:4.5) arc(-30:90:4.5) -- cycle;
\draw[line width=0.25 mm, dashed,fill=white] (0,0) --  (-30:3) arc(-30:90:3) -- cycle;
\node at (30:3.8) {$\omega\mathcal{S}_{3}^{*}$ };

\draw[line width=0.25 mm, dashed,fill=gray!30] (0,0) --  (90:4.5) arc(90:210:4.5) -- cycle;
\draw[line width=0.25 mm, dashed,fill=white] (0,0) --  (90:3) arc(90:210:3) -- cycle;
\node at (150:3.8) {$\omega^{2}\mathcal{S}_{3}^{*}$ };

\draw[line width=0.25 mm, dashed,fill=gray!30] (0,0) --  (30:3) arc(30:150:3) -- cycle;
\node at (90:2) {$\mathcal{S}_{3}^{*}$ };

\draw[line width=0.25 mm, dashed,fill=gray!30] (0,0) --  (150:3) arc(150:270:3) -- cycle;
\node at (210:2) {$\omega\mathcal{S}_{3}^{*}$ };

\draw[line width=0.25 mm, dashed,fill=gray!30] (0,0) --  (270:3) arc(270:390:3) -- cycle;
\node at (330:2) {$\omega^{2}\mathcal{S}_{3}^{*}$ };

\draw[black,line width=0.45 mm] (0,0)--(30:3);
\draw[black,line width=0.45 mm] (90:3)--(90:4.5);
\draw[black,line width=0.45 mm] (0,0)--(150:3);
\draw[black,line width=0.45 mm] (-30:3)--(-30:4.5);
\draw[black,line width=0.45 mm] (0,0)--(-90:3);
\draw[black,line width=0.45 mm] (-150:3)--(-150:4.5);

\draw[black,line width=0.45 mm] ([shift=(-180:3cm)]0,0) arc (-180:180:3cm);

\draw[fill] (0:3) circle (0.08);
\draw[fill] (60:3) circle (0.08);
\draw[fill] (120:3) circle (0.08);
\draw[fill] (180:3) circle (0.08);
\draw[fill] (240:3) circle (0.08);
\draw[fill] (300:3) circle (0.08);
\end{tikzpicture}}

\end{center}
\begin{figuretext}\label{fig: S1S2S3}
The sets $\mathcal{S}_{j},\omega\mathcal{S}_{j},\omega^{2}\mathcal{S}_{j}$, $j=1,2,3$ (top), and $\mathcal{S}_{j}^{*},\omega\mathcal{S}_{j}^{*},\omega^{2}\mathcal{S}_{j}^{*}$, $j=1,2,3$ (bottom).
\end{figuretext}
\end{figure}
\begin{align*}
& \mathcal{S}_{1} := \{\re  l_1 < \re  l_3\} \cap \{\re  l_2 < \re  l_3\} \cap \{\re  z_3 < \re  z_1\} \cap \{\re  z_3 < \re  z_2\}, \\
& \mathcal{S}_{2} := \{\re  l_1 < \re  l_3\} \cap \{\re  l_2 < \re  l_3\} \cap \{\re  z_1 < \re  z_3\} \cap \{\re  z_2 < \re  z_3\}, \\
& \mathcal{S}_{3} := \{\re  l_3 < \re  l_1\} \cap \{\re  l_3 < \re  l_2\}.
\end{align*}
Using \eqref{lmexpressions intro} and \eqref{def of Dj and Ej}, these sets can also be rewritten as 
\begin{align*}
\mathcal{S}_{1} & = \big\{k \in \C: \big(\arg k \in (-\tfrac{5\pi}{6},-\tfrac{7\pi}{12}) \mbox{ and } |k| > 1\big) \mbox{ or } \big(\arg k \in (\tfrac{7\pi}{12},\tfrac{5\pi}{6}) \mbox{ and } |k| < 1\big) \big\} \\
& = \mbox{int}(\bar{D}_{7}\cup \bar{D}_{8} \cup \bar{E}_{2}\cup \bar{E}_{3}), \\
\mathcal{S}_{2} & = \big\{k \in \C: \big(\arg k \in (-\tfrac{5\pi}{12},-\tfrac{\pi}{6}) \mbox{ and } |k| > 1\big) \mbox{ or } \big(\arg k \in (\tfrac{\pi}{6},\tfrac{5\pi}{12}) \mbox{ and } |k| < 1\big) \big\} \\
& = \mbox{int}(\bar{D}_{11}\cup \bar{D}_{12} \cup \bar{E}_{16} \cup \bar{E}_{17}), \\
\mathcal{S}_{3} & = \big\{k \in \C: \big(\arg k \in (\tfrac{\pi}{6},\tfrac{5\pi}{6}) \mbox{ and } |k| > 1\big) \mbox{ or } \big(\arg k \in (-\tfrac{5\pi}{6},-\tfrac{\pi}{6}) \mbox{ and } |k| < 1\big) \big\} \\
& = \mbox{int}(\cup_{j=16}^{18} \bar{D}_{j} \cup \cup_{j=1}^{3} \bar{D}_{j} \cup \cup_{j=7}^{12} \bar{E}_{j}).
\end{align*}
Define also $\hat{\mathcal{S}}_{j} = \partial \D \cup \bar{\mathcal{S}}_{j}$, $j=1,2,3$. Given a set $S\subset \C$, we let $S^{*} := \{z : \overline{z}^{-1}\in S\}$. Using \eqref{lmexpressions intro}, we verify that $\mathcal{S}_{1}^{*},\mathcal{S}_{2}^{*},\mathcal{S}_{3}^{*}$ can be rewritten as
\begin{align*}
& \mathcal{S}_{1}^{*} = \{\re  l_3 < \re  l_1\} \cap \{\re  l_3 < \re  l_2\} \cap \{\re  z_1 < \re  z_3\} \cap \{\re  z_2 < \re  z_3\}, \\
& \mathcal{S}_{2}^{*} = \{\re  l_3 < \re  l_1\} \cap \{\re  l_3 < \re  l_2\} \cap \{\re  z_3 < \re  z_1\} \cap \{\re  z_3 < \re  z_2\}, \\
& \mathcal{S}_{3}^{*} = \{\re  l_1 < \re  l_3\} \cap \{\re  l_2 < \re  l_3\}.
\end{align*}   
The open sets $\mathcal{S}_{j},\omega\mathcal{S}_{j},\omega^{2}\mathcal{S}_{j},\mathcal{S}_{j}^{*},\omega\mathcal{S}_{j}^{*},\omega^{2}\mathcal{S}_{j}^{*}$, $j=1,2,3$ are represented in Figure \ref{fig: S1S2S3}.

\subsection{The eigenfunctions $\mu_{1},\mu_{2},\mu_{3}$ and $\mu_{1}^{A},\mu_{2}^{A},\mu_{3}^{A}$}
We define three contours $\{\gamma_j\}_1^3$ in the $(x, t)$-plane going from $(x_j, t_j)$ to $(x,t)$, where $(x_1, t_1) = (0, T)$, $(x_2, t_2) = (0, 0)$, and $(x_3, t_3) = (+\infty, t)$; we choose the particular contours shown in Figure \ref{mucontours.pdf}. This choice implies the following inequalities for $(x',t') \in \gamma_{j}$, $j=1,2,3$:
\begin{align}
\gamma_1: x' - x \leq 0,& \qquad t' - t \geq 0, \nonumber \\
\gamma_2: x' - x \leq 0,& \qquad t' - t \leq 0, \label{contourinequalities}
	\\ 
\gamma_3: x' - x \geq 0.& \nonumber
\end{align}

\noindent We define three eigenfunctions $\{\mu_j\}_{j=1}^3$ of \eqref{Xlax} by 
\begin{align}
\mu_j(x,t,k) & = I + e^{x\hat{\mathcal{L}}(k) + t\hat{\mathcal{Z}}(k)} \int_{\gamma_j}  W_{j}(x',t',k) \nonumber \\
& = I +  \int_{\gamma_j} e^{(x-x')\hat{\mathcal{L}}(k) + (t-t')\hat{\mathcal{Z}}(k)} \big[\big((\mathsf{U} dx' + \mathsf{V} dt') \mu_{j}\big) (x',t',k)\big],  \label{mujdef}
\end{align}
where $W_{j}$ is given by the right-hand side of the first equation in (\ref{Wdef}) with $\mu$ replaced by $\mu_j$. Similarly, we define three eigenfunctions $\{\mu_j^{A}\}_{j=1}^3$ of \eqref{Xlax adjoint} by 
\begin{align}
\mu_j^{A}(x,t,k) & = I + e^{-x\hat{\mathcal{L}}(k) - t\hat{\mathcal{Z}}(k)} \int_{\gamma_j}  W_{j}^{A}(x',t',k) \nonumber \\
& = I -  \int_{\gamma_j} e^{(x'-x)\hat{\mathcal{L}}(k) + (t'-t)\hat{\mathcal{Z}}(k)} \big[\big((\mathsf{U}^{T} dx' + \mathsf{V}^{T} dt') \mu_{j}^{A}\big) (x',t',k)\big], \label{mujdef adjoint}
\end{align}
where $W_{j}^{A}$ is given by the right-hand side of the second equation in (\ref{Wdef}) with $\mu^{A}$ replaced by $\mu_j^{A}$. 

\begin{figure}
\begin{center}
\begin{tikzpicture}[master]
\draw[line width=0.15 mm, dashed,fill=gray!25] (0,0)--(3.5,0)--(3.5,1.2)--(0,1.2)--(0,0) -- cycle;
\draw[line width=0.2 mm,-<-=0,->-=1] (0,1.8)--(0,0)--(3.5,0);
\draw[line width=0.25 mm, white] (3.5,0)--(3.5,1.8);
\draw[fill] (1.35,0.6) circle (0.04); 
\draw[fill] (0,1.2) circle (0.04); 
\draw[line width=0.35 mm,->-=0.23,->-=0.75] (0,1.2)--(0,0.6)--(1.35,0.6);
\node at (1.85,0.6) {\small $(x,t)$};
\node at (-0.25,1.2) {\small $T$};
\node at (1.75,-0.5) {$\gamma_{1}$};
\end{tikzpicture}
 \hspace{0.6cm}
\begin{tikzpicture}[slave]
\draw[line width=0.15 mm, dashed,fill=gray!25] (0,0)--(3.5,0)--(3.5,1.2)--(0,1.2)--(0,0) -- cycle;
\draw[line width=0.2 mm,-<-=0,->-=1] (0,1.8)--(0,0)--(3.5,0);
\draw[line width=0.25 mm, white] (3.5,0)--(3.5,1.8);
\draw[fill] (1.35,0.6) circle (0.04); 
\draw[fill] (0,0) circle (0.04); 
\draw[line width=0.35 mm,->-=0.23,->-=0.75] (0,0)--(0,0.6)--(1.35,0.6);
\node at (1.85,0.6) {\small $(x,t)$};
\node at (-0.25,1.2) {\small $T$};
\node at (1.75,-0.5) {$\gamma_{2}$};
\end{tikzpicture} \hspace{0.6cm}
\begin{tikzpicture}[slave]
\draw[line width=0.15 mm, dashed,fill=gray!25] (0,0)--(3.5,0)--(3.5,1.2)--(0,1.2)--(0,0) -- cycle;
\draw[line width=0.2 mm,-<-=0,->-=1] (0,1.8)--(0,0)--(3.5,0);
\draw[line width=0.25 mm, white] (3.5,0)--(3.5,1.8);
\draw[fill] (1.35,0.6) circle (0.04); 
\draw[fill] (3.5,0.6) circle (0.04); 
\draw[line width=0.35 mm,->-=0.6] (3.5,0.6)--(1.35,0.6);
\node at (0.85,0.6) {\small $(x,t)$};
\node at (-0.25,1.2) {\small $T$};
\node at (1.75,-0.5) {$\gamma_{3}$};
\end{tikzpicture}
\begin{figuretext}\label{mucontours.pdf}
       The contours $\gamma_1$, $\gamma_2$, and $\gamma_3$ in the $(x, t)$-plane.
   \end{figuretext}
\end{center}
\end{figure}

The third column of the matrix equation (\ref{mujdef}) involves the exponentials
$$e^{(l_1 - l_3)(x - x') + (z_1 - z_3)(t - t')}, \qquad e^{(l_2 - l_3)(x - x') + (z_2 - z_3)(t - t')}.$$ 
Using the inequalities in (\ref{contourinequalities}) it follows that for fixed $x\in \R_{+}$, $t\in [0,T]$ and $(x',t') \in \gamma_{j}$, these exponentials are bounded for $k \in \hat{\mathcal{S}}_{j}$, $j=1,2,3$, respectively. These boundedness properties, combined with a standard analysis of the Volterra integral equations \eqref{mujdef}, imply that the third columns of $\mu_1,\mu_2$ and $\mu_3$ are bounded for $k$ in $\hat{\mathcal{S}}_{1},\hat{\mathcal{S}}_{2}$ and $\hat{\mathcal{S}}_{3}$, respectively, except for $k$ in small neighborhoods of $\mathcal{Q}$ because $\mathsf{U}$ and $\mathsf{V}$ have simple poles at each of the six points $\kappa_{j}$, $j=1,\ldots,6$.

Let $j\in \{1,2,3\}$. A similar argument shows that the $j$-th columns of $\mu_1,\mu_2$ and $\mu_3$ are bounded for $k$ in $\omega^{3-j}\hat{\mathcal{S}}_{1},\omega^{3-j}\hat{\mathcal{S}}_{2}$ and $\omega^{3-j}\hat{\mathcal{S}}_{3}$, respectively, except for $k$ in small neighborhoods of $\mathcal{Q}$.

Note also that $\mu_{1}$ and $\mu_{2}$ are well-defined for all $k\in \C\setminus \hat{\mathcal{Q}}$, because $\gamma_{1}$ and $\gamma_{2}$ are bounded.

The above properties of $\mu_{1},\mu_{2},\mu_{3}$, along with other properties, are summarized in Proposition \ref{XYprop}. The spectral analysis for the initial value problem of \eqref{badboussinesq} was carried out in \cite[Section 3]{CLmain}. The statement of Proposition \ref{XYprop} is similar to \cite[Proposition 3.2]{CLmain}; the main difference is that our initial-boundary value problem require a detailed study of the three functions $\mu_{j}$, $j=1,2,3$, whereas in \cite{CLmain} only two functions were needed (which were denoted $X$ and $Y$). The function $X$ in \cite{CLmain} corresponds to $\mu_{3}$ here. However, neither $\mu_{1}$ nor $\mu_{2}$ corresponds to $Y$ in \cite{CLmain}, so these functions require a new (albeit standard) analysis. The proof is omitted.


\begin{proposition}[Basic properties of $\mu_{1},\mu_{2},\mu_{3}$]\label{XYprop}
The equations (\ref{mujdef}) uniquely define three $3 \times 3$-matrix valued solutions $\mu_{1},\mu_{2}, \mu_{3}$ of (\ref{Xlax}) with the following properties:
\begin{enumerate}[$(a)$]
\item The functions $\mu_{j}$ have the following domains of definition:
\begin{subequations}\label{lol4main}
\begin{align}
& \mu_1(x,t,k)\; \text{is defined for} \; x\in \R_{+}, \; t \in [0,T], \; k\in \C\setminus \hat{\mathcal{Q}}, \\
& \mu_2(x,t,k) \; \text{is defined for} \; x\in \R_{+}, \; t \in [0,T], \; k \in \C\setminus \hat{\mathcal{Q}}, \\
& \mu_3(x,t,k)\; \text{is defined for} \; x\in \R_{+}, \; t \in [0,T], \; k \in (\omega^{2} \hat{\mathcal{S}}_{3}, \omega \hat{\mathcal{S}}_{3}, \hat{\mathcal{S}}_{3})\setminus \hat{\mathcal{Q}}.
\end{align}
\end{subequations}

\smallskip \noindent For each $k$ as in \eqref{lol4main} and $t\in [0,T]$, $\mu_{j}(\cdot,t, k)$ is smooth and satisfies the $x$-part of (\ref{Xlax}).

\smallskip \noindent For each $k$ as in \eqref{lol4main} and $x \geq 0$, $\mu_{j}(x,\cdot, k)$ is smooth and satisfies the $t$-part of (\ref{Xlax}).

\item For each $x \in \R_{+}$ and $t\in [0,T]$, 
\begin{subequations}
\begin{align*}
& \mu_1(x,t,k) \text{ is analytic for } k \in \C\setminus \hat{\mathcal{Q}}, \\
& \mu_2(x,t,k) \text{ is analytic for } k \in \C\setminus\hat{\mathcal{Q}}, \\
& \mu_3(x,t,k) \text{ is continuous for } k \in (\omega^{2} \hat{\mathcal{S}}_{3}, \omega \hat{\mathcal{S}}_{3}, \hat{\mathcal{S}}_{3})\setminus\hat{\mathcal{Q}} \mbox{ and analytic for } k \in (\omega^{2} \mathcal{S}_{3}, \omega \mathcal{S}_{3}, \mathcal{S}_{3})\setminus\hat{\mathcal{Q}}.
\end{align*}
\end{subequations}

\item For each $x \in \R_{+}$, $t\in [0,T]$, and $j = 0,1,\dots$,
\begin{align*}
& \tfrac{\partial^j }{\partial k^j}\mu_3(x,t,\cdot) \text{ is well defined for } k \in (\omega^2 \hat{\mathcal{S}}_{3}, \omega \hat{\mathcal{S}}_{3}, \hat{\mathcal{S}}_{3})\setminus \hat{\mathcal{Q}}.
\end{align*}

\item For each $n \geq 1$ and $\epsilon > 0$, there exists $C>0$ such that the following estimates hold for $x \in \R_{+}$, $t\in [0,T]$, $j=1,2,3$ and $ \ell = 0, 1, \dots, n$:
\begin{align}\label{region of boundedness of muj}
& \big|\tfrac{\partial^\ell}{\partial k^\ell}\mu_{j}(x,t,k) \big| \leq
C, & & k \in (\omega^2 \hat{\mathcal{S}}_{j}, \omega \hat{\mathcal{S}}_{j}, \hat{\mathcal{S}}_{j}), \ \dist(k,\hat{\mathcal{Q}}) > \epsilon.
\end{align} 
 
\item $\mu_{1},\mu_{2},\mu_{3}$ obey the following symmetries for each $x \in \R_{+}$, $t\in[0,T]$ and $k$ as in \eqref{lol4main}:
\begin{align}\label{XYsymm}
&  \mu_{j}(x,t, k) = \mathcal{A} \mu_{j}(x,t,\omega k)\mathcal{A}^{-1} = \mathcal{B} \mu_{j}(x,t,k^{-1})\mathcal{B}.
\end{align}

\item For each $x \in \R_{+}$, $t\in [0,T]$ and $k \in \C \setminus \hat{\mathcal{Q}}$, $\det \mu_{1}(x,t,k) = \det \mu_{2}(x,t,k) = 1$. 

\medskip \noindent Furthermore, if $u_{0}, v_{0}$ have compact support, then, for each  $x \in \R_{+}$ and $t\in [0,T]$, $\mu_{3}(x,t, k)$ is defined and analytic for $k \in \C \setminus \hat{\mathcal{Q}}$ and $\det \mu_{3}(x,t,k) = 1$.
\end{enumerate}
\end{proposition}

We now turn to the properties of $\mu_{1}^{A}, \mu_{2}^{A}, \mu_{3}^{A}$.

\begin{proposition}[Basic properties of $\mu_{1}^{A},\mu_{2}^{A},\mu_{3}^{A}$]\label{XYprop adjoint}
The equations (\ref{mujdef adjoint}) uniquely define three $3 \times 3$-matrix valued solutions $\mu_{1}^{A},\mu_{2}^{A}, \mu_{3}^{A}$ of (\ref{Xlax adjoint}) with the following properties:
\begin{enumerate}[$(a)$]
\item The functions $\mu_{j}^{A}$ have the following domains of definition:
\begin{subequations}\label{lol4main adjoint}
\begin{align}
& \mu_1^{A}(x,t,k)\; \text{is defined for} \; x\in \R_{+}, \; t \in [0,T], \; k\in \C\setminus \hat{\mathcal{Q}}, \\
& \mu_2^{A}(x,t,k) \; \text{is defined for} \; x\in \R_{+}, \; t \in [0,T], \; k \in \C\setminus \hat{\mathcal{Q}}, \\
& \mu_3^{A}(x,t,k)\; \text{is defined for} \; x\in \R_{+}, \; t \in [0,T], \; k \in (\omega^{2} \hat{\mathcal{S}}_{3}^{*}, \omega \hat{\mathcal{S}}_{3}^{*}, \hat{\mathcal{S}}_{3}^{*})\setminus \hat{\mathcal{Q}}.
\end{align}
\end{subequations}

\smallskip \noindent For each $k$ as in \eqref{lol4main} and $t\in [0,T]$, $\mu_{j}^{A}(\cdot,t, k)$ is smooth and satisfies the $x$-part of (\ref{Xlax adjoint}).

\smallskip \noindent For each $k$ as in \eqref{lol4main} and $x \geq 0$, $\mu_{j}^{A}(x,\cdot, k)$ is smooth and satisfies the $t$-part of (\ref{Xlax adjoint}).

\item For each $x \in \R_{+}$ and $t\in [0,T]$, 
\begin{subequations}
\begin{align*}
& \mu_1^{A}(x,t,k) \text{ is analytic for } k \in \C\setminus \hat{\mathcal{Q}}, \\
& \mu_2^{A}(x,t,k) \text{ is analytic for } k \in \C\setminus\hat{\mathcal{Q}}, \\
& \mu_3^{A}(x,t,k) \text{ is continuous for } k \in (\omega^{2} \hat{\mathcal{S}}_{3}^{*}, \omega \hat{\mathcal{S}}_{3}^{*}, \hat{\mathcal{S}}_{3}^{*})\setminus\hat{\mathcal{Q}} \mbox{ and analytic for } k \in (\omega^{2} \mathcal{S}_{3}^{*}, \omega \mathcal{S}_{3}^{*}, \mathcal{S}_{3}^{*})\setminus\hat{\mathcal{Q}}.
\end{align*}
\end{subequations}

\item For each $x \in \R_{+}$, $t\in [0,T]$, and $j = 0,1,\dots$,
\begin{align*}
& \tfrac{\partial^j }{\partial k^j}\mu_3^{A}(x,t,\cdot) \text{ is well defined for } k \in (\omega^2 \hat{\mathcal{S}}_{3}^{*}, \omega \hat{\mathcal{S}}_{3}^{*}, \hat{\mathcal{S}}_{3}^{*})\setminus \hat{\mathcal{Q}}.
\end{align*}

\item For each $n \geq 1$ and $\epsilon > 0$, there exists $C>0$ such that the following estimates hold for $x \in \R_{+}$, $t\in [0,T]$, $j=1,2,3$ and $ \ell = 0, 1, \dots, n$:
\begin{align}\label{region of boundedness of muj adjoint}
& \big|\tfrac{\partial^\ell}{\partial k^\ell}\mu_{j}^{A}(x,t,k) \big| \leq
C, & & k \in (\omega^2 \hat{\mathcal{S}}_{j}^{*}, \omega \hat{\mathcal{S}}_{j}^{*}, \hat{\mathcal{S}}_{j}^{*}), \ \dist(k,\hat{\mathcal{Q}}) > \epsilon.
\end{align} 
\item $\mu_{1}^{A},\mu_{2}^{A},\mu_{3}^{A}$ obey the following symmetries for each $x \in \R_{+}$, $t\in[0,T]$, and $k$ as in \eqref{lol4main adjoint}:
\begin{align}\label{XYsymm adjoint}
&  \mu_{j}^{A}(x,t, k) = \mathcal{A} \mu_{j}^{A}(x,t,\omega k)\mathcal{A}^{-1} = \mathcal{B} \mu_{j}^{A}(x,t,k^{-1})\mathcal{B}.
\end{align}

\item For each $x \in \R_{+}$, $t\in [0,T]$ and $k \in \C \setminus \hat{\mathcal{Q}}$, $\mu_{1}^{A}(x,t,k) = (\mu_{1}(x,t,k)^{-1})^{T}$ and $\mu_{2}^{A}(x,t,k) = (\mu_{2}(x,t,k)^{-1})^{T}$. 

\medskip \noindent Furthermore, if $u_{0}, v_{0}$ have compact support, then, for each  $x \in \R_{+}$, $t\in [0,T]$ and $k \in \C \setminus \hat{\mathcal{Q}}$, $\mu_{3}^{A}(x,t,k) = (\mu_{3}(x,t,k)^{-1})^{T}$.
\end{enumerate}
\end{proposition}

A similar analysis as in \cite[Proposition 3.3]{CLmain} shows that, as $k\to \infty$ within the regions of boundedness \eqref{region of boundedness of muj},
\begin{align}\label{mujlargekexpansion}
\mu_{j}(x,t,k) = I + \frac{\mu_{j}^{(1)}(x,t)}{k} + \frac{\mu_{j}^{(2)}(x,t)}{k^2} + \cdots.
\end{align}
for some matrices $\{\mu_{j}^{(l)}(x,t)\}_{l=1}^{+\infty}$, $j=1,2,3$, that are bounded functions of $x\in \R_{+}$ and $t\in [0,T]$. Moreover, the asymptotic formula \eqref{mujlargekexpansion} can be differentiated any fixed number of times with respect to $k$. 
For any integer $n \geq 1$, $\mathsf{U}$ and $\mathsf{V}$ have expansions as $k\to \infty$ of the form
\begin{align*}
& \mathsf{U}(x,t,k) = \sum_{j=1}^{n} \frac{\mathsf{U}^{(j)}(x,t)}{k^{j}} + O(k^{-n-1}), \qquad k \to \infty, \\
& \mathsf{V}(x,t,k) = \sum_{j=0}^{n} \frac{\mathsf{V}^{(j)}(x,t)}{k^{j}} + O(k^{-n-1}), \qquad k \to \infty,
\end{align*}
for some matrices $\{\mathsf{U}^{(j)}(x,t)\}_{j=1}^{+\infty}$ and $\{\mathsf{V}^{(j)}(x,t)\}_{j=0}^{+\infty}$, where
\begin{align*}
& \mathsf{U}^{(1)} = \frac{i \, u}{\sqrt{3}} \begin{pmatrix}
\omega^{2} & 1 & \omega \\
1 & \omega & \omega^{2} \\
\omega & \omega^{2} & 1
\end{pmatrix}, \qquad \mathsf{U}^{(2)} = \left( \frac{i \, v}{\sqrt{3}} + u_{x} \right) \begin{pmatrix}
\omega & \omega & \omega \\
\omega^{2} & \omega^{2} & \omega^{2} \\
1 & 1 & 1
\end{pmatrix}, \\
& \mathsf{V}^{(0)} = \frac{-i\,u}{2\sqrt{3}}\begin{pmatrix}
0 & 1 & 1 \\ 1 & 0 & 1 \\ 1 & 1 & 0
\end{pmatrix}, \\
& \mathsf{V}^{(1)} = \frac{i\, v}{\sqrt{3}}\begin{pmatrix}
\omega^{2} & - \frac{\omega}{2} & - \frac{1}{2} \\[0.05cm]
- \frac{\omega^{2}}{2} & \omega & - \frac{1}{2} \\[0.05cm]
- \frac{\omega^{2}}{2} & -\frac{\omega}{2} & 1
\end{pmatrix} + \frac{u_{x}(\omega^{2}-\omega)}{6}\begin{pmatrix}
0 & \omega & -1 \\
-\omega^{2} & 0 & 1 \\
\omega^{2} & -\omega & 0
\end{pmatrix}.
\end{align*}
Moreover, $\mathcal{L}$ can be written as
\begin{align*}
& \mathcal{L} = \frac{\mathcal{L}^{(-1)}}{k^{-1}} + \frac{\mathcal{L}^{(1)}}{k}, \quad \mathcal{L}^{(-1)} := \frac{i}{2\sqrt{3}}\begin{pmatrix}
\omega & 0 & 0 \\
0 & \omega^{2} & 0 \\
0 & 0 & 1
\end{pmatrix}, \quad \mathcal{L}^{(1)} := \frac{i}{2\sqrt{3}}\begin{pmatrix}
\omega^{2} & 0 & 0 \\
0 & \omega & 0 \\
0 & 0 & 1
\end{pmatrix}.
\end{align*}
Substituting (\ref{mujlargekexpansion}) into (\ref{Xlax}) and identifying terms of the same order yields the relations
\begin{align}\label{xrecursive}
\begin{cases}
[\mathcal{L}^{(-1)}, \mu_{j}^{(\ell)}] = -[\mathcal{L}^{(1)}, \mu_{j}^{(\ell-2)}] + \big(\partial_x \mu_{j}^{(\ell-1)} - \sum_{m=1}^{\ell-1} \mathsf{U}^{(m)}\mu_{j}^{(\ell-1-m)} \big)^{(o)},
	\\
(\partial_x \mu_{j}^{(\ell)})^{(d)} = \big( \sum_{m=1}^{\ell} \mathsf{U}^{(m)}\mu_{j}^{(\ell-m)} \big)^{(d)}, \\
(\partial_t \mu_{j}^{(\ell)})^{(d)} = \big( \sum_{m=0}^{\ell} \mathsf{V}^{(m)}\mu_{j}^{(\ell-m)} \big)^{(d)},
\end{cases}
\end{align}
where $A^{(d)}$ and $A^{(o)}$ denote the diagonal and off-diagonal parts of a $3 \times 3$ matrix $A$, respectively. The coefficients $\{\mu_{j}^{(\ell)}(x,t)\}$ are uniquely determined from (\ref{xrecursive}), the initial values $\mu_{j}^{(-1)} = 0, \mu_{j}^{(0)} = I$, and the normalizations $\mu_j^{(\ell)}(x_{j},t_{j}) = 0$. The first two coefficients are given in the following proposition.
\begin{proposition}\label{prop:first two coeff at infty}
The coefficients $\{\mu_{j}^{(\ell)}\}_{j=1,2,3}^{\ell=1,2}$ are given by 
\begin{align*}
& \mu_{j}^{(1)}(x,t) =  \int_{(x_{j},t_{j})}^{(x,t)} \Delta^{(1)}(x',t') \begin{pmatrix} \omega^2 & 0 & 0 \\ 
0 & \omega & 0 \\ 
0 & 0 & 1
\end{pmatrix}, \\
& \mu_{j}^{(2)}(x,t) = \frac{-2u(x,t)}{1-\omega}\begin{pmatrix}
0 & \omega^{2} & -\omega \\ - \omega^{2} & 0 & 1 \\ \omega & -1 & 0
\end{pmatrix} + \int_{(x_{j},t_{j})}^{(x,t)} \Delta_{j}^{(2)}(x',t') \begin{pmatrix} \omega & 0 & 0 \\ 
0 & \omega^{2} & 0 \\ 
0 & 0 & 1
\end{pmatrix},
\end{align*}
where $\Delta^{(1)}(x,t) = \frac{i}{\sqrt{3}}\big(u(x,t)dx+v(x,t)dt\big)$ and 
\begin{align*}
\Delta_{j}^{(2)}(x,t) = \frac{\sqrt{3}u_{x}+iv+iu(\mu_{j}^{(1)})_{33}}{\sqrt{3}}dx + \frac{\sqrt{3}v_{x} + i(u+u^{2}+u_{xx}+v \, (\mu_{j}^{(1)})_{33})}{\sqrt{3}}dt.
\end{align*}
\end{proposition}
This leads to the following formulas for recovering $u$ and $v$:
\begin{align}\label{recoveru}
& u = -i \sqrt{3} \frac{d}{dx}(\mu_{j}^{(1)})_{33}, & & v = -i \sqrt{3} \frac{d}{dt}(\mu_{j}^{(1)})_{33}.
\end{align}
Using the symmetries $\mu_{j}(x,t, k) = \mathcal{B} \mu_{j}(x,t,k^{-1})\mathcal{B}$, we can also deduce from \eqref{mujlargekexpansion} the asymptotics for $\mu_{j}(x,t,k)$ as $k \to 0$, $k \in (\omega^2 \hat{\mathcal{S}}_{j}, \omega \hat{\mathcal{S}}_{j}, \hat{\mathcal{S}}_{j})$. In particular, for each $x \in \R_{+}$, $t\in [0,T]$ and $j\in \{1,2,3\}$, the limit of $\mu_{j}(x,t,k)$ as $k\to 0$, $k \in (\omega^2 \hat{\mathcal{S}}_{j}, \omega \hat{\mathcal{S}}_{j}, \hat{\mathcal{S}}_{j})$, is $I$.

\medskip Similarly, as $k\to \infty$ within the regions of boundedness \eqref{region of boundedness of muj adjoint}, we find
\begin{align}\label{mujAlargekexpansion}
\mu_{j}^{A}(x,t,k) = I + \frac{\mu_{j}^{A(1)}(x,t)}{k} + \frac{\mu_{j}^{A(2)}(x,t)}{k^2} + \cdots.
\end{align}
for some matrices $\{\mu_{j}^{A(l)}(x,t)\}_{l=1}^{+\infty}$, $j=1,2,3$, such that
\begin{align*}
\bigg(I + \sum_{j=1}^{p}\frac{\mu_{j}^{(1)}(x,t)}{k^{j}}\bigg)^{T} \bigg( I + \sum_{j=1}^{p}\frac{\mu_{j}^{A(1)}(x,t)}{k^{j}} \bigg) = I + O(k^{-p-1}), \qquad \mbox{as } k \to \infty
\end{align*}
holds for all integers $p$.

\medskip Since $\mathsf{U}$ and $\mathsf{V}$ have simple poles at each of the six points $\kappa_{j}$, $j=1,\ldots,6$, the functions $\mu_{1},\mu_{2},\mu_{3}$ are singular near these points. The next proposition establishes the behavior of $\{\mu_{j}(x,t,k)\}_{j=1}^{3}$ as $k\to \pm 1$. The proof is omitted; we refer to  \cite[Proposition]{CLgoodboussinesq} for a proof of a similar (albeit different) statement. The behavior of $\{\mu_{j}(x,t,k)\}_{j=1}^{3}$ as $k\to \kappa_{j}$, $j=2,3,5,6$ can be obtained by combining Proposition \ref{XYat1prop} with the $\mathcal{A}$-symmetry in \eqref{XYsymm adjoint}.

\begin{proposition}[Asymptotics of $\mu_{j}$ as $k \to \pm 1$]\label{XYat1prop}
Let $p \geq 0$ be an integer. 
There are $3 \times 3$-matrix valued functions $C_j^{(l)}(x,t)$, $j = 1, 2, 3$, $l = -1, \dots, p$, with the following properties:
\begin{itemize}
\item For $x \geq 0$, $t\in [0,T]$ and $k$ within the domains of definition \eqref{lol4main} of $\mu_{j}$, we have
\begin{subequations}\label{XYat1}
\begin{align}
& \bigg|\frac{\partial^m}{\partial k^m}\big(\mu_{j} - I - \sum_{l=-1}^p C_j^{(l)}(x,t)(k-1)^l\big) \bigg| \leq
f_j(x,t)|k-1|^{p+1-m}, \qquad |k-1| \leq \frac{1}{2}, \label{Xat1} \\
& \bigg|\frac{\partial^m}{\partial k^m}\big(\mu_{j} - I - \sum_{l=-1}^p C_{3+j}^{(l)}(x,t)(k+1)^l\big) \bigg| \leq
f_j(x,t)|k+1|^{p+1-m}, \qquad |k+1| \leq \frac{1}{2}, \label{Xatm1}
\end{align}
\end{subequations}
where $m \geq 0$ is any integer, $f_j(x,t)$ are smooth positive functions of $x \geq 0$ and $t \in [0,T]$, and $f_{3}(x,t)$ has rapid decay as $x \to +\infty$.

\item For each $l \geq -1$ and $j\in \{1,2,3\}$, $C_j^{(l)}(x,t),C_{3+j}^{(l)}(x,t)$ are smooth functions of $x \geq 0$ and $t \in [0,T]$, and $C_3^{(l)}(x,t),C_6^{(l)}(x,t)$ have rapid decay as $x \to +\infty$. 

\item The leading coefficients have the form
\begin{align}
  & C_{j}^{(-1)}(x,t)
= \hspace{-0.1cm} \begin{pmatrix}
\alpha_{j}(x,t) & \hspace{-0.1cm}\alpha_{j}(x,t) & \hspace{-0.1cm}\beta_{j}(x,t) \\
-\alpha_{j}(x,t) & \hspace{-0.1cm}-\alpha_{j}(x,t) & \hspace{-0.1cm}-\beta_{j}(x,t) \\
0 & \hspace{-0.1cm}0 & \hspace{-0.1cm}0
\end{pmatrix}, \label{Cjpm1p} \\
& C_{j}^{(0)}(x,t) = - I + \begin{pmatrix}
\gamma_{j,3}(x,t) & \gamma_{j,4}(x,t) & \gamma_{j,5}(x,t) \\
\gamma_{j,4}(x,t)\mp \alpha_{j}(x,t) & \gamma_{j,3}(x,t)\mp \alpha_{j}(x,t) & \gamma_{j,5}(x,t)\mp \beta_{j}(x,t) \\
\gamma_{j,1}(x,t) & \gamma_{j,1}(x,t) & \gamma_{j,2}(x,t)
\end{pmatrix}, \label{Cjp0p}
\end{align}
where $\mp$ in \eqref{Cjp0p} is equal to $-1$ for $j=1,2,3$ and to $+1$ for $j=4,5,6$, $\alpha_j(x,t)$, $\beta_j(x,t)$, $\gamma_{j,\ell}(x,t)$, $j = 1,\ldots,6$, $\ell=1,\ldots,5$, are complex-valued functions of $x \geq 0$ and $t\in [0,T]$,  and for $j=3,6$ these functions have rapid decay as $x \to +\infty$.
\end{itemize}
\end{proposition}
We now consider the behavior of $\{\mu_{j}^{A}(x,t,k)\}_{j=1}^{3}$ as $k\to \pm 1$.
\begin{proposition}[Asymptotics of $\mu_{j}^{A}$ as $k \to \pm 1$]\label{XYat1prop adjoint}
Let $p \geq 0$ be an integer. 
There are $3 \times 3$-matrix valued functions $D_j^{(l)}(x,t)$, $j = 1, \ldots, 6$, $l = -1, \dots, p$, with the following properties:
\begin{itemize}
\item For $x \geq 0$, $t\in [0,T]$ and $k$ within the domains of definition \eqref{lol4main adjoint} of $\mu_{j}^{A}$, we have
\begin{subequations}\label{XYat1 adjoint}
\begin{align}
& \bigg|\frac{\partial^m}{\partial k^m}\big(\mu_{j}^{A} - I - \sum_{l=-1}^p D_j^{(l)}(x,t)(k-1)^l\big) \bigg| \leq
f_j(x,t)|k-1|^{p+1-m}, \qquad |k-1| \leq \frac{1}{2}, \label{XAat1} \\
& \bigg|\frac{\partial^m}{\partial k^m}\big(\mu_{j}^{A} - I - \sum_{l=-1}^p D_{3+j}^{(l)}(x,t)(k+1)^l\big) \bigg| \leq
f_j(x,t)|k+1|^{p+1-m}, \qquad |k+1| \leq \frac{1}{2}, \label{XAatm1}
\end{align}
\end{subequations}
where $m \geq 0$ is any integer, $f_j(x,t)$ are smooth positive functions of $x \geq 0$ and $t \in [0,T]$, and $f_{3}(x,t)$ has rapid decay as $x \to +\infty$.

\item For each $l \geq -1$ and $j\in \{1,2,3\}$, $D_j^{(l)}(x,t),D_{3+j}^{(l)}(x,t)$ are smooth functions of $x \geq 0$ and $t \in [0,T]$, and $D_3^{(l)}(x,t),D_6^{(l)}(x,t)$ have rapid decay as $x \to +\infty$. 

\item The leading coefficients have the form
\begin{align}
D_{j}^{(-1)}(x,t)
= & \;
\begin{pmatrix}
\tilde{\alpha}_{j}(x,t) & -\tilde{\alpha}_{j}(x,t) & 0 \\
\tilde{\alpha}_{j}(x,t) & -\tilde{\alpha}_{j}(x,t) & 0 \\
\tilde{\beta}_{j}(x,t) & -\tilde{\beta}_{j}(x,t) & 0
\end{pmatrix}, \label{Djpm1p} \\
D_{j}^{(0)}(x,t) = & \; -I+ \begin{pmatrix}
\tilde{\gamma}_{j,3}(x,t) & \tilde{\gamma}_{j,4}(x,t)\mp \tilde{\alpha}_{j}(x,t) & \tilde{\gamma}_{j,1}(x,t) \\
\tilde{\gamma}_{j,4}(x,t) & \tilde{\gamma}_{j,3}(x,t)\mp \tilde{\alpha}_{j}(x,t) & \tilde{\gamma}_{j,1}(x,t) \\
\tilde{\gamma}_{j,5}(x,t) & \tilde{\gamma}_{j,5}(x,t)\mp \tilde{\beta}_{j}(x,t) & \tilde{\gamma}_{j,2}(x,t)
\end{pmatrix},  \label{Djp0p}
\end{align}
where $\mp$ in \eqref{Djp0p} is equal to $-1$ for $j=1,2,3$ and to $+1$ for $j=4,5,6$, $\tilde{\alpha}_j(x,t)$, $\tilde{\beta}_j(x,t)$, $\tilde{\gamma}_{j,\ell}(x,t)$, $j = 1,\ldots,6$, $\ell=1,\ldots,5$, are complex-valued functions of $x \geq 0$ and $t\in [0,T]$, and for $j=3,6$ these functions have rapid decay as $x \to +\infty$.
\end{itemize}
\end{proposition}

Recall from Proposition \ref{XYprop} (d) that $\mu_{1}(x,t,k)$ is bounded for $k \in (\omega^2 \hat{\mathcal{S}}_{1}, \omega \hat{\mathcal{S}}_{1}, \hat{\mathcal{S}}_{1}), \ \dist(k,\hat{\mathcal{Q}}) > \epsilon$, where 
\begin{align*}
& \mathcal{S}_{1} = \{\re  l_1 < \re  l_3\} \cap \{\re  l_2 < \re  l_3\} \cap \{\re  z_3 < \re  z_1\} \cap \{\re  z_3 < \re  z_2\}.
\end{align*}
For $x=0$, the contour $\gamma_{1}$ consists of only a vertical segment in the $(x,t)$-plane going from $(0,T)$ to $(0,t)$. This implies that $\mu_{1}(0,t,k)$ is bounded on a larger domain than $(\omega^2 \hat{\mathcal{S}}_{1}, \omega \hat{\mathcal{S}}_{1}, \hat{\mathcal{S}}_{1}), \ \dist(k,\hat{\mathcal{Q}}) > \epsilon$; more precisely, we have that
\begin{align*}
& \mu_{1}(0,t,k) \; \text{is bounded for} \; k \in (\omega^{2}\hat{\mathsf{S}}_{1}, \omega\hat{\mathsf{S}}_{1}, 
\hat{\mathsf{S}}_{1}), \; \; \dist(k,\hat{\mathcal{Q}}) > \epsilon, 
\end{align*}
where $\hat{\mathsf{S}}_{1} = \partial \D \cup \bar{\mathsf{S}}_{1}$ and
\begin{align*}
\mathsf{S}_{1} = \{\re  z_3 < \re  z_1\} \cap \{\re  z_3 < \re  z_2\}.
\end{align*}
Similarly, we find that 
\begin{align*}
& \mu_{2}^{A}(0,t,k) \; \text{is bounded for} \; k \in (\omega^{2}\hat{\mathsf{S}}_{1}, \omega\hat{\mathsf{S}}_{1}, 
\hat{\mathsf{S}}_{1}), \; \; \dist(k,\hat{\mathcal{Q}}) > \epsilon, \\
& \mu_{1}^{A}(0,t,k) \mbox{ and } \mu_{2}(0,t,k) \; \text{are bounded for} \; k \in (\omega^{2}\hat{\mathsf{S}}_{2}, \omega\hat{\mathsf{S}}_{2}, 
\hat{\mathsf{S}}_{2}), \; \; \dist(k,\hat{\mathcal{Q}}) > \epsilon, 
\end{align*}
where $\hat{\mathsf{S}}_{2} = \partial \D \cup \bar{\mathsf{S}}_{2}$ and
\begin{align*}
\mathsf{S}_{2} = \{\re  z_1 < \re  z_3\} \cap \{\re  z_2 < \re  z_3\}.
\end{align*}
Note that $\mathsf{S}_{2}^{*}=\mathsf{S}_{1}$. For convenience, we also define $\mathsf{S}_{3}=\mathcal{S}_{3}$ and $\hat{\mathsf{S}}_{3}=\partial \D \cup \mathcal{S}_{3}$, see Figure \ref{fig: S1S2S3 sf}. 

It follows from the above considerations that the statements of Proposition \ref{XYprop} (d), Proposition \ref{XYprop adjoint} (d), \eqref{mujlargekexpansion} and \eqref{mujAlargekexpansion} can be extended as follows for $x=0$:
\begin{proposition}\label{prop: basic properties for t=0}
The functions $\mu_{1},\mu_{2}, \mu_{3},\mu_{1}^{A},\mu_{2}^{A}, \mu_{3}^{A}$ have the following properties:
\begin{enumerate}[$(a)$]
\item For each $n \geq 1$ and $\epsilon > 0$, there exists $C>0$ such that the following estimates hold for $t\in [0,T]$, $j=1,2,3$ and $ \ell = 0, 1, \dots, n$:
\begin{subequations}\label{region of boundedness}
\begin{align}
& \big|\tfrac{\partial^\ell}{\partial k^\ell}\mu_{j}(0,t,k) \big| \leq
C, & & k \in (\omega^2 \hat{\mathsf{S}}_{j}, \omega \hat{\mathsf{S}}_{j}, \hat{\mathsf{S}}_{j}), \ \dist(k,\hat{\mathcal{Q}}) > \epsilon, \label{region of boundedness of muj t=0} \\
& \big|\tfrac{\partial^\ell}{\partial k^\ell}\mu_{j}^{A}(0,t,k) \big| \leq
C, & & k \in (\omega^2 \hat{\mathsf{S}}_{j}^{*}, \omega \hat{\mathsf{S}}_{j}^{*}, \hat{\mathsf{S}}_{j}^{*}), \ \dist(k,\hat{\mathcal{Q}}) > \epsilon. \label{region of boundedness of mujA t=0}
\end{align}
\end{subequations}

\item For any integer $p \geq 0$ and any $j\in \{1,2,3\}$, as $k\to \infty$ within the regions of boundedness \eqref{region of boundedness}, we have
\begin{subequations}\label{mujlargekexpansion t=0}
\begin{align}
& \mu_{j}(0,t,k) = I + \sum_{l=1}^{p}\frac{\mu_{j}^{(l)}(0,t)}{k^{l}} + O(k^{-p-1}), & & \mbox{as } k\to \infty, \; k \in (\omega^2 \hat{\mathsf{S}}_{j}, \omega \hat{\mathsf{S}}_{j}, \hat{\mathsf{S}}_{j}), \label{mujlargekexpansion t=0 muj} \\
& \mu_{j}^{A}(0,t,k) = I + \sum_{l=1}^{p}\frac{\mu_{j}^{A(l)}(0,t)}{k^{l}} + O(k^{-p-1}), & & \mbox{as } k\to \infty, \; k \in (\omega^2 \hat{\mathsf{S}}_{j}^{*}, \omega \hat{\mathsf{S}}_{j}^{*}, \hat{\mathsf{S}}_{j}^{*}), \label{mujlargekexpansion t=0 muAj}
\end{align}
\end{subequations}
where the error terms are uniform for $t\in [0,T]$. Moreover, \eqref{mujlargekexpansion t=0} can be differentiated any fixed number of times with respect to $k$.
\end{enumerate}
\end{proposition}

\subsection{The spectral functions $s(k)$ and $S(k)$}
Recall from Proposition \ref{XYprop} (b) that all columns of $\mu_{1},\mu_{2}$ are analytic for $k \in \C\setminus \hat{\mathcal{Q}}$. Suppose temporarily that the initial data $u_0, v_0$ have compact supports. Then all columns of $\mu_{3}$ are also analytic for $k \in \C\setminus \hat{\mathcal{Q}}$ (see Proposition \ref{XYprop} (f)), and since each of the three matrix valued functions $\psi_{j}(x,t,k):=\mu_{j}(x,t,k)e^{x \mathcal{L}(k)+t \mathcal{Z}(k)}$, $j = 1,2,3$, satisfies the system
\begin{align*}
\begin{cases}
\psi_{x} = L \psi, \\
\psi_{t} = Z \psi,
\end{cases}
\end{align*}
they are related by some $(x,t)$-independent matrices, which we denote by $S(k)$ and $s(k)$, so that 
\begin{align}\label{mu3mu2mu1sS}
& \mu_{1}(x,t,k) = \mu_{2}(x,t,k)e^{x\hat{\mathcal{L}}(k)+t\hat{\mathcal{Z}}(k)}S(k), & & \mu_{3}(x,t,k) = \mu_{2}(x,t,k)e^{x\hat{\mathcal{L}}(k)+t\hat{\mathcal{Z}}(k)}s(k).
\end{align}
Evaluating both equations at $(x,t)=(0,0)$ and recalling that $\mu_{2}(0,0,k)=I$, we get
\begin{align}\label{def of S and s}
S(k) = \mu_1(0,0,k), \qquad  s(k) = \mu_3(0,0,k).
\end{align}

Suppose now that the initial data are not necessarily compactly supported. Then the second equation in \eqref{mu3mu2mu1sS} does not necessarily hold for all $k \in \C\setminus \hat{\mathcal{Q}}$, but we nevertheless define $s$ as in \eqref{def of S and s}. 
Similarly, we define $S^A$ and $s^A$ by
\begin{align}\label{SAsAdef}
& S^{A}(k) = \mu_{1}^{A}(0,0,k), \qquad s^{A}(k) = \mu_{3}^{A}(0,0,k).
\end{align}

\begin{figure}
\begin{center}
\scalebox{0.5}{\hspace{-0.15cm}\begin{tikzpicture}[scale=1.1, master]
\node at (0,0) {};

\draw[line width=0.025 mm, dashed,fill=gray!30] (0,0) --  (-105:4.5) arc(-105:-165:4.5) -- cycle;
\draw[line width=0.25 mm, dashed,fill=white] (0,0) --  (-105:3) arc(-105:-165:3) -- cycle;
\node at (-135:3.8) {$\mathsf{S}_{1}$ };
\draw[line width=0.025 mm, dashed,fill=gray!30] (0,0) --  (15:4.5) arc(15:75:4.5) -- cycle;
\draw[line width=0.25 mm, dashed,fill=white] (0,0) --  (15:3) arc(15:75:3) -- cycle;
\node at (45:3.8) {$\mathsf{S}_{1}$ };

\draw[line width=0.025 mm, dashed,fill=gray!30] (0,0) --  (-45:4.5) arc(-45:15:4.5) -- cycle;
\draw[line width=0.25 mm, dashed,fill=white] (0,0) --  (-45:3) arc(-45:15:3) -- cycle;
\node at (-15:3.8) {$\omega\mathsf{S}_{1}$ };
\draw[line width=0.025 mm, dashed,fill=gray!30] (0,0) --  (135:4.5) arc(135:195:4.5) -- cycle;
\draw[line width=0.25 mm, dashed,fill=white] (0,0) --  (135:3) arc(135:195:3) -- cycle;
\node at (165:3.8) {$\omega\mathsf{S}_{1}$ };

\draw[line width=0.025 mm, dashed,fill=gray!30] (0,0) --  (75:4.5) arc(75:135:4.5) -- cycle;
\draw[line width=0.25 mm, dashed,fill=white] (0,0) --  (75:3) arc(75:135:3) -- cycle;
\node at (105:3.8) {$\omega^{2}\mathsf{S}_{1}$ };
\draw[line width=0.025 mm, dashed,fill=gray!30] (0,0) --  (-105:4.5) arc(-105:-45:4.5) -- cycle;
\draw[line width=0.25 mm, dashed,fill=white] (0,0) --  (-105:3) arc(-105:-45:3) -- cycle;
\node at (-75:3.8) {$\omega^{2}\mathsf{S}_{1}$ };

\draw[line width=0.25 mm, dashed,fill=gray!30] (0,0) --  (105:3) arc(105:165:3) -- cycle;
\node at (135:2) {$\mathsf{S}_{1}$ };
\draw[line width=0.25 mm, dashed,fill=gray!30] (0,0) --  (-15:3) arc(-15:-75:3) -- cycle;
\node at (-45:2) {$\mathsf{S}_{1}$ };

\draw[line width=0.25 mm, dashed,fill=gray!30] (0,0) --  (-135:3) arc(-135:-75:3) -- cycle;
\node at (-105:2) {$\omega \mathsf{S}_{1}$ };
\draw[line width=0.25 mm, dashed,fill=gray!30] (0,0) --  (45:3) arc(45:105:3) -- cycle;
\node at (75:2) {$\omega \mathsf{S}_{1}$ };

\draw[line width=0.25 mm, dashed,fill=gray!30] (0,0) --  (-15:3) arc(-15:45:3) -- cycle;
\node at (15:2) {$\omega^{2} \mathsf{S}_{1}$ };
\draw[line width=0.25 mm, dashed,fill=gray!30] (0,0) --  (-195:3) arc(-195:-135:3) -- cycle;
\node at (-165:2) {$\omega^{2} \mathsf{S}_{1}$ };

\draw[black,line width=0.45 mm] (15:3)--(15:4.5);
\draw[black,line width=0.45 mm] (0,0)--(45:3);
\draw[black,line width=0.45 mm] (75:3)--(75:4.5);
\draw[black,line width=0.45 mm] (0,0)--(105:3);
\draw[black,line width=0.45 mm] (135:3)--(135:4.5);
\draw[black,line width=0.45 mm] (0,0)--(165:3);
\draw[black,line width=0.45 mm] (0,0)--(-15:3);
\draw[black,line width=0.45 mm] (-45:3)--(-45:4.5);
\draw[black,line width=0.45 mm] (0,0)--(-75:3);
\draw[black,line width=0.45 mm] (-105:3)--(-105:4.5);
\draw[black,line width=0.45 mm] (0,0)--(-135:3);
\draw[black,line width=0.45 mm] (-165:3)--(-165:4.5);

\draw[black,line width=0.45 mm] ([shift=(-180:3cm)]0,0) arc (-180:180:3cm);

\draw[fill] (0:3) circle (0.08);
\draw[fill] (60:3) circle (0.08);
\draw[fill] (120:3) circle (0.08);
\draw[fill] (180:3) circle (0.08);
\draw[fill] (240:3) circle (0.08);
\draw[fill] (300:3) circle (0.08);

\end{tikzpicture} \hspace{-0.4cm} \begin{tikzpicture}[scale=1.1, master]
\node at (0,0) {};

\draw[line width=0.025 mm, dashed,fill=gray!30] (0,0) --  (-15:4.5) arc(-15:-75:4.5) -- cycle;
\draw[line width=0.25 mm, dashed,fill=white] (0,0) --  (-15:3) arc(-15:-75:3) -- cycle;
\node at (-45:3.8) {$\mathsf{S}_{2}$ };
\draw[line width=0.025 mm, dashed,fill=gray!30] (0,0) --  (105:4.5) arc(105:165:4.5) -- cycle;
\draw[line width=0.25 mm, dashed,fill=white] (0,0) --  (105:3) arc(105:165:3) -- cycle;
\node at (135:3.8) {$\mathsf{S}_{2}$ };

\draw[line width=0.025 mm, dashed,fill=gray!30] (0,0) --  (45:4.5) arc(45:105:4.5) -- cycle;
\draw[line width=0.25 mm, dashed,fill=white] (0,0) --  (45:3) arc(45:105:3) -- cycle;
\node at (75:3.8) {$\omega\mathsf{S}_{2}$ };
\draw[line width=0.025 mm, dashed,fill=gray!30] (0,0) --  (225:4.5) arc(225:285:4.5) -- cycle;
\draw[line width=0.25 mm, dashed,fill=white] (0,0) --  (225:3) arc(225:285:3) -- cycle;
\node at (255:3.8) {$\omega\mathsf{S}_{2}$ };

\draw[line width=0.025 mm, dashed,fill=gray!30] (0,0) --  (165:4.5) arc(165:225:4.5) -- cycle;
\draw[line width=0.25 mm, dashed,fill=white] (0,0) --  (165:3) arc(165:225:3) -- cycle;
\node at (195:3.8) {$\omega^{2}\mathsf{S}_{2}$ };
\draw[line width=0.025 mm, dashed,fill=gray!30] (0,0) --  (-15:4.5) arc(-15:45:4.5) -- cycle;
\draw[line width=0.25 mm, dashed,fill=white] (0,0) --  (-15:3) arc(-15:45:3) -- cycle;
\node at (15:3.8) {$\omega^{2}\mathsf{S}_{2}$ };

\draw[line width=0.25 mm, dashed,fill=gray!30] (0,0) --  (195:3) arc(195:255:3) -- cycle;
\node at (225:2) {$\mathsf{S}_{2}$ };
\draw[line width=0.25 mm, dashed,fill=gray!30] (0,0) --  (75:3) arc(75:15:3) -- cycle;
\node at (45:2) {$\mathsf{S}_{2}$ };

\draw[line width=0.25 mm, dashed,fill=gray!30] (0,0) --  (-45:3) arc(-45:15:3) -- cycle;
\node at (-15:2) {$\omega \mathsf{S}_{2}$ };
\draw[line width=0.25 mm, dashed,fill=gray!30] (0,0) --  (135:3) arc(135:195:3) -- cycle;
\node at (165:2) {$\omega \mathsf{S}_{2}$ };

\draw[line width=0.25 mm, dashed,fill=gray!30] (0,0) --  (75:3) arc(75:135:3) -- cycle;
\node at (105:2) {$\omega^{2} \mathsf{S}_{2}$ };
\draw[line width=0.25 mm, dashed,fill=gray!30] (0,0) --  (-105:3) arc(-105:-45:3) -- cycle;
\node at (-75:2) {$\omega^{2} \mathsf{S}_{2}$ };

\draw[black,line width=0.45 mm] (105:3)--(105:4.5);
\draw[black,line width=0.45 mm] (0,0)--(135:3);
\draw[black,line width=0.45 mm] (165:3)--(165:4.5);
\draw[black,line width=0.45 mm] (0,0)--(195:3);
\draw[black,line width=0.45 mm] (225:3)--(225:4.5);
\draw[black,line width=0.45 mm] (0,0)--(255:3);
\draw[black,line width=0.45 mm] (0,0)--(75:3);
\draw[black,line width=0.45 mm] (45:3)--(45:4.5);
\draw[black,line width=0.45 mm] (0,0)--(15:3);
\draw[black,line width=0.45 mm] (-15:3)--(-15:4.5);
\draw[black,line width=0.45 mm] (0,0)--(-45:3);
\draw[black,line width=0.45 mm] (-75:3)--(-75:4.5);

\draw[black,line width=0.45 mm] ([shift=(-180:3cm)]0,0) arc (-180:180:3cm);

\draw[fill] (0:3) circle (0.08);
\draw[fill] (60:3) circle (0.08);
\draw[fill] (120:3) circle (0.08);
\draw[fill] (180:3) circle (0.08);
\draw[fill] (240:3) circle (0.08);
\draw[fill] (300:3) circle (0.08);

\end{tikzpicture} \hspace{-0.4cm} \begin{tikzpicture}[scale=1.1,slave]
\node at (0,0) {};

\draw[line width=0.025 mm, dashed,fill=gray!30] (0,0) --  (30:4.5) arc(30:150:4.5) -- cycle;
\draw[line width=0.25 mm, dashed,fill=white] (0,0) --  (30:3) arc(30:150:3) -- cycle;
\node at (90:3.8) {$\mathsf{S}_{3}$ };

\draw[line width=0.025 mm, dashed,fill=gray!30] (0,0) --  (150:4.5) arc(150:270:4.5) -- cycle;
\draw[line width=0.25 mm, dashed,fill=white] (0,0) --  (150:3) arc(150:270:3) -- cycle;
\node at (210:3.8) {$\omega\mathsf{S}_{3}$ };

\draw[line width=0.025 mm, dashed,fill=gray!30] (0,0) --  (270:4.5) arc(270:390:4.5) -- cycle;
\draw[line width=0.25 mm, dashed,fill=white] (0,0) --  (270:3) arc(270:390:3) -- cycle;
\node at (330:3.8) {$\omega^{2}\mathsf{S}_{3}$ };

\draw[line width=0.25 mm, dashed,fill=gray!30] (0,0) --  (-150:3) arc(-150:-30:3) -- cycle;
\node at (-90:2) {$\mathsf{S}_{3}$ };

\draw[line width=0.25 mm, dashed,fill=gray!30] (0,0) --  (-30:3) arc(-30:90:3) -- cycle;
\node at (30:2) {$\omega\mathsf{S}_{3}$ };

\draw[line width=0.25 mm, dashed,fill=gray!30] (0,0) --  (90:3) arc(90:210:3) -- cycle;
\node at (150:2) {$\omega^{2}\mathsf{S}_{3}$ };

\draw[black,line width=0.45 mm] (30:3)--(30:4.5);
\draw[black,line width=0.45 mm] (0,0)--(90:3);
\draw[black,line width=0.45 mm] (150:3)--(150:4.5);
\draw[black,line width=0.45 mm] (0,0)--(-30:3);
\draw[black,line width=0.45 mm] (-90:3)--(-90:4.5);
\draw[black,line width=0.45 mm] (0,0)--(-150:3);

\draw[black,line width=0.45 mm] ([shift=(-180:3cm)]0,0) arc (-180:180:3cm);

\draw[fill] (0:3) circle (0.08);
\draw[fill] (60:3) circle (0.08);
\draw[fill] (120:3) circle (0.08);
\draw[fill] (180:3) circle (0.08);
\draw[fill] (240:3) circle (0.08);
\draw[fill] (300:3) circle (0.08);
\end{tikzpicture}}
\end{center}
\begin{figuretext}\label{fig: S1S2S3 sf}
The sets $\mathsf{S}_{j},\omega\mathsf{S}_{j},\omega^{2}\mathsf{S}_{j}$, $j=1,2,3$.
\end{figuretext}
\end{figure}

The next two propositions are a consequence of the properties of $\mu_1$ and $\mu_3$ established in Propositions \ref{XYprop}--\ref{prop: basic properties for t=0}. An important difference compared to \cite[Proposition 3.5]{CLmain} is that, in our case, the matrices $\{s_{l},S_{l}\}_{1}^{\infty}$ of Proposition \ref{sprop} (e) are not necessarily diagonal. 

\begin{proposition}[Properties of $s(k)$ and $S(k)$]\label{sprop}
The functions $s(k)$ and $S(k)$ defined in \eqref{def of S and s} have the following properties:
\begin{enumerate}[$(a)$]
\item The entries of $s(k)$ and $S(k)$ have the following domains of definition:
\begin{subequations}\label{sSdomainofdefinition}
\begin{align}
& s(k) \mbox{ is well-defined and continuous for } k \in (\omega^2 \hat{\mathcal{S}}_{3}, \omega \hat{\mathcal{S}}_{3}, \hat{\mathcal{S}}_{3}) \setminus \hat{\mathcal{Q}}, \label{region of def of s} \\
& S(k) \mbox{ is well-defined and continuous for } k \in \C \setminus \hat{\mathcal{Q}}. \label{region of def of S}
\end{align}
\end{subequations}

\item The entries of $s(k)$ and $S(k)$ are analytic in the interior of their domains of definition as given in \eqref{sSdomainofdefinition}.

\item For $j = 1, 2, \dots$, the derivatives $\partial_k^js(k)$ and $\partial_k^j S(k)$ are well-defined and continuous for $k$ as in \eqref{sSdomainofdefinition}.

\item $s(k)$ and $S(k)$ obey the symmetries
\begin{align}\label{symmetries of s}
&  s(k) = \mathcal{A} s(\omega k)\mathcal{A}^{-1} = \mathcal{B} s(k^{-1})\mathcal{B}, \qquad S(k) = \mathcal{A} S(\omega k)\mathcal{A}^{-1} = \mathcal{B} S(k^{-1})\mathcal{B}.
\end{align}

\item $s(k)$ and $S(k)$ approach the identity matrix as $k \to \infty$. More precisely, for any $N \geq 1$, 
\begin{subequations}
\begin{align}
& \Big|\partial_k^j \big(s(k) - I - \sum_{l=1}^N \frac{s_l}{k^l}\big)\Big| = O(k^{-N-1}), & & \mbox{as } k \to \infty, \;\; k \in (\omega^2 \hat{\mathsf{S}}_{3}, \omega \hat{\mathsf{S}}_{3}, \hat{\mathsf{S}}_{3}), \label{asymp of s at inf} \\
& \Big|\partial_k^j \big(S(k) - I - \sum_{l=1}^N \frac{S_l}{k^l}\big)\Big| = O(k^{-N-1}), & & \mbox{as }k \to \infty, \;\; k \in (\omega^2 \hat{\mathsf{S}}_{1}, \omega \hat{\mathsf{S}}_{1}, \hat{\mathsf{S}}_{1}), \label{asymp of S at inf}
\end{align}
\end{subequations}
for $j = 0, 1, \dots, N$, where $s_{l} = \mu_{3}^{(l)}(0,0)$, $S_{l} = \mu_{1}^{(l)}(0,0)$, and the $\mu_{j}^{(l)}$ are as in \eqref{mujlargekexpansion}.

\item For each integer $l\geq -1$, let 
\begin{align*}
& s_{1}^{(l)}=C_{3}^{(l)}(0,0), & & s_{-1}^{(l)}=C_{6}^{(l)}(0,0), & & S_{1}^{(l)}=C_{1}^{(l)}(0,0), & & S_{-1}^{(l)}=C_{4}^{(l)}(0,0),
\end{align*}
where the quantities $C_{j}^{(l)}$ are as in Proposition \ref{XYat1prop}. We have
\begin{subequations}
\begin{align}
& s(k) = \frac{s_{1}^{(-1)}}{k-1} + I+s_{1}^{(0)} + s_{1}^{(1)}(k-1) + \cdots & & \mbox{as } k \to 1, \;\; k \in (\omega^2 \hat{\mathcal{S}}_{3}, \omega \hat{\mathcal{S}}_{3}, \hat{\mathcal{S}}_{3}), \label{s at 1} \\
& s(k) = \frac{s_{-1}^{(-1)}}{k+1} + I+s_{-1}^{(0)} + s_{-1}^{(1)}(k+1) + \cdots & & \mbox{as } k \to -1, \;\; k \in (\omega^2 \hat{\mathcal{S}}_{3}, \omega \hat{\mathcal{S}}_{3}, \hat{\mathcal{S}}_{3}), \label{s at -1} \\
& S(k) = \frac{S_{1}^{(-1)}}{k-1} + I+S_{1}^{(0)} + S_{1}^{(1)}(k-1) + \cdots & & \mbox{as } k \to 1,  \label{S at 1} \\
& S(k) = \frac{S_{-1}^{(-1)}}{k+1} + I+S_{-1}^{(0)} + S_{-1}^{(1)}(k+1) + \cdots & & \mbox{as } k \to -1,  \label{S at -1}
\end{align}
\end{subequations}
and the expansions can be differentiated termwise any number of times. The first matrices satisfy
\begin{align}
& s_{1}^{(-1)} = \begin{pmatrix}
\mathfrak{s}_{1} & \mathfrak{s}_{1} & \mathfrak{s}_{2} \\
-\mathfrak{s}_{1} & -\mathfrak{s}_{1} & -\mathfrak{s}_{2} \\
0 & 0 & 0
\end{pmatrix}, & & s_{-1}^{(-1)} = \begin{pmatrix}
\mathfrak{s}_{3} & \mathfrak{s}_{3} & \mathfrak{s}_{4} \\
-\mathfrak{s}_{3} & -\mathfrak{s}_{3} & -\mathfrak{s}_{4} \\
0 & 0 & 0
\end{pmatrix}, \label{s1pm1p and sm1pm1p} \\
& (s_{1}^{(0)})_{31} = (s_{1}^{(0)})_{32}, & & (s_{-1}^{(0)})_{31} = (s_{-1}^{(0)})_{32}, \\
& S_{1}^{(-1)} = \begin{pmatrix}
\mathscr{S}_{1} & \mathscr{S}_{1} & \mathscr{S}_{2} \\
-\mathscr{S}_{1} & -\mathscr{S}_{1} & -\mathscr{S}_{2} \\
0 & 0 & 0
\end{pmatrix}, & & S_{-1}^{(-1)} = \begin{pmatrix}
\mathscr{S}_{3} & \mathscr{S}_{3} & \mathscr{S}_{4} \\
-\mathscr{S}_{3} & -\mathscr{S}_{3} & -\mathscr{S}_{4} \\
0 & 0 & 0
\end{pmatrix}, \label{S1pm1p and Sm1pm1p} \\
& (S_{1}^{(0)})_{31} = (S_{1}^{(0)})_{32}, & & (S_{-1}^{(0)})_{31} = (S_{-1}^{(0)})_{32},
\end{align}
for certain constants $\mathfrak{s}_{1},\mathfrak{s}_{2},\mathfrak{s}_{3},\mathfrak{s}_{4},\mathscr{S}_{1},\mathscr{S}_{2},\mathscr{S}_{3},\mathscr{S}_{4} \in  \mathbb{C}$.
\item If $u_0, v_0$ have compact support, then $s(k)$ is defined and analytic for $k \in \C \setminus \hat{\mathcal{Q}}$, $\det s(k) = 1$ for $k \in \C \setminus \hat{\mathcal{Q}}$, and the second equation in \eqref{mu3mu2mu1sS} holds for $k \in \C \setminus \hat{\mathcal{Q}}$.

\medskip \noindent $\det S(k) = 1$ for $k \in \C \setminus \hat{\mathcal{Q}}$, and the first equation in \eqref{mu3mu2mu1sS} holds for $k \in \C \setminus \hat{\mathcal{Q}}$.

\item For each $n \geq 1$ and $\epsilon > 0$, there exists $C>0$ such that the following estimates hold for $ \ell = 0, 1, \dots, n$:
\begin{align}
& \big|\tfrac{\partial^\ell}{\partial k^\ell}s(k) \big| \leq C, & & k \in (\omega^2 \hat{\mathsf{S}}_{3}, \omega \hat{\mathsf{S}}_{3}, \hat{\mathsf{S}}_{3}), \ \dist(k,\hat{\mathcal{Q}}) > \epsilon, \label{region of boundedness of s} \\
& \big|\tfrac{\partial^\ell}{\partial k^\ell}S(k) \big| \leq C, & & k \in (\omega^2 \hat{\mathsf{S}}_{1}, \omega \hat{\mathsf{S}}_{1}, \hat{\mathsf{S}}_{1}), \ \dist(k,\hat{\mathcal{Q}}) > \epsilon. \label{region of boundedness of S}
\end{align} 
\end{enumerate}
\end{proposition}

\begin{proposition}[Properties of $s^{A}(k)$ and $S^{A}(k)$]\label{sAprop}
The functions $s^A(k)$ and $S^A(k)$ defined in (\ref{SAsAdef}) have the following properties:
\begin{enumerate}[$(a)$]
\item The entries of $s^A(k)$ and $S^A(k)$ have the following domains of definition:
\begin{subequations}\label{sASAdomainofdefinition}
\begin{align}
& s^{A}(k) \mbox{ is well-defined and continuous for } k \in (\omega^2 \hat{\mathcal{S}}_{3}^{*}, \omega \hat{\mathcal{S}}_{3}^{*}, \hat{\mathcal{S}}_{3}^{*}) \setminus \hat{\mathcal{Q}}, \label{region of def of sA} \\
& S^{A}(k) \mbox{ is well-defined and continuous for } k \in \C \setminus \hat{\mathcal{Q}}. \label{region of def of SA}
\end{align}
\end{subequations}
 
\item The entries of $s^A(k)$ and $S^A(k)$ are analytic in the interior of their domains of definition as given in (\ref{sASAdomainofdefinition}). 
 
\item For $j = 1, 2, \dots$, the derivatives $\partial_k^js^A(k)$ and $\partial_k^j S^A(k)$ are well-defined and continuous for $k$ as in (\ref{sASAdomainofdefinition}).

\item $s^A(k)$ and $S^A(k)$ obey the symmetries
\begin{align}\label{symmetries of sASA}
&  s^A(k) = \mathcal{A} s^A(\omega k)\mathcal{A}^{-1} = \mathcal{B} s^A(k^{-1})\mathcal{B}, \qquad S^A(k) = \mathcal{A} S^A(\omega k)\mathcal{A}^{-1} = \mathcal{B} S^A(k^{-1})\mathcal{B}.
\end{align}

\item $s^{A}(k)$ and $S^{A}(k)$ approach the identity matrix as $k \to \infty$. More precisely, for any $N \geq 1$,
\begin{subequations}
\begin{align}
& \Big|\partial_k^j \big(s^{A}(k) - I - \sum_{l=1}^N \frac{s_l^{A}}{k^l}\big)\Big| = O(k^{-N-1}), & & \mbox{as } k \to \infty, \; \; k \in (\omega^2 \hat{\mathsf{S}}_{3}^{*}, \omega \hat{\mathsf{S}}_{3}^{*},  \hat{\mathsf{S}}_{3}^{*}), \label{asymp of sA at inf} \\
& \Big|\partial_k^j \big(S^{A}(k) - I - \sum_{l=1}^N \frac{S_l^{A}}{k^l}\big)\Big| = O(k^{-N-1}), & & \mbox{as } k \to \infty, \; \; k \in (\omega^2 \hat{\mathsf{S}}_{2}, \omega \hat{\mathsf{S}}_{2},  \hat{\mathsf{S}}_{2}), \label{asymp of SA at inf}
\end{align}
\end{subequations}
for $j = 0, \dots, N$, where $s_{l}^{A} = \mu_{3}^{A(l)}(0,0)$, $S_{l}^{A}=\mu_{1}^{A(l)}(0,0)$, and the $\mu_{j}^{A(l)}$ are as in \eqref{mujAlargekexpansion}.

\item For each interger $l\geq -1$, let
\begin{align*}
& s_{1}^{A(l)}=D_{3}^{(l)}(0,0), & & s_{-1}^{A(l)}=D_{6}^{(l)}(0,0), & & S_{1}^{A(l)}=D_{1}^{(l)}(0,0), & & S_{-1}^{A(l)}=D_{4}^{(l)}(0,0),
\end{align*}
where the quantities $D_{j}^{(l)}$ are as in Proposition \ref{XYat1prop adjoint}. We have
\begin{subequations}
\begin{align}
& s^{A}(k) = \frac{s_{1}^{A(-1)}}{k-1} + I+s_{1}^{A(0)} + s_{1}^{A(1)}(k-1) + \cdots & & \hspace{-0.25cm}\mbox{as } k \to 1, \;\; k \in (\omega^2 \hat{\mathcal{S}}_{3}^{*}, \omega \hat{\mathcal{S}}_{3}^{*}, \hat{\mathcal{S}}_{3}^{*}), \label{sA at 1} \\
& s^{A}(k) = \frac{s_{-1}^{A(-1)}}{k+1}  + I+s_{-1}^{A(0)} + s_{-1}^{A(1)}(k+1) + \cdots & & \hspace{-0.25cm} \mbox{as } k \to -1, \; k \in (\omega^2 \hat{\mathcal{S}}_{3}^{*}, \omega \hat{\mathcal{S}}_{3}^{*}, \hat{\mathcal{S}}_{3}^{*}), \label{sA at -1} \\
& S^{A}(k) = \frac{S_{1}^{A(-1)}}{k-1} + I+S_{1}^{A(0)} + S_{1}^{A(1)}(k-1) + \cdots & & \hspace{-0.25cm}\mbox{as } k \to 1,  \label{SA at 1} \\
& S^{A}(k) = \frac{S_{-1}^{A(-1)}}{k+1} + I+S_{-1}^{A(0)} + S_{-1}^{A(1)}(k+1) + \cdots & & \hspace{-0.25cm}\mbox{as } k \to -1,  \label{SA at -1}
\end{align}
\end{subequations}
and the expansions can be differentiated termwise any number of times. The first matrices satisfy
\begin{align}
& s_{1}^{A(-1)} = \begin{pmatrix}
\mathfrak{s}_{1}^{A} & -\mathfrak{s}_{1}^{A} & 0 \\
\mathfrak{s}_{1}^{A} & -\mathfrak{s}_{1}^{A} & 0 \\
\mathfrak{s}_{2}^{A} & -\mathfrak{s}_{2}^{A} & 0
\end{pmatrix}, & & s_{-1}^{A(-1)} = \begin{pmatrix}
\mathfrak{s}_{3}^{A} & -\mathfrak{s}_{3}^{A} & 0 \\
\mathfrak{s}_{3}^{A} & -\mathfrak{s}_{3}^{A} & 0 \\
\mathfrak{s}_{4}^{A} & -\mathfrak{s}_{4}^{A} & 0
\end{pmatrix}, \label{sA1pm1p and sAm1pm1p} \\
& (s_{1}^{A(0)})_{13} = (s_{1}^{A(0)})_{23}, & & (s_{-1}^{A(0)})_{13} = (s_{-1}^{A(0)})_{23}, \\
& S_{1}^{A(-1)} = \begin{pmatrix}
\mathscr{S}_{1}^{A} & -\mathscr{S}_{1}^{A} & 0 \\
\mathscr{S}_{1}^{A} & -\mathscr{S}_{1}^{A} & 0 \\
\mathscr{S}_{2}^{A} & -\mathscr{S}_{2}^{A} & 0
\end{pmatrix}, & & S_{-1}^{A(-1)} = \begin{pmatrix}
\mathscr{S}_{3}^{A} & -\mathscr{S}_{3}^{A} & 0 \\
\mathscr{S}_{3}^{A} & -\mathscr{S}_{3}^{A} & 0 \\
\mathscr{S}_{4}^{A} & -\mathscr{S}_{4}^{A} & 0
\end{pmatrix}, \label{SA1pm1p and SAm1pm1p} \\
& (S_{1}^{A(0)})_{13} = (S_{1}^{A(0)})_{23}, & & (S_{-1}^{A(0)})_{13} = (S_{-1}^{A(0)})_{23},
\end{align}
for certain constants $\mathfrak{s}_{1}^{A},\mathfrak{s}_{2}^{A},\mathfrak{s}_{3}^{A},\mathfrak{s}_{4}^{A},\mathscr{S}_{1}^{A},\mathscr{S}_{2}^{A},\mathscr{S}_{3}^{A},\mathscr{S}_{4}^{A} \in  \mathbb{C}$.
\item If $u_0, v_0$ have compact support, then $s^A(k)$ is defined and equals the cofactor matrix of $s(k)$ for all $k \in \C \setminus \hat{\mathcal{Q}}$.

\medskip \noindent $S^A(k)$ is defined and equals the cofactor matrix of $S(k)$ for all $k \in \C \setminus \hat{\mathcal{Q}}$.
\item For each $n \geq 1$ and $\epsilon > 0$, there exists $C>0$ such that the following estimates hold for $ \ell = 0, 1, \dots, n$:
\begin{align}
& \big|\tfrac{\partial^\ell}{\partial k^\ell}s^{A}(k) \big| \leq C, & & k \in (\omega^2 \hat{\mathsf{S}}_{3}^{*}, \omega \hat{\mathsf{S}}_{3}^{*}, \hat{\mathsf{S}}_{3}^{*}), \ \dist(k,\hat{\mathcal{Q}}) > \epsilon, \label{region of boundedness of sA} \\
& \big|\tfrac{\partial^\ell}{\partial k^\ell}S^{A}(k) \big| \leq C, & & k \in (\omega^2 \hat{\mathsf{S}}_{2}, \omega \hat{\mathsf{S}}_{2}, \hat{\mathsf{S}}_{2}), \ \dist(k,\hat{\mathcal{Q}}) > \epsilon. \label{region of boundedness of SA}
\end{align} 
\end{enumerate}
\end{proposition}

\subsection{Proof of Theorem \ref{thm:r1r2}}
Recall from (\ref{r1r2def}) that 
\begin{align}
& r_1(k) = \frac{s_{12}(k)}{s_{11}(k)}, \qquad \tilde{r}_1(k) = \frac{(S^{-1}s)_{12}(k)}{(S^{-1}s)_{11}(k)}, \qquad r_2(k) = \frac{s^A_{12}(k)}{s^A_{11}(k)}, \qquad \tilde{r}_2(k) = \frac{(S^{T}s^{A})_{12}(k)}{(S^{T}s^{A})_{11}(k)}, \nonumber \\
& \hat{r}_2(k) = \frac{s^A_{12}S^{A}_{33}-s^{A}_{32}S^{A}_{13}}{s^A_{11}S^{A}_{33}-s^{A}_{31}S^{A}_{13}}(k), \qquad \check{r}_2(k) = \frac{s^A_{12}S^{A}_{22}-s^{A}_{22}S^{A}_{12}}{s^A_{11}S^{A}_{22}-s^{A}_{21}S^{A}_{12}}(k),
\end{align}
and that $R_{1},R_{2},\tilde{R}_{2}$ are defined in \eqref{def of R1 R2}. 
Statements $(a)$ and $(c)$ of Propositions \ref{sprop} and \ref{sAprop} imply that 
$r_{1},\tilde{r}_{1},r_{2},\tilde{r}_{2},\hat{r}_{2},\check{r}_{2},R_{1},R_{2},\tilde{R}_{2}$ are smooth on their respective domains of definition, except at possible values of $k$ where a denominator vanishes.  Assumption \ref{solitonlessassumption}, combined with Remark \ref{remark:sym combined with no soliton assumption}, implies that the denominators of $r_{1},\tilde{r}_{1},r_{2},\tilde{r}_{2},\hat{r}_{2},\check{r}_{2},R_{1},R_{2},\tilde{R}_{2}$ do not vanish on their domains of definition. It follows that these functions are smooth everywhere on their domains; in particular,
\begin{align}\label{lol6}
r_1, \tilde{r}_{1} \in C^\infty\big((\Gamma_{1'}\cup \Gamma_{10''}\cup \partial \D)\setminus \hat{\mathcal{Q}}\big), \qquad r_2,\tilde{r}_{2},\hat{r}_{2},\check{r}_{2} \in C^\infty \big((\Gamma_{1''}\cup \Gamma_{10'}\cup \partial \D)\setminus \hat{\mathcal{Q}}\big).
\end{align}
Moreover, statements $(b)$ of Propositions \ref{sprop} and \ref{sAprop}, together with Assumption \ref{solitonlessassumption}, imply that 
$R_{1},R_{2},\tilde{R}_{2}$ are analytic on the interior of their domains of definition.

\medskip By combining \eqref{lol6}, Assumption \ref{originassumption}, the asymptotics of $s(k), s^{A}(k)$ as $k \to k_{\star}\in \{-1,1\}$ (see the statements $(f)$ of Propositions \ref{sprop} and \ref{sAprop}), and the symmetries \eqref{symmetries of s} and \eqref{symmetries of sASA}, we infer that assertions $(\ref{Theorem2.3itemi})$ and $(\ref{Theorem2.3itemii})$ hold.

\bigskip  For $k \in \partial \mathbb{D}\setminus \hat{\mathcal{Q}}$, all entries of $s$ and $s^{A}$ are well-defined and thus $s^{A}=(s^{-1})^{T}$. The relations in \eqref{relations on the unit circle} can be verified by a long but direct computation, using \eqref{r1r2def}, \eqref{def of R1 R2}, the symmetries \eqref{symmetries of s} and \eqref{symmetries of sASA}, $s^{A}=(s^{-1})^{T}$ and $S^{A}=(S^{-1})^{T}$. This finishes the proof of $(\ref{Theorem2.3itemiv})$. 


\medskip Using $r_1(k) = s_{12}(k)/s_{11}(k)$, the definition \eqref{def of R1 R2} of $\tilde{R}_{2}$ and  \eqref{asymp of s at inf}--\eqref{asymp of S at inf}, \eqref{asymp of sA at inf}--\eqref{asymp of SA at inf}, we find \eqref{r1 at inf} and \eqref{Rt2 at inf}. Similarly, \eqref{r2 at 0}, \eqref{R1 at 0} and \eqref{R2 at 0} directly follow from $r_2(k) = s^A_{12}(k)/s^A_{11}(k)$, the definitions \eqref{def of R1 R2} of $R_{1},R_{2}$, \eqref{asymp of s at inf}--\eqref{asymp of S at inf}, \eqref{asymp of sA at inf}--\eqref{asymp of SA at inf} and the $\mathcal{B}$-symmetries in \eqref{symmetries of s} and \eqref{symmetries of sASA}. On the other hand, the asymptotic formulas \eqref{rt1 at 0} and \eqref{rt2 at inf} do not directly follow from Propositions \ref{sprop} and \ref{sAprop}. For example, since $\tilde{r}_2(k) = (S^{T}s^{A})_{12}(k)/(S^{T}s^{A})_{11}(k)$, the behavior of $\tilde{r}_2(k)$ as $k\to \infty$, $k\in \Gamma_{1''}$ apriori requires asymptotic formulas for $[s^{A}]_{1}(k),[s^{A}]_{2}(k),[S]_{1}(k)$ as $k\to \infty$, $k\in \Gamma_{1''}$. While such formulas can be obtained from Propositions \ref{sprop} and \ref{sAprop}, a computation only yields $\tilde{r}_{2}(k)=O(k^{-3})$ as $k\to \infty$, $k\in \Gamma_{1''}$. To improve on this formula, we note that the two equations in \eqref{mu3mu2mu1sS}, together with $\mu_{1}(0,T,k)=I$, imply that
\begin{align}\label{global rel}
(S^{-1}s)(k) = e^{-T\hat{\mathcal{Z}}(k)} \mu_{3}(0,T,k), \qquad k \in (\omega^2 \hat{\mathcal{S}}_{3}, \omega \hat{\mathcal{S}}_{3}, \hat{\mathcal{S}}_{3}) \setminus \hat{\mathcal{Q}}.
\end{align} 
Similarly, we have
\begin{align}\label{global rel Adj}
(S^{T}s^{A})(k) = e^{T\hat{\mathcal{Z}}(k)} \mu_{3}^{A}(0,T,k), \qquad k \in (\omega^2 \hat{\mathcal{S}}_{3}^{*}, \omega \hat{\mathcal{S}}_{3}^{*}, \hat{\mathcal{S}}_{3}^{*}) \setminus \hat{\mathcal{Q}}.
\end{align}
Now \eqref{rt2 at inf} directly follows from \eqref{global rel Adj} and \eqref{mujlargekexpansion t=0 muAj}, and \eqref{rt1 at 0} follows from \eqref{global rel}, \eqref{mujlargekexpansion t=0 muj} and the $\mathcal{B}$-symmetry in \eqref{XYsymm}. This finishes the proof of $(\ref{Theorem2.3itemiii})$.


\section{The function $M$}\nequation

For convenience, we define
\begin{align*}
F_{n} = \begin{cases}
D_{n}, & \mbox{if } n \in \{1,\ldots,18\}, \\
E_{n-18}, & \mbox{if } n \in \{19,\ldots,36\}.
\end{cases}
\end{align*}
For each $n = 1, \dots, 36$, we denote by $M_n(x,t,k)$ the $3\times 3$-matrix valued solution of (\ref{Xlax}) defined for $k \in \bar{F}_n\setminus \hat{\mathcal{Q}}$ by the following Fredholm integral equations
\begin{align}\label{Mndef}
(M_n)_{ij}(x,t,k) = \delta_{ij} + \int_{\gamma_{ij}^n} \left(e^{x\hat{\mathcal{L}}+t\hat{\mathcal{Z}}} W_{n}(x',t',k) \right)_{ij}, \qquad  i,j = 1, 2,3,
\end{align}
where $\gamma^n_{ij}$, $n = 1, \dots, 36$, $i,j = 1,2,3$, are given by
\begin{align} \label{gammaijndef}
 \gamma_{ij}^n =  \begin{cases}
 \gamma_1,  & \mbox{if } \text{Re}\, l_i(k) < \text{Re}\, l_j(k), \quad \text{Re}\, z_i(k) \geq \text{Re}\, z_j(k),
	\\
\gamma_2,  & \mbox{if } \text{Re}\, l_i(k) < \text{Re}\, l_j(k),\quad \text{Re}\, z_i(k) < \text{Re}\, z_j(k),
	\\
\gamma_3,  & \mbox{if } \text{Re}\, l_i(k) \geq \text{Re}\, l_j(k),
	\\
\end{cases} \quad \text{for} \quad k \in F_n,
\end{align}
and $W_{n}$ is given by the first equation in \eqref{Wdef} with $\mu$ replaced by $M_{n}$.

The definition of $\gamma_{ij}^n$ guarantees that the exponential $e^{(l_i - l_j)(x-x')+(z_i - z_j)(t-t')}$ appearing in the integrand in \eqref{Mndef} remains bounded for $k \in F_n$ and $(x',t') \in \gamma_{ij}^n$. The next proposition states that this makes all entries of $M_n$ well-defined for $k \in \bar{F}_n\setminus (\hat{\mathcal{Q}}\cup\mathsf{Z})$, where $\mathsf{Z}$ denotes the set of zeros of the Fredholm determinants associated with (\ref{Mndef}). The proof is omitted (see \cite[Proposition 4.1 and Section 5]{CLgoodboussinesq} for a detailed proof of a similar statement).

\begin{proposition}[Basic properties of $M_n$]\label{Mnprop}
Equation (\ref{Mndef}) uniquely defines thirty-six $3 \times 3$-matrix valued solutions $\{M_n\}_1^{36}$ of (\ref{Xlax}) with the following properties:
\begin{enumerate}[$(a)$]
\item The function $M_n(x,t, k)$ is defined for $x \geq 0$, $t\in [0,T]$ and $k \in \bar{F}_n \setminus (\hat{\mathcal{Q}}\cup \mathsf{Z})$. For each $k \in \bar{F}_n  \setminus (\hat{\mathcal{Q}}\cup \mathsf{Z})$ and $t\in [0,T]$, $M_n(\cdot,t, k)$ is smooth and satisfies the $x$-part of (\ref{Xlax}). For each $k \in \bar{F}_n  \setminus (\hat{\mathcal{Q}}\cup \mathsf{Z})$ and $x \geq 0$, $M_n(x,\cdot, k)$ is smooth and satisfies the $t$-part of (\ref{Xlax}).

\item For each $x \geq 0$ and $t\in [0,T]$, $M_n(x,t,\cdot)$ is continuous for $k \in \bar{F}_n \setminus (\hat{\mathcal{Q}}\cup \mathsf{Z})$ and analytic for $k \in F_n\setminus (\hat{\mathcal{Q}}\cup \mathsf{Z})$.

\item For each $\epsilon > 0$, there exists a $C = C(\epsilon)$ such that
\begin{align}\label{Mnbounded}
|M_n(x,t,k)| \leq C, \qquad x \geq 0, \; t\in[0,T], \ k \in \bar{F}_n, \ \mbox{dist}(k, \hat{\mathcal{Q}}\cup \mathsf{Z}) \geq \epsilon.
\end{align}

\item For each $x \geq 0$, $t\in [0,T]$ and each integer $j \geq 1$, $\partial_{k}^j M_n(x,t, \cdot)$ has a continuous extension to $\bar{F}_n \setminus (\hat{\mathcal{Q}}\cup \mathsf{Z})$.

\item $\det M_n(x,t,k) = 1$ for $x \geq 0$, $t\in [0,T]$ and $k \in \bar{F}_n \setminus (\hat{\mathcal{Q}}\cup \mathsf{Z})$.

\item For each $x \geq 0$ and $t\in [0,T]$, the function $M(x,t,k)$ defined by $M(x,t,k) = M_n(x,t,k)$ for $k \in F_n$ obeys the symmetries
\begin{align}\label{Msymm}
& M(x,t, k) = \mathcal{A} M(x,t,\omega k)\mathcal{A}^{-1} = \mathcal{B} M(x,t,k^{-1})\mathcal{B}, \qquad k \in \C \setminus (\hat{\mathcal{Q}}\cup \mathsf{Z}).
\end{align}
\end{enumerate}
\end{proposition}
Because $\gamma_{ii}^n = \gamma_{3}$ for $i = 1,2,3$ and $n = 1, \dots,36$, the large $k$ expansion of $M_{n}(x,t,k)$ is given by the right-hand side of \eqref{mujlargekexpansion} with $j=3$. This fact can be proved in the same way as in \cite[Lemma 4.3]{CLgoodboussinesq}, so we omit the proof. 
\begin{lemma}[Asymptotics of $M$ as $k \to \infty$]\label{Matinftylemma}
For any integer $p \geq 1$, there exists $R > 0$ and $C>0$ such that
\begin{align*}
& \bigg|M(x,t,k) - \bigg(I + \frac{\mu_{3}^{(1)}(x,t)}{k} + \cdots + \frac{\mu_{3}^{(p)}(x,t)}{k^{p}}\bigg) \bigg| \leq
\frac{C}{|k|^{p+1}}, \quad x \geq 0, \, t\in [0,T], \   |k| \geq R
\end{align*}
with $k \in \C \setminus \Gamma$. The coefficients $\mu_{3}^{(1)}$ and $\mu_{3}^{(2)}$ were computed explicitly in Proposition \ref{prop:first two coeff at infty}.
\end{lemma}
The next lemma expresses $M_{n}$ in terms of $\mu_{1},\mu_{2}$ and $\mu_{3}$, in the case where $u_0, v_0 \in \mathcal{S}(\R_{+})$ have compact support. 
\begin{lemma}[Relation between $M_n$ and $\mu_{j}$]\label{Snexplicitlemma}
Suppose $u_0, v_0 \in \mathcal{S}(\R_{+})$ have compact support. Then for $x \geq 0$ and $t \in [0,T]$ we have
\begin{align}
M_n(x,t,k) & = \mu_{1}(x,t, k) e^{x\hat{\mathcal{L}}(k)+t\hat{\mathcal{Z}}(k)} R_n(k) \nonumber \\
& = \mu_{2}(x,t, k) e^{x\hat{\mathcal{L}}(k)+t\hat{\mathcal{Z}}(k)} S_n(k) \nonumber \\
& = \mu_{3}(x,t, k) e^{x\hat{\mathcal{L}}(k)+t\hat{\mathcal{Z}}(k)} T_n(k), \qquad  \ k \in \bar{F}_n\setminus (\hat{\mathcal{Q}}\cup \mathsf{Z}), \ n = 1, \dots, 36, \label{Mn mu1mu2mu3}
\end{align}
where $\{T_{n}\}_{n=7}^{12}$ are given by
\begin{align}
& T_{7}(k) = \begin{pmatrix}
1 & 0 & \frac{(S^{T}s^{A})_{31}}{(S^{T}s^{A})_{33}} \\
- \frac{s_{21}}{s_{22}} & 1 & \frac{(S^{T}s^{A})_{32}}{(S^{T}s^{A})_{33}} \\
0 & 0 & 1
\end{pmatrix}, \qquad T_{8}(k) = \begin{pmatrix}
1 & 0 & \frac{(S^{T}s^{A})_{31}}{(S^{T}s^{A})_{33}} \\
- \frac{(S^{-1}s)_{21}}{(S^{-1}s)_{22}} & 1 & \frac{(S^{T}s^{A})_{32}}{(S^{T}s^{A})_{33}} \\
0 & 0 & 1
\end{pmatrix}, \nonumber \\
&  T_{9}(k) = \begin{pmatrix}
1 & 0 & \frac{s^{A}_{31}S^{A}_{22}-s^{A}_{21}S^{A}_{32}}{s^{A}_{33}S^{A}_{22}-s^{A}_{23}S^{A}_{32}} \\[0.2cm]
- \frac{(S^{-1}s)_{21}}{(S^{-1}s)_{22}} & 1 & \frac{s^{A}_{32}S^{A}_{22}-s^{A}_{22}S^{A}_{32}}{s^{A}_{33}S^{A}_{22}-s^{A}_{23}S^{A}_{32}} \\[0.2cm]
0 & 0 & 1
\end{pmatrix},
\quad
  T_{10}(k) = \begin{pmatrix}
  1 & -\frac{s_{12}}{s_{11}} & \frac{s^{A}_{31}S^{A}_{22}-s^{A}_{21}S^{A}_{32}}{s^{A}_{33}S^{A}_{22}-s^{A}_{23}S^{A}_{32}} \\[0.2cm]
 0 & 1 & \frac{s^{A}_{32}S^{A}_{22}-s^{A}_{22}S^{A}_{32}}{s^{A}_{33}S^{A}_{22}-s^{A}_{23}S^{A}_{32}} \\[0.2cm]
 0 & 0 & 1
   \end{pmatrix}, \nonumber \\
& T_{11}(k) =  \begin{pmatrix}
  1 & -\frac{s_{12}}{s_{11}} & \frac{s^{A}_{31}}{s^{A}_{33}} \\[0.2cm]
 0 & 1 & \frac{s^{A}_{32}}{s^{A}_{33}} \\[0.2cm]
 0 & 0 & 1
   \end{pmatrix}, \quad T_{12}(k) = \begin{pmatrix}
1 & - \frac{(S^{-1}s)_{12}}{(S^{-1}s)_{11}} & \frac{s^{A}_{31}}{s^{A}_{33}} \\[0.2cm]
0 & 1 & \frac{s^{A}_{32}}{s^{A}_{33}} \\[0.2cm]
0 & 0 & 1
\end{pmatrix}, \label{SnTnexplicit}
\end{align}
and $\{T_{n}\}_{n=1}^{6}, \{T_{n}\}_{n=13}^{36}$ are defined such that the sectionally analytic function $T(x,t,k)$ defined by $T(x,t,k) = T_n(x,t,k)$ for $k \in F_n \setminus (\hat{\mathcal{Q}}\cup \mathsf{Z})$ satisfies the symmetries
\begin{align}\label{T symmetries}
& T(x,t, k) = \mathcal{A} T(x,t,\omega k)\mathcal{A}^{-1} = \mathcal{B} T(x,t,k^{-1})\mathcal{B}, \qquad k \in \C \setminus (\hat{\mathcal{Q}}\cup \mathsf{Z}).
\end{align}
The functions $\{S_{n},T_{n}\}_{n=1}^{36}$ appearing  in \eqref{Mn mu1mu2mu3} are given by
\begin{align*}
S_n(k)=s(k)T_n(k), \qquad R_n(k)=S(k)^{-1}s(k)T_{n}(k)=S(k)^{-1}S_{n}(k),
\end{align*}
where $s$ and $S$ are given by \eqref{def of S and s}.
\end{lemma}
\proofbegin
Since $u_0, v_0$ have compact support, Propositions \ref{XYprop} and \ref{sprop} imply that 
$\{\mu_{j}(x,t,k)\}_{1}^{3}$, $s(k),S(k)$, are well-defined for all $k \in \C \setminus \hat{\mathcal{Q}}$.
Moreover, since each of the matrix valued functions $\{\mu_{j}\}_{1}^{3},M_{n}$ solves the Lax pair \eqref{Xlax}, they must satisfy
\begin{align*}
M_n(x,t,k) = \mu_{1}(x,t,k) e^{x\hat{\mathcal{L}}+t\hat{\mathcal{Z}}} R_n(k) = \mu_{2}(x,t,k) e^{x\hat{\mathcal{L}}+t\hat{\mathcal{Z}}} S_n(k) = \mu_{3}(x,t,k) e^{x\hat{\mathcal{L}}+t\hat{\mathcal{Z}}} T_n(k),
\end{align*}   
for some matrices $R_{n}(k)$, $S_n(k)$ and $T_n(k)$, $n = 1, \dots, 36$. Since $\mu_{j}(x_{j},t_{j},k)=I$, $j=1,2,3$, we have
\begin{align}\label{SnTndef}
\begin{cases}
R_n(k) = e^{-T\hat{\mathcal{Z}}(k)}M_n(0,T,k), \\
S_n(k) = M_n(0,0,k), \\
T_n(k) =  \displaystyle{\lim_{x \to \infty}} e^{-x\hat{\mathcal{L}}(k)}M_n(x,0,k), 
\end{cases}
\quad k \in \bar{F}_n\setminus (\hat{\mathcal{Q}}\cup \mathsf{Z}).
\end{align}
Let $K > 0$ be large enough so that the supports of $u_0, v_0$ are in $[0,K] \subset \R_{+}$. Since $\mathsf{U}(x,0,k) = 0$ for $|x| > K$, it follows from \eqref{Mndef} that $e^{-x\hat{\mathcal{L}}}M_n(x,0,k)$ is independent of $x$  for $|x| > K$; hence, $T_{n}$ in \eqref{SnTndef} is well-defined.

Comparing with (\ref{mu3mu2mu1sS}), we have
\begin{equation}\label{sSSnrelations}  
s(k) = S_n(k)T_n(k)^{-1}, \quad S(k)=S_{n}(k)R_{n}(k)^{-1}, \qquad k \in \bar{F}_n\setminus (\hat{\mathcal{Q}}\cup \mathsf{Z}).
\end{equation}
Moreover, the integral equations (\ref{Mndef}) imply that
\begin{subequations}\label{sSSnrelations 2} 
\begin{align}
& \left(R_n(k)\right)_{ij} = 0 \quad \text{if} \quad \gamma_{ij}^n = \gamma_{1}, \\
& \left(S_n(k)\right)_{ij} = 0 \quad \text{if} \quad \gamma_{ij}^n = \gamma_{2}, \\
& \left(T_n(k)\right)_{ij} = \delta_{ij} \quad \text{if} \quad \gamma_{ij}^n = \gamma_{3}, 
\end{align}
\end{subequations}
For each $n\in \{1,\ldots,36\}$, equations \eqref{sSSnrelations}--\eqref{sSSnrelations 2} yields $18$ scalar equations for the $18$ unknowns entries of $R_{n}(k),S_n(k)$ and $T_n(k)$. 
For $n = 7, \ldots, 12$, solving the algebraic system explicitly yields the expressions in \eqref{SnTnexplicit}. The fact that $T(x,t,k)$ satisfies \eqref{T symmetries} is a direct consequence of \eqref{SnTndef} and \eqref{Msymm}.
\proofend

Let $\eta \in C_c^\infty(\R_{+})$ be a decreasing function such that $\eta(x) = 1$ for $x \leq 1$ and $\eta(x) = 0$ for $x \geq 2$. For each $\ell \geq 1$, define $\eta_{\ell}(x) = \eta(x/\ell)$. Then, for any $f \in \mathcal{S}(\R_{+})$, the sequence $(\eta_{\ell}f)_{\ell\geq 1} \subset C_c^\infty(\R_{+})$ converges to $f$ in $\mathcal{S}(\R_{+})$ as $\ell \to \infty$. The following lemma can be proved using the same argument as in \cite[Lemma 4.5]{CLgoodboussinesq}. 

\begin{lemma}\label{sequencelemma}
Let $\{\mu_{j}(x,t,k),$ $\mu_{j}^{A}(x,t,k), s(k), s^{A}(k),S(k), S^{A}(k), M(x,t,k)\}$ and $\{\mu_{j}^{(\ell)}(x,t,k)$, $\mu_{j}^{A(\ell)}(x,t,k), s^{(\ell)}(k), s^{A(\ell)}(k),$ $S^{(\ell)}(k), S^{A(\ell)}(k),$ $M^{(\ell)}(x,t,k)\}$ be the spectral functions associated with $(u_0, v_0, \tilde{u}_{0}, \tilde{u}_{1}, \tilde{u}_{2}, \tilde{v}_0)$ and 
\begin{align}\label{uvsequence}
(u_0^{(\ell)}, v_0^{(\ell)}, \tilde{u}_{0}, \tilde{u}_{1}, \tilde{u}_{2}, \tilde{v}_0) := (\eta_{\ell} u_0, \eta_{\ell} v_0, \tilde{u}_{0}, \tilde{u}_{1}, \tilde{u}_{2}, \tilde{v}_0) \in (C_c^\infty(\R_{+}))^{2} \times (C^{\infty}([0,T]))^{4}, 
\end{align} 
respectively. Then 
\begin{align*}
& \lim_{\ell\to\infty} s^{(\ell)}(k) = s(k), \qquad k \in (\omega^2 \hat{\mathcal{S}}_{3}, \omega \hat{\mathcal{S}}_{3}, \hat{\mathcal{S}}_{3}) \setminus \hat{\mathcal{Q}},  \\
& \lim_{\ell\to\infty} s^{A(\ell)}(k) = s^A(k), \qquad k \in	(\omega^2 \hat{\mathcal{S}}_{3}^{*}, \omega \hat{\mathcal{S}}_{3}^{*}, \hat{\mathcal{S}}_{3}^{*}) \setminus \hat{\mathcal{Q}}, 	\\ 
& \lim_{\ell\to\infty} S^{(\ell)}(k) = S(k), \; \lim_{\ell\to\infty} S^{A(\ell)}(k) = S^{A}(k), \qquad k \in \C \setminus \hat{\mathcal{Q}},  \\
& \lim_{\ell\to \infty} \mu_{j}^{(\ell)}(x,t,k) = \mu_{j}(x,t,k), \qquad x \geq 0, \; t \in [0,T], \ k \in \C \setminus \hat{\mathcal{Q}}, j=1,2, \\ 
& \lim_{\ell\to \infty} \mu_{3}^{(\ell)}(x,t,k) = \mu_{3}(x,t,k), \qquad x \geq 0, \; t \in [0,T], \ k \in (\omega^2 \hat{\mathcal{S}}_{3}, \omega \hat{\mathcal{S}}_{3}, \hat{\mathcal{S}}_{3}) \setminus \hat{\mathcal{Q}},  \\ 
& \lim_{\ell\to \infty} \mu_{j}^{A(\ell)}(x,t,k) = \mu_{j}^{A}(x,t,k), \qquad x \geq 0, \; t \in [0,T], \ k \in \C \setminus \hat{\mathcal{Q}}, j=1,2,  \\ 
& \lim_{\ell\to \infty} \mu_{3}^{A(\ell)}(x,t,k) = \mu_{3}^{A}(x,t,k), \qquad x \geq 0, \; t \in [0,T], \ k \in (\omega^2 \hat{\mathcal{S}}_{3}^{*}, \omega \hat{\mathcal{S}}_{3}^{*}, \hat{\mathcal{S}}_{3}^{*}) \setminus \hat{\mathcal{Q}},  \\ 
& \lim_{\ell\to \infty} M_n^{(\ell)}(x,t,k) = M_n(x,t,k), \qquad x \geq 0, \; t \in [0,T], \ k \in \bar{F}_n\setminus (\hat{\mathcal{Q}}\cup\mathsf{Z}), \ n = 1, \dots, 36. 
\end{align*}
\end{lemma}

\begin{lemma}[Jump condition for $M$]\label{Mjumplemma}
For each $x \geq 0$ and $t\in [0,T]$, $M(x,t,k)$ satisfies the jump condition
\begin{align}\label{jump M first}
  M_+(x,t,k) = M_-(x,t, k) v(x, t, k), \qquad k \in \Gamma \setminus (\Gamma_{\star}\cup \hat{\mathcal{Q}}\cup\mathsf{Z}),
\end{align}
where $v$ is the jump matrix defined in (\ref{vdef}).
\end{lemma}
\begin{proof}
For each $k\in \overline{F_{n}}\setminus (\hat{\mathcal{Q}}\cup\mathsf{Z})$, $M_n(x,t,k)$ is a smooth function of $x \geq 0$ and $t\in [0,T]$ which satisfies (\ref{Xlax}). 
Hence, for $k\in \Gamma_{1}=\overline{F}_{1}\cap \overline{F}_{18}\setminus\{i\}$, $k\notin \hat{\mathcal{Q}}\cup\mathsf{Z}$, 
\begin{align}\label{M1M6J1}
M_1(x,t,k) = M_{18}(x,t,k) e^{x\widehat{\mathcal{L}(k)}+t\widehat{\mathcal{Z}(k)}}J_1(k), \qquad k \in \Gamma_{1}\setminus (\hat{\mathcal{Q}}\cup\mathsf{Z}),
\end{align}
for some matrix $J_1(k)$ independent of $x$ and $t$.
If $u_0, v_0$ have support in some compact subset $[0,K] \subset [0,+\infty)$,  $K > 0$, then $\mu_{3}(x,t,k)=I$ for all $x>K$, and thus, by \eqref{Mn mu1mu2mu3}, $M_n(x,t,k) = e^{x\hat{\mathcal{L}}(k)+t\hat{\mathcal{Z}}(k)}T_n(k)$ for all $x>K$. Therefore, replacing $x$ by $K+1$ in \eqref{M1M6J1} yields
$$J_1(k) = T_{18}(k)^{-1}T_1(k),$$
and substituting the expressions from Lemma \ref{Snexplicitlemma} for $T_1,T_{18}$, we find 
\begin{align*}
J_1(k) & = \begin{pmatrix}
1 & r_{2}(\frac{1}{k}) & 0 \\
0 & 1 & 0 \\
-\tilde{r}_{1}(\omega^{2}k) & r_{2}(\omega k) & 1
\end{pmatrix}^{-1} \begin{pmatrix}
1 & 0 & 0 \\
\tilde{r}_{2}(k) & 1 & 0 \\
\tilde{r}_{2}(\frac{1}{\omega^{2}k}) & -r_{1}(\frac{1}{\omega k}) & 1
\end{pmatrix} \\
& = \begin{pmatrix}
1-r_{2}(\frac{1}{k})\tilde{r}_{2}(k) & -r_{2}(\frac{1}{k}) & 0 \\
\tilde{r}_{2}(k) & 1 & 0 \\
\tilde{g}(k) & -h(k) & 1
\end{pmatrix} = \tilde{v}_{1''}(k),
\end{align*}
where for the second equality we have used \eqref{relations on the unit circle}, $\tilde{g},h$ are defined in \eqref{def of h} and $\tilde{v}_{1''}$ is as in \eqref{vdef}. This finishes the proof of 
\begin{align}\label{lol7}
M_+(x,t,k) = M_-(x,t, k) v(x, t, k), \qquad k \in \Gamma_{1} \setminus (\hat{\mathcal{Q}}\cup\mathsf{Z}),
\end{align}
in the case where $u_0,v_0$ are compactly supported.

If $u_0,v_0$ are not compactly supported, then we find that \eqref{lol7} still holds by applying Lemma \ref{sequencelemma} to the sequence $(u_0^{(\ell)}, v_0^{(\ell)})$ defined in \eqref{uvsequence}.

This finishes the proof of \eqref{jump M first} for $k\in \Gamma_{1}$. The proof of \eqref{jump M first} on the other parts of $\Gamma$ is similar.
\end{proof}

Define the complex-valued functions $\theta_{ij}(x,t, k)$ for $1 \leq i,j \leq 3$ by 
\begin{align}\label{def of Phi ij}
\theta_{ij}(x,t,k) = (l_{i}(k)-l_{j}(k))x + (z_{i}(k)-z_{j}(k))t.
\end{align}
The next lemma expresses $M_{n}$ in terms of $\mu_{1},\mu_{2},\mu_{3},s,S,s^{A},S^{A}$.

\begin{lemma}\label{M1XYlemma}
The functions $M_{13},M_{14},M_{15},M_{22},M_{23},M_{24}$, which are defined on $D_{13},D_{14}$, $D_{15},E_{4},E_{5},E_{6}$, respectively, can be expressed in terms of the entries of $\mu_{1},\mu_{2},\mu_{3},s,S,s^{A}$ and $S^{A}$ as follows:
\begin{align*}
& M_{13} = \begin{pmatrix} 
(\mu_{3})_{11} & \frac{(\mu_{1})_{12}}{(S^{T}s^{A})_{22}} & (\mu_{2})_{12} \frac{s^{A}_{32}e^{\theta_{23}}}{-s_{11}} + (\mu_{2})_{13} \frac{s^{A}_{22}}{s_{11}}  \\
(\mu_{3})_{21} & \frac{(\mu_{1})_{22}}{(S^{T}s^{A})_{22}} & (\mu_{2})_{22} \frac{s^{A}_{32}e^{\theta_{23}}}{-s_{11}} + (\mu_{2})_{23} \frac{s^{A}_{22}}{s_{11}}  \\
(\mu_{3})_{31} & \frac{(\mu_{1})_{32}}{(S^{T}s^{A})_{22}} & (\mu_{2})_{32} \frac{s^{A}_{32}e^{\theta_{23}}}{-s_{11}} + (\mu_{2})_{33} \frac{s^{A}_{22}}{s_{11}}
\end{pmatrix}, \\
& M_{14} = \begin{pmatrix} 
(\mu_{3})_{11} & \frac{(\mu_{1})_{12}}{(S^{T}s^{A})_{22}} & *  \\
(\mu_{3})_{21} & \frac{(\mu_{1})_{22}}{(S^{T}s^{A})_{22}} & *  \\
(\mu_{3})_{31} & \frac{(\mu_{1})_{32}}{(S^{T}s^{A})_{22}} & *
\end{pmatrix}, \\
& [M_{14}]_{3} = \begin{pmatrix}
(\mu_{2})_{11} \frac{(s^{A}_{32}S^{A}_{21}-s^{A}_{22}S^{A}_{31})e^{\theta_{13}}}{(S^{-1}s)_{11}} + (\mu_{2})_{12} \frac{(s^{A}_{12}S^{A}_{31}-s^{A}_{32}S^{A}_{11})e^{\theta_{23}}}{(S^{-1}s)_{11}} + (\mu_{2})_{13} \frac{s^{A}_{22}S^{A}_{11}-s^{A}_{12}S^{A}_{21}}{(S^{-1}s)_{11}} \\
(\mu_{2})_{21} \frac{(s^{A}_{32}S^{A}_{21}-s^{A}_{22}S^{A}_{31})e^{\theta_{13}}}{(S^{-1}s)_{11}} + (\mu_{2})_{22} \frac{(s^{A}_{12}S^{A}_{31}-s^{A}_{32}S^{A}_{11})e^{\theta_{23}}}{(S^{-1}s)_{11}} + (\mu_{2})_{23} \frac{s^{A}_{22}S^{A}_{11}-s^{A}_{12}S^{A}_{21}}{(S^{-1}s)_{11}} \\
(\mu_{2})_{31} \frac{(s^{A}_{32}S^{A}_{21}-s^{A}_{22}S^{A}_{31})e^{\theta_{13}}}{(S^{-1}s)_{11}} + (\mu_{2})_{32} \frac{(s^{A}_{12}S^{A}_{31}-s^{A}_{32}S^{A}_{11})e^{\theta_{23}}}{(S^{-1}s)_{11}} + (\mu_{2})_{33} \frac{s^{A}_{22}S^{A}_{11}-s^{A}_{12}S^{A}_{21}}{(S^{-1}s)_{11}}
\end{pmatrix}, \\
& M_{15} = \hspace{-0.1cm} \begin{pmatrix} 
(\mu_{3})_{11} & (\mu_{2})_{11}\frac{S^{A}_{21}e^{\theta_{12}}}{s^{A}_{12}S^{A}_{21}-s^{A}_{22}S^{A}_{11}}-(\mu_{2})_{12}\frac{S^{A}_{11}}{s^{A}_{12}S^{A}_{21}-s^{A}_{22}S^{A}_{11}} & *  \\
(\mu_{3})_{21} & (\mu_{2})_{21}\frac{S^{A}_{21}e^{\theta_{12}}}{s^{A}_{12}S^{A}_{21}-s^{A}_{22}S^{A}_{11}}-(\mu_{2})_{22}\frac{S^{A}_{11}}{s^{A}_{12}S^{A}_{21}-s^{A}_{22}S^{A}_{11}} & *  \\
(\mu_{3})_{31} & (\mu_{2})_{31}\frac{S^{A}_{21}e^{\theta_{12}}}{s^{A}_{12}S^{A}_{21}-s^{A}_{22}S^{A}_{11}}-(\mu_{2})_{32}\frac{S^{A}_{11}}{s^{A}_{12}S^{A}_{21}-s^{A}_{22}S^{A}_{11}} & *
\end{pmatrix}\hspace{-0.1cm}, \; M_{22} = \hspace{-0.1cm} \begin{pmatrix} 
(\mu_{3})_{11} & \frac{(\mu_{2})_{12}}{s_{22}^A} & *  \\
(\mu_{3})_{21} & \frac{(\mu_{2})_{22}}{s_{22}^A} & *  \\
(\mu_{3})_{31} & \frac{(\mu_{2})_{32}}{s_{22}^A} & *
\end{pmatrix}\hspace{-0.1cm}, \\
& M_{23} = \begin{pmatrix} 
(\mu_{3})_{11} & \frac{(\mu_{2})_{12}}{s_{22}^A} & *  \\
(\mu_{3})_{21} & \frac{(\mu_{2})_{22}}{s_{22}^A} & *  \\
(\mu_{3})_{31} & \frac{(\mu_{2})_{32}}{s_{22}^A} & *
\end{pmatrix}\hspace{-0.1cm}, \; M_{24} = \hspace{-0.1cm} \begin{pmatrix} 
(\mu_{3})_{11} & (\mu_{2})_{12}\frac{S^{A}_{33}}{s_{22}^A S^{A}_{33}-s^{A}_{32}S^{A}_{23}}-(\mu_{2})_{13}\frac{S^{A}_{23}e^{\theta_{32}}}{s_{22}^A S^{A}_{33}-s^{A}_{32}S^{A}_{23}} & *  \\
(\mu_{3})_{21} & (\mu_{2})_{22}\frac{S^{A}_{33}}{s_{22}^A S^{A}_{33}-s^{A}_{32}S^{A}_{23}}-(\mu_{2})_{23}\frac{S^{A}_{23}e^{\theta_{32}}}{s_{22}^A S^{A}_{33}-s^{A}_{32}S^{A}_{23}} & *  \\
(\mu_{3})_{31} & (\mu_{2})_{32}\frac{S^{A}_{33}}{s_{22}^A S^{A}_{33}-s^{A}_{32}S^{A}_{23}}-(\mu_{2})_{33}\frac{S^{A}_{23}e^{\theta_{32}}}{s_{22}^A S^{A}_{33}-s^{A}_{32}S^{A}_{23}} & * 
\end{pmatrix}\hspace{-0.1cm}, \\
& [M_{13}]_{3} = [M_{23}]_{3} = [M_{24}]_{3}, \quad [M_{14}]_{3} = [M_{15}]_{3} = [M_{22}]_{3},
\end{align*}
for all $x \in \R_{+}$, $t\in [0,T]$ and $k \in F_{n} \setminus (\hat{\mathcal{Q}}\cup\mathsf{Z})$, with $n=13,14,15,22,23,24,$ respectively.
\end{lemma}
\begin{proof}
Let $(u_0^{(\ell)}, v_0^{(\ell)})$ be as in (\ref{uvsequence}). By Lemma \ref{Snexplicitlemma} we have
\begin{align*}
M_n^{(\ell)}(x,t,k) & = \mu_{1}^{(\ell)}(x,t, k) e^{x\hat{\mathcal{L}}(k)+t\hat{\mathcal{Z}}(k)} R_n^{(\ell)}(k) = \mu_{2}^{(\ell)}(x,t, k) e^{x\hat{\mathcal{L}}(k)+t\hat{\mathcal{Z}}(k)} S_n^{(i)}(k) \nonumber \\
& = \mu_{3}^{(\ell)}(x,t, k) e^{x\hat{\mathcal{L}}(k)+t\hat{\mathcal{Z}}(k)} T_n^{(\ell)}(k), \qquad  \ k \in \bar{F}_n\setminus (\hat{\mathcal{Q}}\cup \mathsf{Z}), \ n = 1, \dots, 36.
\end{align*}
In particular, the columns of $M_{14}^{(\ell)}$ are given by
\begin{align*}
\begin{cases}
 [M_{14}^{(\ell)}(x,t,k)]_1 = [\mu_{3}^{(\ell)}(x,t,k)]_1,
	\\
 [M_{14}^{(\ell)}(x,t,k)]_2 = \frac{[\mu_{1}^{(\ell)}(x,t,k)]_2}{(S^{(\ell)T}s^{A(\ell)})_{22}(k)}, \\
 [M_{14}^{(\ell)}(x,t,k)]_3 = [\mu_{2}^{(\ell)}(x,t,k)]_{1} \frac{(s^{(\ell)A}_{32}S^{(\ell)A}_{21}-s^{(\ell)A}_{22}S^{(\ell)A}_{31})(k)e^{\theta_{13}(x,t,k)}}{((S^{(\ell)})^{-1}s^{(\ell)})_{11}(k)} \\[0.15cm]
 \qquad + [\mu_{2}^{(\ell)}(x,t,k)]_{2} \frac{(s^{(\ell)A}_{12}S^{(\ell)A}_{31}-s^{(\ell)A}_{32}S^{(\ell)A}_{11})(k)e^{\theta_{23}(x,t,k)}}{((S^{(\ell)})^{-1}s^{(\ell)})_{11}(k)} \\[0.15cm]
 \qquad + [\mu_{2}^{(\ell)}(x,t,k)]_{3} \frac{(s^{(\ell)A}_{22}S^{(\ell)A}_{11}-s^{(\ell)A}_{12}S^{(\ell)A}_{21})(k)e^{\theta_{23}(x,t,k)}}{((S^{(\ell)})^{-1}s^{(\ell)})_{11}(k)},
\end{cases} 
\end{align*}
for $x \in \R_{+}$, $t\in[0,T]$, $k \in \bar{F}_{14}\setminus (\hat{\mathcal{Q}}\cup\mathsf{Z})$, and $\ell \geq 1$. 
Using Lemma \ref{sequencelemma}  to take $\ell \to \infty$, we find
\begin{align*}
\begin{cases}
 [M_{14}(x,t,k)]_1 = [\mu_{3}(x,t,k)]_1,
	\\
 [M_{14}(x,t,k)]_2 = \frac{[\mu_{1}(x,t,k)]_2}{(S^{T}s^{A})_{22}(k)}, \\
 [M_{14}(x,t,k)]_3 = [\mu_{2}(x,t,k)]_{1} \frac{(s^{A}_{32}S^{A}_{21}-s^{A}_{22}S^{A}_{31})(k)e^{\theta_{13}(x,t,k)}}{(S^{-1}s)_{11}(k)} \\[0.15cm]
 \qquad + [\mu_{2}(x,t,k)]_{2} \frac{(s^{A}_{12}S^{A}_{31}-s^{A}_{32}S^{A}_{11})(k)e^{\theta_{23}(x,t,k)}}{(S^{-1}s)_{11}(k)}  + [\mu_{2}(x,t,k)]_{3} \frac{(s^{A}_{22}S^{A}_{11}-s^{A}_{12}S^{A}_{21})(k)e^{\theta_{23}(x,t,k)}}{(S^{-1}s)_{11}(k)},
\end{cases} 
\end{align*}
for $x \in \R_{+}$, $t\in[0,T]$, $k \in \bar{F}_{14}\setminus (\hat{\mathcal{Q}}\cup\mathsf{Z})$. This proves the claim for $M_{14}$. The proofs for $M_{13},M_{15},M_{22},M_{23},M_{24}$ are analogous.
\end{proof}

\begin{lemma}\label{QtildeQlemma}
Suppose $u_0,v_0 \in \mathcal{S}(\R_{+})$ and $\tilde{u}_{0}, \tilde{u}_{1}, \tilde{u}_{2}, \tilde{v}_0 \in C^{\infty}([0,T])$ are such that Assumption \ref{solitonlessassumption} holds. Then the statements of Proposition \ref{Mnprop} and Lemmas \ref{Snexplicitlemma}-\ref{M1XYlemma} hold with $\mathsf{Z}$ replaced by the empty set.
\end{lemma}
\begin{proof}
Lemma \ref{M1XYlemma} implies that the first columns of $M_{n}$ have no singularities on $\bar{F}_{n}\setminus \hat{\mathcal{Q}}$, $n\in \{13,14,15,22,23,24\}$. Moreover, if the initial-boundary values satisfy Assumption \ref{solitonlessassumption}, we have in particular that 
\begin{align*}
s_{11},(S^{-1}s)_{11} \mbox{ are nonzero on } (\bar{D}_{13} \cup \bar{D}_{14} \cup \bar{D}_{15} \cup \bar{E}_{4} \cup \bar{E}_{5} \cup \bar{E}_{6} \cup \partial \D) \setminus \hat{\mathcal{Q}}.
\end{align*}
Using Lemma \ref{M1XYlemma}, this implies that the third columns of $M_{n}$ have no singularities on $\bar{F}_{n}\setminus \hat{\mathcal{Q}}$, $n\in \{13,14,15,22,23,24\}$. By Assumption \ref{solitonlessassumption}, we also have that
\begin{align*}
s^A_{11},(S^{T}s^{A})_{11} \mbox{ are nonzero on } (\bar{E}_{13} \cup \bar{E}_{14} \cup \bar{E}_{15} \cup \bar{D}_{4} \cup \bar{D}_{5} \cup \bar{D}_{6} \cup \partial \D) \setminus \hat{\mathcal{Q}}.
\end{align*}
The $\mathcal{B}$-symmetries in \eqref{symmetries of s} and \eqref{symmetries of sASA} then implies that
\begin{align*}
s^A_{22},(S^{T}s^{A})_{22} \mbox{ are nonzero on } (\bar{D}_{13} \cup \bar{D}_{14} \cup \bar{D}_{15} \cup \bar{E}_{4} \cup \bar{E}_{5} \cup \bar{E}_{6} \cup \partial \D) \setminus \hat{\mathcal{Q}},
\end{align*}
and therefore, by Lemma \ref{M1XYlemma}, the second columns of $M_{n}$ have no singularities on $\bar{F}_{n}\setminus \hat{\mathcal{Q}}$, $n\in \{13,14,22,23\}$. Finally, by Assumption \ref{solitonlessassumption} and Remark \ref{remark:sym combined with no soliton assumption},
\begin{align*}
s^A_{11}S^{A}_{33}-s^{A}_{31}S^{A}_{13} \mbox{ and } s^A_{11}S^{A}_{22}-s^{A}_{21}S^{A}_{12}  \mbox{ are nonzero on } \omega^{2}\hat{\mathcal{S}}_{3}^{*} \setminus\hat{\mathcal{Q}},
\end{align*}
which is equivalent, thanks to the $\mathcal{B}$-symmetries in \eqref{symmetries of sASA}, to
\begin{align*}
s^A_{22}S^{A}_{33}-s^{A}_{32}S^{A}_{23} \mbox{ and } s^A_{22}S^{A}_{11}-s^{A}_{12}S^{A}_{21}  \mbox{ are nonzero on } \omega\hat{\mathcal{S}}_{3}^{*} \setminus \hat{\mathcal{Q}}.
\end{align*}
This implies, again by Lemma \ref{M1XYlemma}, that the second columns of $M_{n}$ have no singularities on $\bar{F}_{n}\setminus \hat{\mathcal{Q}}$, $n\in \{15,24\}$. The $\mathcal{A}$- and $\mathcal{B}$-symmetries in (\ref{Msymm}) then imply that $M_n$ has no singularities in $\bar{F}_n \setminus \hat{\mathcal{Q}}$ for any $n\in \{1,\ldots,36\}$. 
\end{proof}



The next lemma establishes the behavior of $M$ as $k \to \pm 1$.

\begin{lemma}\label{Mat1lemma}
Suppose $u_0,v_0 \in \mathcal{S}(\R_{+})$ and $\tilde{u}_{0}, \tilde{u}_{1}, \tilde{u}_{2}, \tilde{v}_0 \in C^{\infty}([0,T])$ are such that Assumptions \ref{solitonlessassumption} and \ref{originassumption} hold.
Let $p \geq 1$ be an integer.
Then there are $3 \times 3$-matrix valued functions $\{\mathcal{M}_{14}^{(l)}(x,t),\widetilde{\mathcal{M}}_{5}^{(l)}(x,t)\}$, $l = -1,0, \dots, p$, with the following properties:
\begin{enumerate}[$(a)$]
\item The function $M$ satisfies, for $x \in \R_{+}$ and $t\in[0,T]$,
\begin{align*}
\begin{cases}
& \big|M(x,t,k) - \sum_{l=-1}^p \mathcal{M}_{14}^{(l)}(x,t)(k-1)^l\big| \leq C
|k-1|^{p+1}, \qquad |k-1| \leq \frac{1}{2}, \ k \in \bar{D}_{14}, \\
& \big|M(x,t,k) - \sum_{l=-1}^p \widetilde{\mathcal{M}}_5^{(l)}(x,t)(k+1)^l\big| \leq C
|k+1|^{p+1}, \qquad |k+1| \leq \frac{1}{2}, \ k \in \bar{E}_5.
\end{cases}
\end{align*}
\item For each $l \geq -1$, $\{\mathcal{M}_{14}^{(l)}(x,t),\widetilde{\mathcal{M}}_{5}^{(l)}(x,t)\}$ are smooth functions of $x \in \R_{+}$ and $t\in[0,T]$.
\item The first coefficients are of the form
\begin{align*}
& \mathcal{M}_{14}^{(-1)}(x,t) = \begin{pmatrix}
\alpha(x,t) & 0 & \beta(x,t) \\
-\alpha(x,t) & 0 & -\beta(x,t) \\
0 & 0 & 0
\end{pmatrix}, & & \mathcal{M}_{14}^{(0)}(x,t) = \begin{pmatrix}
\star & \gamma(x,t) & \star \\
\star & -\gamma(x,t) & \star \\
\star & 0 & \star
\end{pmatrix}, \\
& \widetilde{\mathcal{M}}_{5}^{(-1)}(x,t) = \begin{pmatrix}
\tilde{\alpha}(x,t) & 0 & \tilde{\beta}(x,t) \\
-\tilde{\alpha}(x,t) & 0 & -\tilde{\beta}(x,t) \\
0 & 0 & 0
\end{pmatrix}, & & \widetilde{\mathcal{M}}_{5}^{(0)}(x,t) = \begin{pmatrix}
\star & \tilde{\gamma}(x,t) & \star \\
\star & -\tilde{\gamma}(x,t) & \star \\
\star & 0 & \star
\end{pmatrix},
\end{align*}
for some functions $\alpha,\beta,\tilde{\alpha},\tilde{\beta}$.
\end{enumerate}
\end{lemma}
\begin{proof}
Lemma \ref{M1XYlemma} provides expressions for $M_{n}$, with $n\in \{14,23\}$ and $k \in \bar{F}_n\setminus \hat{\mathcal{Q}}$, in terms of $\mu_{1}, \mu_{2}, \mu_{3}, s, s^A, S, S^{A}$. Also, by Assumption \ref{originassumption}, $\lim_{k\to -1,  k \in \overline{E}_{5}} (k+1) s_{11}(k) \neq 0$ and
\begin{align*}
& \lim_{\substack{ k\to 1 \\ k \in \overline{D}_{14}}} (k-1)(S^{T}s^{A})_{22}(k)\neq 0, && \lim_{\substack{ k\to 1 \\ k \in \overline{D}_{14}}} (k-1) (S^{-1}s)_{11}(k) \neq 0, && \lim_{\substack{ k\to -1 \\ k \in \overline{E}_{5}}} (k+1)s^{A}_{22}(k)\neq 0. 
\end{align*} 
Hence, using also the expansions of $\mu_{1}, \mu_{2}, \mu_{3}, s, s^A, S, S^{A}$ as $k \to \pm 1$ given in Propositions \ref{XYat1prop}, \ref{XYat1prop adjoint}, \ref{sprop}, and \ref{sAprop}, as well as \eqref{global rel} and \eqref{global rel Adj}, the claim follows.
%
%
  
\end{proof}

\begin{remark}\upshape
Lemma \ref{Mat1lemma} gives the asymptotic behavior of $M(x,t,k)$ as $k \to 1$, $k \in \bar{D}_{14}$ and as $k \to -1$, $k \in \bar{E}_{5}$. The behavior of $M(x,t,k)$ near any of the critical points $\kappa_{j}$, $j=1,...,6$, from either the inside or the outside of the unit circle, can then be derived using $M(x,t,k) = \mathcal{A}M(x,t,\omega k) \mathcal{A}^{-1} = \mathcal{B} M(x,t,1/k)\mathcal{B}$.
\end{remark}

%
%
%

\begin{proposition}\label{RHth}
Suppose $\{u, v\}$ is a Schwartz class solution of (\ref{boussinesqsystem}) on $\R_{+} \times [0,T]$ with initial data $u_0, v_0 \in \mathcal{S}(\R_{+})$ and boundary values $\tilde{u}_{0}, \tilde{u}_{1}, \tilde{u}_{2}, \tilde{v}_0 \in C^{\infty}([0,T])$ such that Assumptions \ref{solitonlessassumption} and \ref{originassumption} hold. 
Then, for each $(x,t) \in \R_{+} \times [0,T]$,  $M(x,t,k)$ satisfies RH problem \ref{RH problem for M}  and
\begin{align}\label{recoveruv}
\begin{cases}
 \displaystyle{u(x,t) = -i\sqrt{3}\frac{\partial}{\partial x} M^{(1)}(x,t)_{33} = \frac{1-\omega}{2} M^{(2)}(x,t)_{32},}
	\vspace{.1cm}\\
 \displaystyle{v(x,t) = -i\sqrt{3}\frac{\partial}{\partial t}M^{(1)}(x,t)_{33}.}
\end{cases}
\end{align}
\end{proposition}
\begin{proof}
The fact that $M$ satisfies properties (a), (b) and (f) of RH problem \ref{RH problem for M} is a direct consequence of Proposition \ref{Mnprop} and Lemmas \ref{Mjumplemma} and \ref{QtildeQlemma}. Lemma \ref{Matinftylemma} implies that $M$ satisfies property (c) of RH problem \ref{RH problem for M}, whereas property (e) follows by combining \eqref{Msymm} with Lemma \ref{QtildeQlemma}. Moreover, by combining Lemma \ref{Matinftylemma} with the expressions of $\mu_{3}^{(1)}$ and $\mu_{3}^{(2)}$ of Proposition \ref{prop:first two coeff at infty}, we have
\begin{align*}
& M^{(1)}(x,t)_{33} = \frac{i}{\sqrt{3}}  \int_{+\infty}^{x} u(x^{\prime}, t) dx', \qquad  M^{(2)}(x,t)_{32} = \frac{2u(x,t)}{1-\omega}.
\end{align*}
Since $u,v$ have rapid decay as $x \to +\infty$ and satisfy $u_t = v_x$, this proves (\ref{recoveruv}).
Finally, property $(d)$ related to the behavior of $M$ as $k \to \pm 1$ follows from Lemma \ref{Mat1lemma}. 
\end{proof}

Theorem \ref{thm:inverse sca} is now a direct consequence of Proposition \ref{RHth} and Lemma \ref{lemma:equivalence}.

\appendix \setcounter{equation}{0}\renewcommand\theequation{A.\arabic{equation}}

\section{Uniqueness for RH problem \ref{RH problem for M}}\label{Appendix:uniqueness}

In this appendix, we prove the following result.

\begin{proposition}\label{uniquenessprop}
The solution $M$ of RH problem \ref{RH problem for M} is unique, if it exists.
\end{proposition}

Proposition \ref{uniquenessprop} is proved via a series of lemmas.

\begin{lemma}[Asymptotics of $M$ as $k \to \infty$]\label{lemma: asymp at inf appedix}
If $M$ satisfies RH problem \ref{RH problem for M}, then
\begin{align}\nonumber
M(x,t,k) = & \; I + \frac{M_{33}^{(1)}}{k} \begin{pmatrix} \omega^2 & 0 & 0 \\ 
0 & \omega & 0 \\ 
0 & 0 & 1
\end{pmatrix} 
+ \frac{M_{33}^{(2)}}{k^2}\begin{pmatrix} 
\omega  & 0 & 0 \\ 
0 & \omega^2 & 0 \\ 
0 & 0 &  1
\end{pmatrix}  
	\\ \label{singMatinfty}
&  + \frac{M_{12}^{(2)}}{k^2}\begin{pmatrix} 
0 & 1 & -\omega^{2} \\ 
-1 & 0 & \omega \\ 
\omega^{2} & -\omega &  0
\end{pmatrix} + O(k^{-3}),  \quad \mbox{as } k \to \infty,
\end{align}
where $M_{33}^{(1)}(x,t)$, $M_{33}^{(2)}(x,t)$, and $M_{12}^{(2)}(x,t)$ are complex-valued functions of $x$ and $t$.
\end{lemma}
\begin{proof}
Using the asymptotic formula \eqref{asymp for M at infty in RH def} together with the symmetries \eqref{symmetry of M}, we obtain the relations $M^{(1)} = \omega^{2}\mathcal{A}M^{(1)}\mathcal{A}^{-1}$ and $M^{(2)} = \omega\mathcal{A}M^{(2)}\mathcal{A}^{-1}$. Substituting the conditions \eqref{singRHMatinftyb} into these relations directly yields the claim.
\end{proof}

\begin{lemma}[Unit determinant of $M$]\label{unitdetlemma}
If $M$ is a solution of RH problem \ref{RH problem for M}, then $M$ has unit determinant. 
\end{lemma}
\begin{proof}
Since $\det v \equiv 1$, criteria (a) and (b) of RH problem \ref{RH problem for M}, combined with Morera's theorem, imply that $\det M$ is analytic in $\C \setminus (\Gamma_{\star} \cup \mathcal{Q})$. Criteria (f) then implies that $\det M$ is in fact analytic in $\C \setminus \mathcal{Q}$. Moreover, by criteria (c), $\det M(x,t,k) = 1+O(k^{-1})$ as $k \to \infty$. The formulas (\ref{singRHMat0}) on the behavior of $M(x,t,k)$ as $k \to k_{\star}\in \{1,-1\}$ implies (by a direct computation) that $\det M$ has no pole at $k = 1$ and $k=-1$. The $\mathcal{A}$-symmetry in \eqref{symmetry of M} then imply that $\det M$ has no pole at each of the points $\kappa_{j}$, $j=1,\ldots,6$. Thus $\det M$ is entire and the claim follows from Liouville's theorem.
\end{proof}

\begin{proof}[Proof of Proposition \ref{uniquenessprop}]
Suppose $M$ and $N$ are two solutions of RH problem \ref{RH problem for M}. By Lemma \ref{unitdetlemma}, $\det N \equiv 1$. In particular, $N^{-1}$ can be expressed in terms of the minors of $N$. Using this expression and (\ref{singRHMat0}) to expand $N^{-1}$ as $k \to 1$ in $\bar{D}_{14}$, we get 
\begin{align}\label{N1Aat0}
N^{-1}(x,t,k) = 
\frac{1}{k-1} 
\begin{pmatrix}
0 & 0 & 0  \\
a & a & * \\
0 & 0 & 0 
\end{pmatrix} 
+ \begin{pmatrix}
c  & c & * \\
*  & * & * \\
d  & d & * 
\end{pmatrix} 
+  O(k-1), \qquad \mbox{as } k\to 1, \; k \in \bar{D}_{14}.
\end{align}
Therefore, 
\begin{align}\label{asymp MNinv 1}
(MN^{-1})(x,t,k) = \frac{K_{1}}{k-1} + O(1), \qquad \mbox{as }k \to 1, \, k \in \bar{D}_{14}, 
\end{align}
where $K_{1}=K_{1}(x,t)$ is of the form
\begin{align*}
K_{1}=\begin{pmatrix}
p_{1} & p_{1} & p_{2} \\
-p_{1} & -p_{1} & -p_{2} \\
0 & 0 & 0
\end{pmatrix}
\end{align*}
for some functions $p_{1},p_{2}$ depending on $x$ and $t$ but not on $k$. Similarly, we find
\begin{align}\label{asymp MNinv 4}
(MN^{-1})(x,t,k) = \frac{K_{4}}{k+1} + O(1), \quad \mbox{as }k \to -1, \, k \in \bar{E}_{5},
\end{align}
where $K_{4}=K_{4}(x,t)$ is of the form
\begin{align*}
K_{4} = \begin{pmatrix}
\tilde{p}_{1} & \tilde{p}_{1} & \tilde{p}_{2} \\
-\tilde{p}_{1} & -\tilde{p}_{1} & -\tilde{p}_{2} \\
0 & 0 & 0
\end{pmatrix}
\end{align*}
for some functions $\tilde{p}_{1},\tilde{p}_{2}$ depending on $x$ and $t$ but not on $k$.

Since $M$ and $N$ have the same jumps and are bounded as $k\to k_{\star}\in \Gamma_{\star}$, $MN^{-1}$ is a meromorphic function whose only poles are at $\kappa_{j}$, $j=1,\ldots,6$. The $\mathcal{A}$-symmetry in \eqref{symmetry of M}, combined with \eqref{asymp MNinv 1}-\eqref{asymp MNinv 4}, then implies that
\begin{align*}
(MN^{-1})(x,t,k) = \frac{K_{j}(x,t)}{k-\kappa_{j}}+O(1), \qquad \mbox{as } k \to \kappa_{j}, \; j=1,\ldots,6,
\end{align*}
where
\begin{align*}
& K_{2} = \omega^{2} \mathcal{A}K_{4}\mathcal{A}^{-1} = \omega^{2} \begin{pmatrix}
0 & 0 & 0 \\
\tilde{p}_{2} & \tilde{p}_{1} & \tilde{p}_{1} \\
-\tilde{p}_{2} & -\tilde{p}_{1} & -\tilde{p}_{1}
\end{pmatrix}, & & K_{3} = \omega \mathcal{A}^{-1}K_{1}\mathcal{A} = \omega \begin{pmatrix}
-p_{1} & -p_{2} & -p_{1} \\
0 & 0 & 0 \\
p_{1} & p_{2} & p_{1}
\end{pmatrix}, \\
& K_{5} = \omega^{2} \mathcal{A}K_{1}\mathcal{A}^{-1} = \omega^{2} \begin{pmatrix}
0 & 0 & 0 \\
p_{2} & p_{1} & p_{1} \\
-p_{2} & -p_{1} & -p_{1}
\end{pmatrix}, & & K_{6} = \omega \mathcal{A}^{-1}K_{4}\mathcal{A} = \omega \begin{pmatrix}
-\tilde{p}_{1} & -\tilde{p}_{2} & -\tilde{p}_{1} \\
0 & 0 & 0 \\
\tilde{p}_{1} & \tilde{p}_{2} & \tilde{p}_{1}
\end{pmatrix}. 
\end{align*}
Since $M(x,t,k)= I+O(k^{-1})$ and $N(x,t,k) = I+O(k^{-1})$ as $k\to \infty$, it follows that 
\begin{align}\label{lol3}
(MN^{-1})(x,t,k) = I + \sum_{j=1}^{6} \frac{K_{j}(x,t)}{k-\kappa_{j}}, \qquad k \in \C\setminus \mathcal{Q}.
\end{align}
Expanding this expression as $k \to \infty$, we obtain
\begin{align}\label{lol1}
& (MN^{-1})(x,t,k) = I + \frac{1}{k}\begin{pmatrix}
(1-\omega)q_{1} & q_{1}-\omega q_{2} & q_{2} - \omega q_{1} \\
\omega^{2} q_{2} - q_{1} & (\omega^{2}-1)q_{1} & \omega^{2} q_{1} - q_{2} \\
\omega q_{1} - \omega^{2} q_{2} & \omega q_{2} - \omega^{2} q_{1} & (\omega - \omega^{2})q_{1}
\end{pmatrix} + O(k^{-2}),
\end{align}
where $q_{j} := p_{j}+\tilde{p}_{j}$, $j=1,2$. 
On the other hand, using Lemma \ref{lemma: asymp at inf appedix} to expand $MN^{-1}$ as $k \to \infty$, we find 
\begin{align}\nonumber
(MN^{-1})(x,t,k) = &\; I +\frac{M_{33}^{(1)}-N_{33}^{(1)}}{k}\begin{pmatrix}
\omega^{2} & 0 & 0 \\
0 & \omega & 0 \\
0 & 0 & 1
\end{pmatrix} + \frac{1}{k^{2}}\bigg\{ (M_{12}^{(2)}-N_{12}^{(2)}) \begin{pmatrix}
0 & 1 & -\omega^{2} \\
-1 & 0 & \omega \\
\omega^{2} & -\omega & 0
\end{pmatrix} \\
& + \big( M_{33}^{(2)}-N_{33}^{(2)} + N_{33}^{(1)}(N_{33}^{(1)}-M_{33}^{(1)}) \big) \begin{pmatrix}
\omega & 0 & 0 \\
0 & \omega^{2} & 0 \\
0 & 0 & 1
\end{pmatrix} \bigg\} + O(k^{-3}). \label{lol2}
\end{align}
Comparing the coefficients of $k^{-1}$ in \eqref{lol1} and \eqref{lol2}, we infer that $q_{1}=q_{2}=0$, i.e.
\begin{align*}
p_{1}+\tilde{p}_{1}= 0, \qquad p_{2}+\tilde{p}_{2}=0.
\end{align*}
Substituting these identities in \eqref{lol3}, we obtain
\begin{align}\label{lol5}
(MN^{-1})(x,t,k) = I + \frac{2}{k^{2}}\begin{pmatrix}
(1-\omega^{2})p_{1} & p_{1}-\omega^{2} p_{2} & p_{2}-\omega^{2} p_{1} \\
\omega p_{2}-p_{1} & (\omega-1)p_{1} & \omega p_{1}-p_{2} \\
\omega^{2}p_{1}-\omega p_{2} & \omega^{2} p_{2} - \omega p_{1} & (\omega^{2}-\omega)p_{1}
\end{pmatrix} + O(k^{-3})
\end{align}
as $k \to \infty$. Equating the coefficients of $k^{-2}$ in \eqref{lol2} and \eqref{lol5}, we get $p_{2} = 0$. Moreover, using again \eqref{lol3}, we also have
\begin{align*}
\det(MN^{-1}(x,t,k)) & = 1 - \frac{3\big( (p_{1}-p_{2})^{2}-(\tilde{p}_{1}-\tilde{p}_{2})^{2} \big)}{k^{3}} - \frac{3(p_{1}-p_{2}-\tilde{p}_{1}+\tilde{p}_{2})^{2}}{k^{6}}+O(k^{-7}) \\
& = 1 - \frac{12p_{1}^{2}}{k^{6}} + O(k^{-7}).
\end{align*}
By Lemma \ref{unitdetlemma}, $MN^{-1}$ has unit determinant, and therefore $p_{1} = 0$. Hence $K_{j}=0$ for all $j=1,\ldots,6$, and thus $MN^{-1}\equiv I$, i.e. $M = N$.
\end{proof}

\section{Proof of Lemma \ref{lemma:equivalence}}\label{Appendix:equivalence}
Suppose $\{u,v\}$ is a Schwartz class solution of \eqref{boussinesqsystem} with existence time $T\in (0,\infty)$, initial data $u_0, v_0 \in \mathcal{S}(\R_{+})$ and boundary values $\tilde{u}_{0},\tilde{u}_{1}, \tilde{u}_{2}, \tilde{v}_0 \in C^{\infty}([0,T])$.
Taking the derivative with respect to $x$ of the first equation in (\ref{boussinesqsystem}), and using $v_{tx}=u_{tt}$ (which follows from the second equation in (\ref{boussinesqsystem})), we infer that $u$ satisfies (\ref{badboussinesq}) with initial data $u(x,0)=u_0(x)$, $u_t(x,0) = v_{x}(x,0) = v_{0x}(x) =: u_1(x)$, and with boundary values 
\begin{align*}
& u(0,t) = \tilde{u}_{0}(t), \quad u_{x}(0,t) = \tilde{u}_{1}(t), \quad u_{xx}(0,t) = \tilde{u}_{2}(t), \\
& u_{xxx}(0,t) = (v_{t} - u_{x} - (u^2)_{x})(0,t) = (\tilde{v}_{0t}-\tilde{u}_{1}-2\tilde{u}_{0}\tilde{u}_{1})(t)=:\tilde{u}_{3}(t),
\end{align*}
where for the last equation we have used the first equation in (\ref{boussinesqsystem}). Clearly, $u_1 \in \mathcal{S}(\R_{+})$ and $\tilde{u}_{3} \in C^{\infty}([0,T])$, so this finishes the first part of the claim.

For the second part, suppose $u$ is Schwartz class solution of \eqref{badboussinesq} with existence time $T\in (0,\infty)$, initial data $u_0, u_1 \in \mathcal{S}(\R_{+})$ and boundary values $\tilde{u}_{0},\tilde{u}_{1}, \tilde{u}_{2}, \tilde{u}_{3} \in C^{\infty}([0,T])$. It follows from (\ref{badboussinesq}) that (\ref{rapiddecay u}) holds with $u$ replaced by $u_{tt}$. Since $u_t(x,t) = u_{1}(x) + \int_{0}^t u_{tt}(x,t') dt'$ and $u_{1}\in \mathcal{S}(\R_{+})$, we have that (\ref{rapiddecay u}) holds also with $u$ replaced by $u_{t}$.
Define $v$ by (\ref{vxtdef}). Since $u$ is smooth on $\R_{+}\times [0,T]$ and satisfies \eqref{rapiddecay u}, $v$ is also smooth on $\R_{+}\times [0,T]$. Straightforward estimates show then that (\ref{rapiddecay u}) holds also with $u$ replaced by $v$, and therefore \eqref{rapiddecay u v} holds. Integration of (\ref{badboussinesq}) from $+\infty$ to $x$ yields
\begin{align*}
\int_{+\infty}^x u_{tt}(x', t) dx' = u_{x}(x,t) + (u^{2})_{x}(x,t) + u_{xxx}(x,t).
\end{align*}
Since $\int_{+\infty}^x u_{tt}(x', t) dx' = v_t(x,t)$, 
we infer that $u,v$ satisfy (\ref{boussinesqsystem}).

Moreover, $v(x,0) = \int_{+\infty}^x u_{t}(x', 0) dx' = \int_{+\infty}^x u_{1}(x') dx' =: v_{0}(x)$ and
\begin{align*}
v(0,t) & = v(0,0) + \int_{0}^{t}v_{t}(0,t')dt' = \int_{+\infty}^0 u_t(x',0) dx' + \int_{0}^{t}(u_{x}+(u^{2})_{x}+u_{xxx})(0,t')dt' \\
& = \int_{+\infty}^0 u_1(x') dx' + \int_{0}^{t}(\tilde{u}_{1}+2\tilde{u}_{0}\tilde{u}_{1}+\tilde{u}_{3})(t')dt' =: \tilde{v}_{0}(t).
\end{align*}
Since $v\in \mathcal{S}(\R_{+})$ and $\tilde{v}_{0}\in C^{\infty}([0,T])$, this finishes the proof.

\bigskip
\noindent
{\bf Acknowledgement.} The author is a Research Associate of the Fonds de la Recherche Scientifique - FNRS. The author also acknowledges support from the Swedish Research Council, Grant No. 2021-04626, and from the European Research Council (ERC), Grant Agreement No. 101115687.


\end{document}